\documentclass[11pt,a4paper,reqno]{amsart}

\usepackage{amsmath,amsfonts,amssymb,amsthm}

\usepackage{bm,url}

\usepackage{verbatim}

\usepackage{tikz}
\usetikzlibrary{arrows}

\usepackage{color}

\usepackage[all]{xy}

\usepackage{hyperref}
\hypersetup{colorlinks,linkcolor=[rgb]{0,0,1},citecolor=[rgb]{1,0,0}}

\newlength{\wideitemsep}
\let\olditem\item
\renewcommand{\item}{\setlength{\itemsep}{\wideitemsep}\olditem}

\numberwithin{equation}{section}

\renewcommand{\geq}{\geqslant}
\renewcommand{\leq}{\leqslant}

\newcommand{\vs}{\vspace{0.5cm}}
\newcommand{\xra}{\xrightarrow}
\newcommand{\trm}{\textrm}
\newcommand{\mrm}{\mathrm}
\newcommand{\pt}{\mathrm{pt}}

\renewcommand{\tilde}{\widetilde}
\renewcommand{\bar}{\overline}

\newtheorem{theorem}{Theorem}[section]
\newtheorem*{theorem*}{Theorem}
\newtheorem{lemma}[theorem]{Lemma}
\newtheorem{proposition}[theorem]{Proposition}
\newtheorem{corollary}[theorem]{Corollary}
\newtheorem*{corollary*}{Corollary}

\theoremstyle{definition}
\newtheorem{definition}[theorem]{Definition}
\newtheorem{example}[theorem]{Example}

\theoremstyle{remark}
\newtheorem{remark}[theorem]{Remark}

\newenvironment{acknowledgements}{\textbf{Acknowledgements}:}{}

\newcommand{\cC}{\mathcal{C}}
\newcommand{\cD}{\mathcal{D}}
\newcommand{\cE}{\mathcal{E}}
\newcommand{\cF}{\mathcal{F}}
\newcommand{\cI}{\mathcal{I}}
\newcommand{\cO}{\mathcal{O}}
\newcommand{\cP}{\mathcal{P}}
\newcommand{\cM}{\mathcal{M}}

\newcommand{\cH}{\mathcal{H}}
\newcommand{\cW}{\mathcal{W}}
\newcommand{\cX}{\mathcal{X}}
\newcommand{\cZ}{\mathcal{Z}}

\newcommand{\bC}{\mathbb{C}}

\newcommand{\bP}{\mathbb{P}}
\newcommand{\bR}{\mathbb{R}}
\newcommand{\bZ}{\mathbb{Z}}

\newcommand{\bQ}{\mathbb{Q}}

\newcommand{\frh}{\mathfrak{h}}

\newcommand{\vv}{\bm{\mathrm v}}
\renewcommand{\ss}{\bm{\mathrm s}}
\newcommand{\ww}{\bm{\mathrm w}}
\renewcommand{\aa}{\bm{\mathrm a}}
\newcommand{\bb}{\bm{\mathrm b}}
\renewcommand{\tt}{\bm{\mathrm t}}
\newcommand{\uu}{\bm{\mathrm u}}

\DeclareMathOperator{\re}{Re}
\DeclareMathOperator{\im}{Im}

\DeclareMathOperator{\bigcone}{Big}

\DeclareMathOperator{\mov}{Mov}
\DeclareMathOperator{\nef}{Nef}
\DeclareMathOperator{\Pic}{Pic}
\DeclareMathOperator{\pos}{Pos}

\DeclareMathOperator{\amp}{Amp}
\DeclareMathOperator{\ns}{NS}

\DeclareMathOperator{\sym}{Sym}
\DeclareMathOperator{\ext}{ext}
\DeclareMathOperator{\Ext}{Ext}
\DeclareMathOperator{\gr}{Gr}
\DeclareMathOperator{\br}{Br}

\DeclareMathOperator{\spanning}{Span}
\DeclareMathOperator{\chern}{ch}

\DeclareMathOperator{\stab}{Stab^{\dagger}}
\DeclareMathOperator{\stabno}{Stab}

\newcommand{\bmm}{Bayer-Macr\`{i} }

\begin{document}

\title{Birational Geometry of Singular Moduli Spaces of O'Grady Type}
\author{Ciaran Meachan}
\address{School of Mathematics, University of Edinburgh, Scotland}
\email{ciaran.meachan@ed.ac.uk}
\author{Ziyu Zhang}
\address{Department of Mathematical Sciences, University of Bath, Bath BA2 7AY, United Kingdom}
\email{zz505@bath.ac.uk}


\begin{abstract}
Following Bayer and Macr\`{i}, we study the birational geometry of singular moduli spaces $M$ of sheaves on a K3 surface $X$ which admit symplectic resolutions. More precisely, we use the \bmm map from the space of Bridgeland stability conditions $\stabno(X)$ to the cone of movable divisors on $M$ to relate wall-crossing in $\stabno(X)$ to birational transformations of $M$. We give a complete classification of walls in $\stabno(X)$ and show that every minimal birational model of $M$ in the sense of the log minimal model program appears as a moduli space of Bridgeland semistable objects on $X$. An essential ingredient of our proof is an isometry between the orthogonal complement of a Mukai vector inside the algebraic Mukai lattice of $X$ and the N\'{e}ron-Severi lattice of $M$ which generalises results of Yoshioka, as well as Perego and Rapagnetta. Moreover, this allows us to conclude that the symplectic resolution of $M$ is deformation equivalent to the $10$-dimensional irreducible holomorphic symplectic manifold found by O'Grady.
\end{abstract}

\maketitle

\section{Introduction}

\subsection{Background}

Let $X$ be a complex projective K3 surface,  $\vv \in H^*_{\mathrm{alg}}(X,\bZ)$ a Mukai vector and $H$ an ample line bundle which is generic with respect to $\vv$ (in the sense of \cite[Section 4.C]{HuLe10a} and \cite[Section 1.4]{Yos01}). We can always write the Mukai vector as a multiple of a primitive class, say $\vv=m \vv_p$. If we further assume that $\vv^2>0$ with respect to the Mukai pairing, then there is a precise classification of moduli spaces $M_H(\vv)$ of Gieseker $H$-semistable sheaves on $X$ with Mukai vector $\vv$ which goes as follows:
\begin{itemize}
\item If $m=1$, then $\vv$ is a primitive Mukai vector and $M_H(\vv)$ is smooth. Moreover, in his seminal paper \cite{Muk84}, Mukai showed that $M_H(\vv)$ is an irreducible holomorphic symplectic manifold, parametrising $H$-stable sheaves.
\item If $m\geq 2$, then $M_H(\vv)$ has symplectic singularities (in the sense of \cite{Bea00}). Its smooth locus parametrises $H$-stable sheaves, while its singular locus parametrises S-equivalence classes of strictly semistable sheaves. This case splits into two radically different situations:
\begin{itemize}
\item If $m=2$ and $\vv_p^2=2$, then $M_H(\vv)$ has a symplectic resolution; see \cite{OG99a,LS06}.
\item If $m>2$ or $\vv_p^2>2$, then $M_H(\vv)$ does \emph{not} admit a symplectic resolution; see \cite{KaLeSo06a}.
\end{itemize}
\end{itemize}

The geometry of smooth moduli spaces $M_H(\vv)$, or irreducible holomorphic symplectic manifolds in general, has been the subject of intensive study for many years. In particular, there are many results about their birational geometry in the literature; see \cite{Huy03,HT09} for instance. On the other hand, Bridgeland \cite[Section 14]{Bri08} showed that these Gieseker moduli spaces can be realised as moduli spaces of $\sigma$-semistable objects in the `large volume limit' of (a certain connected component $\stab(X)$ of) his stability manifold $\stabno(X)$. More precisely, he proved that the stability manifold comes with a wall and chamber decomposition in the sense that the set of $\sigma$-semistable objects with some fixed numerical invariants is constant in each chamber and an object of the bounded derived category $\cD(X)$ of coherent sheaves on $X$ can only become stable or unstable by crossing a wall, that is, a real codimension one submanifold of $\stab(X)$. Furthermore, Bridgeland conjectured that crossing a wall should induce a birational transformation between the corresponding moduli spaces and this vision was recently crystallised in a revolutionary paper \cite{BM13} by Bayer and Macr\`i. 

Before going any further, we should say that these ideas concerning wall-crossing have been successfully applied to the study of the birational geometry of moduli spaces of sheaves on abelian surfaces; see \cite{Yos09,MYY11a,Mea12,Mac12,Yos12,AB13,MM13,MYY11b,YY14}. However, the focus of this paper will be on moduli spaces of sheaves on K3 surfaces as above. 

For any Mukai vector $\vv$ with $\vv^2>0$, on each moduli space of Bridgeland semistable objects $M_\sigma(\vv)$, Bayer and Macr\`i used the classical technique of determinant line bundles to construct an ample line bundle $\ell_\sigma$ on $M_\sigma(\vv)$. In particular, this gives rise to a `linearisation map' from any chamber $\cC$ in $\stab(X)$ to the nef cone of the moduli space $\nef (M_\cC(\vv))$ corresponding to that chamber. This map is an essential part of this paper; we call it the \emph{\bmm map}.

Moreover, when the Mukai vector $\vv$ is primitive, they obtained a complete picture relating wall crossing on $\stab(X)$ to the birational geometry of the corresponding moduli spaces. They studied how the moduli space changes when crossing a wall and classified all the walls in terms of the Mukai lattice. In particular, they verified the conjecture of Bridgeland, by showing that moduli spaces corresponding to two neighbouring chambers are indeed birational. For each type of wall, they could produce an explicit birational map, which identifies the two moduli spaces away from loci of codimension at least two and hence identifies the N\'{e}ron-Severi groups of them. Using these birational maps, they could glue the linearisation maps defined on each chamber together to get a global continuous \bmm map from $\stab(X)$ to the N\'{e}ron-Severi group of any generic moduli space $M_\sigma(\vv)$ and prove that, via wall crossings on $\stab(X)$, every birational minimal model of $M_\sigma(\vv)$ appears as a Bridgeland moduli space.

\subsection{Summary of main results}

This paper grew out of an attempt to first understand \cite{BM12,BM13} and then generalise their techniques to the simplest singular case in the hope of obtaining similar results. In particular, we are interested in the case when $\vv=2\vv_p$ and $\vv_p$ is a primitive Mukai vector with square equal to $2$ with respect to the Mukai pairing. We say that such a Mukai vector $\vv$ is of \emph{O'Grady type}. The benefit of considering this type of Mukai vector $\vv$ is that we will be able to show the existence of a symplectic resolution of the moduli space $M_\sigma(\vv)$ under any generic stability condition $\sigma$, which will allow us to reuse many arguments in \cite{BM13}. Our first main result is the following

\begin{theorem*}[\ref{thm:identify-cones}]
Let $X$ be a projective K3 surface and $\vv$ be a Mukai vector of O'Grady type. For any two generic stability conditions $\sigma,\tau \in \stab(X)$ with respect to $\vv$, there is a birational map $\Phi_*: M_\sigma(\vv) \dashrightarrow M_\tau(\vv)$ induced by a derived (anti-)autoequivalence\footnote{An \emph{anti-autoequivalence} is an equivalence of categories from $\cD(X)^{\mathrm{op}}$ to $\cD(X)$, which takes every exact triangle in $\cD(X)$ to an exact triangle in $\cD(X)$, but with all the arrows reversed. All examples of this notion encountered in this paper are compositions of autoequivalences with the derived dual functor $\mathbf{R}\mathcal{H}\!\mathit{om}(-, \cO_X)$; see the discussion after \cite[Theorem 1.1]{BM13}.} on $\cD(X)$, which is an isomorphism in codimension one.
\end{theorem*}

This result is a generalisation of \cite[Theorem 1.1]{BM13}. Although the statements are quite similar, the singular version is more involved. For instance, in the primitive case, every moduli space under a generic stability condition is smooth and has trivial canonical class. Therefore, any birational map between two such moduli spaces naturally extends to an isomorphism in codimension one \cite[Proposition 21.6]{GHJ03}. However, this does not hold in the singular case. 

To fix this issue, we introduce the notation of a \emph{stratum preserving} birational map between two singular moduli spaces of O'Grady type and show that the behaviour of such birational maps in codimension one is as nice as it is in the smooth world; see Proposition \ref{prop:codim-two}. 

A refinement of Theorem \ref{thm:identify-cones}, which generalises \cite[Theorem 1.2]{BM13}, is the following

\begin{theorem*}[\ref{thm:every-model}]
Let $\vv$ be a Mukai vector of O'Grady type and $\sigma \in \stab(X)$ be any generic stability condition with respect to $\vv$. Then
\begin{enumerate}
\item We have a globally defined continuous \bmm map $ \ell: \stab(X) \to \ns (M_\sigma(\vv)) $, which is independent of the choice of $\sigma$. Moreover, for any generic stability condition $\tau \in \stab(X)$, the moduli space $M_\tau(\vv)$ is the birational model corresponding to $\ell_\tau$.
\item If $\cC \subset \stab(X)$ is the open chamber containing $\sigma$, then $\ell(\cC)=\amp(M_\sigma(\vv))$.
\item The image of $\ell$ is equal to $\bigcone(M_\sigma(\vv)) \, \cap \, \mov(M_\sigma(\vv))$. In particular, every $K$-trivial $\bQ$-factorial birational model of $M_\sigma(\vv)$ which is isomorphic to $M_\sigma(\vv)$ in codimension $1$ appears as a moduli space $M_\tau(\vv)$ for some generic stability condition $\tau \in \stab(X)$.
\end{enumerate}
\end{theorem*}

In our situation, $M_\sigma(\vv)$ is a $K$-trivial $\bQ$-factorial variety with canonical (hence log-terminal) singularities. Therefore $M_\sigma(\vv)$ (together with an empty divisor) is a log minimal model. For such symplectic varieties which admit symplectic resolutions, the existence and termination of log-flips have been established in \cite[Corollary 1.4.1]{BCHM10} and \cite[Theorem 4.1]{LP14} respectively. Therefore the log minimal model program works in this case. In particular, every $K$-trivial $\bQ$-factorial birational model of $M_\sigma(\vv)$ which is isomorphic to it in codimension $1$ (in other words, every log minimal model which is log-MMP related to $M_\sigma(\vv)$) can be obtained through a finite sequence of log-flops. Theorem \ref{thm:every-model} shows that every such log minimal model has an interpretation as a Bridgeland moduli space $M_\tau(\vv)$ for some generic stability condition $\tau$. This picture is parallel to \cite[Theorem 1.2]{BM13} which deals with the case of primitive Mukai vectors.

%In our situation, $M_\sigma(\vv)$ is a singular variety with trivial canonical class. We note that $M_\sigma(\vv)$ is not a minimal model in its birational class, because its singularities are canonical and not terminal, but that its symplectic resolution is minimal. However, just as in the smooth world, we can still think of any variety which can be connected to $M_\sigma(\vv)$ by a finite sequence of flops as a `birational model' of $M_\sigma(\vv)$. Then, Theorem \ref{thm:every-model} shows that all such birational models of $M_\sigma(\vv)$ can be exhausted by performing wall crossings on $\stab(X)$ and each flop is precisely realised by crossing a wall. 

As an application of Theorem \ref{thm:every-model}, we can also formulate a Torelli-type theorem for singular moduli spaces of O'Grady type, which is parallel to \cite[Corollary 1.3]{BM13}; see Corollary \ref{cor:torelli} for more details. 

The proof of Theorem \ref{thm:identify-cones} and Theorem \ref{thm:every-model} relies on a complete classification of walls, as stated in Theorem \ref{thm:1st-classification} and Theorem \ref{thm:2nd-classification}, which generalises \cite[Theorem 5.7]{BM13}. Although our proofs follow their approach very closely, the technical details are much more involved. The main difficulty is that the proof of \cite[Theorem 5.7]{BM13} uses many results on irreducible holomorphic symplectic manifolds which are not available to us in the singular world. Each time they use an argument which relies on smoothness in an implicit way, we have to either find a way around it or prove a new version that works in the singular case. One instance of this, which leads to an interesting by-product of the project, goes as follows.

The main ingredient for the \bmm map is the classical construction of determinant line bundles on $M_H(\vv)$; see \cite[Section 8.1]{HuLe10a}. Its algebraic version, which is often referred to as the Mukai morphism, is a map of lattices $\theta_\sigma: \vv^\perp \to \ns(M_\sigma(\vv))$ where the orthogonal complement is taken in the algebraic cohomology $H^*_{\mathrm{alg}}(X,\bZ)$. When $\vv$ is primitive, Yoshioka proved in \cite{Yos01} that $\theta_\sigma$ is an isometry with respect to the Mukai pairing on $H^*_{\mathrm{alg}}(X,\bZ)$ and the Beauville-Bogomolov pairing on $\ns(M_\sigma(\vv))$. 

However, when $\vv$ is a Mukai vector of O'Grady type, such a theorem for Bridgeland moduli spaces does not seem to exist in literature. In fact, to make sense of it, we have to find a well-defined pairing on the N\'{e}ron-Severi lattice of the moduli space in the first place. Luckily, a special case of it, on the level of cohomology, concerning only Gieseker moduli spaces, was proved in \cite{PeRa10a}. Using their approach, we can similarly define the bilinear pairing on $\ns(M_\sigma(\vv))$, or more generally on $H^2(M_\sigma(\vv))$. Although their proof of the isometry does not adapt immediately to the case of Bridgeland moduli spaces, all the key ideas are there and we are able to deduce the desired result using Fourier-Mukai transforms and deformation theory.

A necessary and crucial step in this project, as well as an interesting result in its own right, is the analogous result concerning the cohomological version of the Mukai morphism in the singular setting. In particular, we have

\begin{theorem*}[\ref{thm:hodge-isom}]
Let $\vv$ be a Mukai vector of O'Grady type and $\sigma \in \stab(X)$ be a generic stability condition with respect to $\vv$. Denote $M_\sigma(\vv)$ by $M$ and its symplectic resolution by $\pi:\tilde{M} \to M$. Then
\begin{enumerate}
\item The pullback map $\pi^*: H^2(M,\bZ) \to H^2(\tilde{M},\bZ)$ is injective and compatible with the (mixed) Hodge structures. In particular, $H^2(M,\bZ)$ carries a pure Hodge structure of weight two, and the restriction of the Beauville-Bogomolov quadratic form $\tilde{q}(-,-)$ on $H^2(\tilde{M},\bZ)$ defines a quadratic form $q(-,-)$ on $H^2(M,\bZ)$.
\item There exists a well-defined Mukai morphism $\theta_\sigma^{\mathrm{tr}}: \vv^{\perp,\mathrm{tr}} \to H^2(M,\bZ)$ induced by the (quasi-)universal family over $M_{\sigma}^{\mathrm{st}}(\vv)$, which is a Hodge isometry.
\end{enumerate}
\end{theorem*}

We point out that $\vv^{\perp,\mathrm{tr}}$ in the above theorem denotes the orthogonal complement in the total cohomology $H^*(X,\bZ)$. 

If $\sigma$ lies in the Gieseker chamber then Theorem \ref{thm:hodge-isom} is precisely \cite[Theorem 1.7]{PeRa10a}. The more general statement follows by combining the ideas of Perego and Rapagnetta with some new ingredients inspired by \cite{BM13}. More precisely, for an arbitrary $\sigma \in \stab(X)$, we use a Fourier-Mukai transform to identify $M_\sigma(\vv)$ with a moduli space of twisted sheaves on another K3 surface. Then we deform the underlying twisted K3 surface to an untwisted K3 surface where the isometry has been proved in \cite[Theorem 1.7]{PeRa10a}. The observation in \cite{PeRa10a} that this isometry is preserved under these two operations finishes the proof. However, the presence of a non-trivial Brauer class makes the deformation argument more complicated. For instance, the issue of ampleness caused by $(-2)$-classes and genericness of the polarisation, turn out to be far from straightforward.

As an immediate consequence, we obtain the algebraic version of the Theorem \ref{thm:hodge-isom} as follows.

\begin{corollary*}[\ref{cor:algebraic-mukai}]
Under the assumptions of Theorem \ref{thm:hodge-isom}, we have
\begin{enumerate}
\item Lefschetz $(1,1)$ theorem holds for $M$. That is, $\ns(M)=H^{1,1}(M, \bZ)$;
\item The restriction of the pullback map $\pi$ on $\ns(X)$ is an injective map $\pi^*:\ns(M) \to \ns(\tilde{M})$, which is compatible with the Beauville-Bogomolov pairings $q(-,-)$ and $\tilde{q}(-,-)$;
\item The restriction of the Mukai morphism $\theta_\sigma^{\mathrm{tr}}$ on the algebraic Mukai lattice is an isometry $\theta_\sigma: \vv^\perp \to \ns(M)$. In particular, $q(-,-)$ is a non-degenerate pairing on $\ns(M)$ with signature $(1,\rho(X))$.
\end{enumerate}
\end{corollary*}

We remark that it is only the third statement in the above corollary which is needed in our study of the \bmm map. However, we do not know a direct proof of this statement which does not utilise the cohomological Mukai morphism; the reason being that our proof involves a deformation argument. Unlike cohomology groups which stay constant in a family as a local system, the N\'{e}ron-Severi groups do not behave well under deformations. 

As a consequence of Theorem \ref{thm:hodge-isom}, we obtain a generalisation of \cite[Theorem 1.6]{PeRa10a} from Gieseker moduli spaces to Bridgeland moduli spaces. The proof combines the arguments in \cite{PeRa10a} and the deformation techniques developed in Section \ref{mainpf} below.

\begin{corollary*}[\ref{cor:deformation-equivalent}]
Let $\vv$ be a Mukai vector of O'Grady type and $\sigma \in \stab(X)$ be any generic stability condition with respect to $\vv$. Then the symplectic resolution of $M_\sigma(\vv)$ is deformation equivalent to the irreducible holomorphic symplectic manifold constructed by O'Grady in \cite{OG99a}.
\end{corollary*}

This result in particular implies that, by resolving singular moduli spaces of Bridgeland semistable objects on K3 surfaces, we cannot get any new deformation types of irreducible holomorphic symplectic manifolds other than the one discovered by O'Grady in \cite{OG99a}. It is somewhat disappointing, but supports the long-standing belief that deformation types of irreducible holomorphic symplectic manifolds are rare.

\subsection{Outline of the paper}

In Section \ref{singmodsp}, we collect together the necessary properties of moduli spaces of O'Grady type that we will need. After briefly mentioning their stratifications and resolutions, we study birational maps between them and state Theorem \ref{thm:hodge-isom}. We also give a brief account of various cones of divisors on these moduli spaces. 

Section \ref{mainpf} is devoted to the proof of Theorem \ref{thm:hodge-isom}. We provide some lemmas on the deformations of twisted polarised K3 surfaces, as well as on the existence of local relative moduli spaces. The main proof comes after all these lemmas. 

In Section \ref{localbmmap}, we briefly review the \bmm map constructed in \cite{BM12}, but only from the aspect that will become important in our later discussion. We also use the \bmm map to generalise the ampleness results proved in \cite{BM12}. We refer interested readers to \cite[Section 3]{BM12} for the original construction, which provides a very conceptual way to understand the positivity lemma there. 

We state our classification theorems of potential walls in Section \ref{classificationresults}; see Theorems \ref{thm:1st-classification} and \ref{thm:2nd-classification}. We also establish the birational maps relating moduli spaces for two chambers separated by a wall; see Theorem \ref{thm:identify-cones}. In this section we only state the results, while leaving all proofs for next section. It is worth pointing out that our first classification theorem is true for arbitrary Mukai vectors, while the second classification theorem only works for Mukai vectors of O'Grady type. 

Section \ref{classificationproofs} is devoted to the proof of all results in the section 5. We try to make our proofs short and avoid repeating any existing arguments by making many references to results in \cite[Section 6-9]{BM13}. Readers who are not interested in proofs can safely skip this section without affecting the coherence of logic. 

In Section \ref{globalbmmap}, we describe and prove our main result. In particular, Theorem \ref{thm:every-model} provides a precise relationship between wall crossings on the stability manifold and the birational geometry of the corresponding moduli spaces. 

Our main references of the paper are \cite{BM12,BM13}. Since our presentation closely follows theirs, all background knowledge required for this paper are already included in the first few sections of those. Nevertheless, we recommend readers the following references for general knowledge of some relevant topics: \cite{HuLe10a} for moduli spaces of sheaves, \cite[Part 3]{GHJ03} for geometry of irreducible holomorphic symplectic manifolds, \cite{Bri08} for stability conditions on K3 surfaces, and \cite{Cal00} for twisted sheaves. 

\begin{acknowledgements}
We are most grateful to Arend Bayer and Alastair Craw for all their help, support and encouragement throughout this project. Special thanks to Daniel Huybrechts for his invaluable advice and guidance at various stages of this work. We also thank Christian Lehn, S\"{o}nke Rollenske and K\={o}ta Yoshioka for kindly answering our questions, and the referees for their very helpful comments. Z. Z. would also like to thank Jun Li for his initial suggestion of looking into this topic. C. M. is supported by an EPSRC Doctoral Prize Research Fellowship Grant EP/K503034/1 and Z. Z. is supported by an EPSRC Standard Research Grant EP/J019410/1. We also appreciate the support from the University of Bonn, the Max Planck Institute for Mathematics, and the SFB/TR-45 during the initial stage of this project as well as the Hausdorff Research Institute for Mathematics for its conclusion. 
\end{acknowledgements}

\section{Singular Moduli Spaces of O'Grady Type}\label{singmodsp}

We start our discussion by collecting some useful properties of the singular moduli spaces of O'Grady type, which will be the central geometric objects that are studied in the whole paper. We will see that despite of the singularities, these moduli spaces behave very similar to smooth moduli spaces, in the sense that many nice properties of smooth moduli spaces can be generalised to these with some extra care of the singular loci. After introducing necessary notations and background materials, we will mainly focus on two aspects of these moduli spaces: birational maps between them and Mukai morphisms on their cohomology. After that, we will briefly mention various cones of divisors on singular moduli spaces of O'Grady type.

\subsection{Moduli spaces and symplectic resolutions}

We start by recalling the basic notion of a Bridgeland moduli space. Let $X$ be a projective K3 surface and $\vv \in H^*_{\mathrm{alg}}(X, \bZ)$ be a Mukai vector. Throughout this paper we will always assume $\vv^2>0$. Moreover, there is a unique way to write $\vv=m\vv_p$ for some positive integer $m$ and primitive class $\vv_p \in H^*_{\mathrm{alg}}(X, \bZ)$. We say $m$ is the \emph{divisibility} and $\vv_p$ is the \emph{primitive part} of $\vv$. Let $\sigma \in \stab(X)$ be a Bridgeland stability condition in the distinguished component of the stability manifold. The stability manifold $\stab(X)$ comes with a wall and chamber structure with respect to $\vv$ as described above (see also \cite[Section 9]{Bri08} and \cite[Proposition 2.3]{BM12}), and we say $\sigma \in \stab(X)$ is \emph{generic} if it does not lie on any wall. 

It was proven in \cite[Theorem 1.3]{BM12} (which generalises \cite[Theorem 0.0.2]{MYY11b}) that, for a stability condition $\sigma \in \stab(X)$ which is generic with respect to $\vv$, there exists a coarse moduli space $M_{X,\sigma}(\vv)$, which parametrises the S-equivalence classes of $\sigma$-semistable objects of class $\vv$ on $X$. Furthermore, it is a normal projective irreducible variety with $\bQ$-factorial singularities. By \cite[Theorem 2.15]{BM13} (or originally \cite{Yos01,Toda08}), $M_{X, \sigma}(\vv)$ is non-empty, and a generic point of it represents a $\sigma$-stable object. 

A classical theorem, originally proved by Mukai in \cite{Muk84} (see also \cite[Theorem 6.10]{BM12} and \cite[Theorem 3.6]{BM13}), says that when $\sigma$ is generic and $\vv$ is primitive, the moduli space $M_{X,\sigma}(\vv)$ is an irreducible holomorphic symplectic manifold, which parametrises $\sigma$-stable objects of class $\vv$ only. However, when $\vv$ is non-primitive, it is proved in \cite[Theorem 6.2]{KaLeSo06a} for Gieseker moduli spaces and \cite[Theorem 3.10]{BM13} for Bridgeland moduli spaces, that $M_{X, \sigma}(\vv)$ has symplectic singularities.

For simplicity, we often drop the K3 surface $X$ from the notation when it is clear from the context. In fact, the moduli space only depends on the choice of the chamber $\cC$ containing $\sigma$. Therefore we sometimes also denote the moduli space for any stability condition contained in a chamber $\cC$ by $M_{X,\cC}(\vv)$, or simply $M_\cC(\vv)$. For the convenience of later discussion, we also make the following definition.

\begin{definition}
\label{def:ogrady}
We say a Mukai vector $\vv \in H^*_{\mathrm{alg}}(X,\bZ)$ is of \emph{O'Grady type} if it can be written as $\vv=m\vv_p$, where $m=2$ and $\vv_p^2=2$. We say a moduli space $M_{X,\sigma}(\vv)$ is of \emph{O'Grady type} if $\sigma$ is a generic stability condition, and $\vv$ is a Mukai vector of O'Grady type. 
\end{definition}

The importance of this particular type of moduli spaces lies in the study of symplectic resolutions of $M_{X,\sigma}(\vv)$. In \cite{OG99a}, O'Grady constructed a symplectic resolution of a Gieseker moduli space with Mukai vector $\vv=(2,0,-2)$ and showed that it was not deformation equivalent to any existing example of homomorphic symplectic manifolds at the time. In \cite{LS06}, Lehn and Sorger generalised the result to arbitrary Gieseker moduli spaces of O'Grady type, and gave a slightly different description of their symplectic resolutions. It was proved in \cite{PeRa10a} that all these symplectic resolutions are in fact deformation equivalent to the one constructed by O'Grady. In \cite{KaLeSo06a}, it was proved that for a generic Gieseker moduli space with any other non-primitive Mukai vector, a symplectic resolution does not exist. 

By using the techniques developed in \cite[Section 7]{BM12}, we can easily generalise the existence of symplectic resolutions to Bridgeland moduli spaces with the following proposition. Later in Corollary \ref{cor:deformation-equivalent}, we will show that all these symplectic resolutions are still deformation equivalent to the one constructed by O'Grady, hence do not provide new deformation type of irreducible holomorphic symplectic manifolds.

\begin{proposition}
Let $X$ be a projective K3 surface and $\vv=m \vv_p \in H^*_{\mathrm{alg}}(X,\bZ)$ be a Mukai vector with $m \geq 2$, $\vv_p$ primitive and $\vv_p^2>0$. Let $\sigma \in \stab(X)$ be generic with respect to $\vv$. Then
\begin{itemize}
\item If $m=2$ and $\vv_p^2=2$, then $M_\sigma(\vv)$ admits a symplectic resolution;
\item If $m > 2$ or $\vv_p^2 > 2$, then $M_\sigma(\vv)$ does not admit a symplectic resolution.
\end{itemize}
\end{proposition}

\begin{proof}
By \cite[Lemma 7.3]{BM12} and the discussion after that, there exists a twisted K3 surface $(Y, \alpha)$ where $\alpha \in \br(Y)$ and a derived equivalence $\Phi: \cD(X) \to \cD(Y, \alpha)$ in the form of a Fourier-Mukai transform, such that $\Phi$ induces an isomorphism $M_{X,\sigma}(\vv) \cong M_{Y,\alpha,H}(-\Phi(\vv))$, where $M_{Y,\alpha,H}(-\Phi(\vv))$ is the moduli space of $\alpha$-twisted Gieseker $H$-semistable locally free sheaves on $Y$. Hence it suffices to prove the claims for twisted Gieseker moduli spaces. Therefore without loss of generality, we replace the Bridgeland moduli space in question by a twisted Gieseker moduli space $M_{X,\alpha,H}(\vv)$.

By \cite[Proposition 2.3.3.6]{Lie07}, there is a GIT construction for the twisted Gieseker moduli spaces, which is precisely the same as in the case of untwisted Gieseker moduli spaces. And by \cite[Proposition 2.2.4.9]{Lie07}, the local deformation theory of twisted sheaves is also the same as that of untwisted sheaves. Therefore the argument in \cite{LS06} shows that $M_{X,\alpha,H}(\vv)$ admits a symplectic resolution as the blowup of its singular locus when $m=2$ and $\vv_p^2=2$. And the argument in \cite{KaLeSo06a} shows that $M_{X,\alpha,H}(\vv)$ has no symplectic resolution when $m>2$ or $\vv_p^2>2$. 
\end{proof}

The existence of symplectic resolutions is critical for most results in the present paper. The following proposition is such an example. One property of irreducible holomorphic symplectic manifolds used in \cite{BM13} is that, for a divisorial contraction on an irreducible holomorphic symplectic manifold, the image of the contracted divisor has codimension exactly two, which is a special case of \cite[Proposition 1.4]{Nam01}, \cite[Theorem 1.2]{Wie03} or \cite[Lemma 2.11]{Kal06}. For our purpose, we need a version of this result in singular case. Thanks to the existence of symplectic resolutions, the same result can be easily proved for moduli spaces of O'Grady type. But we nevertheless state it under a more general setup as follows.

\begin{proposition}
\label{prop:symp-divisor}
Let $M$ be a variety with symplectic singularities of dimension $2n$ admitting a symplectic resolution, and let $N$ be a normal projective variety. Let $\varphi: M \to N$ be a birational projective morphism. We denote by $S_i$ the set of points $p\in N$ such that $\dim \varphi^{-1}(p)=i$. Then $\dim S_i\leq 2n-2i$. In particular, if $\varphi$ contracts a divisor $D\subset M$, then we have $\dim\varphi(D)=2n-2$.
\end{proposition}

\begin{proof}
The statement in question is equivalent to the following statement: if $V \subset M$ is any closed subvariety of dimension $2n-i$, then $\dim \varphi(V) \geq 2n-2i$; see for instance the proof of \cite[Proposition 1.4]{Nam01}. Without loss of generality, it suffices to prove this statement with the extra assumption that $V$ is irreducible with the generic point $\xi_V$.

Let the symplectic resolution of $M$ be $\pi: \tilde{M} \to M$. We denote the closure of $\pi^{-1}(\xi_V)$ in $\tilde{M}$ by $\tilde{V}$. Then we have $\pi(\tilde{V})=V$, hence $\dim \tilde{V} \geq \dim V =2n-i$. Then we apply \cite[Proposition 1.4]{Nam01}, \cite[Theorem 1.2]{Wie03} or \cite[Lemma 2.11]{Kal06} on the composition $\varphi\circ\pi: \tilde{M}\to N$ and conclude $\dim\varphi(V)=\dim(\varphi\circ\pi)(\tilde{V})\geq 2n-2i$.
\end{proof}

However, we remind the readers that although many of our results are proved in the context of definition \ref{def:ogrady}, some of our results do work in more general situations. We will state very clearly which assumptions are made in every result.

\subsection{Stratum preserving birational maps}

Whenever $\sigma$ is generic, there is a stratification of the moduli space $M_{X,\sigma}(\vv)$ given by locally closed strata. The stable locus $M_{X,\sigma}^{\mathrm{st}}(\vv)$, which agrees with the smooth locus, is the unique open stratum. All the other lower dimensional strata are formed by lower dimensional moduli spaces. We refer the readers to the proof of \cite[Theorem 2.15]{BM13} for the general case. Here we only describe the stratification for moduli spaces of O'Grady type. 

For a moduli space of O'Grady type $M_{X,\sigma}(\vv)$, there is a chain of closed subschemes as follows:
\begin{equation}
\label{eqn:stratification}
M_{X,\sigma}(\vv_p) \subset \sym^2 M_{X,\sigma}(\vv_p) \subset M_{X,\sigma}(\vv),
\end{equation}
where the first inclusion is given by the diagonal morphism, and the second inclusion gives precisely the strictly semistable locus of $M_{X,\sigma}(\vv)$, which agrees with the singular locus. This chain of inclusions decomposes $M_{X,\sigma}(\vv)$ into the disjoint union of three locally closed strata.

More precisely, $M_{X,\sigma}^{\mathrm{st}}(\vv) = M_{X,\sigma}(\vv) \backslash \sym^2 M_{X,\sigma}(\vv_p)$ parametrises all $\sigma$-stable objects of class $\vv$. Every point in $\sym^2 M_{X,\sigma}(\vv_p)$ represents the S-equivalent class containing a polystable object $E_1 \oplus E_2$, where $E_1, E_2 \in M_{X,\sigma}(\vv_p)$ are both $\sigma$-stable objects of class $\vv_p$. Such a point lies in the diagonal $M_{X,\sigma}(\vv_p)$ if and only if $E_1$ and $E_2$ are isomorphic.

With this stratification at hand, we are now ready to discuss birational maps between moduli spaces of O'Grady type, which are compatible with the above stratifications.

When talking about birational maps between singular moduli spaces of O'Grady type, we emphasise a special class of them, which preserve the natural stratifications described above. Almost all birational maps between these moduli spaces which occur in this paper belong to this class. Although the definition could be made for arbitrary moduli spaces under generic stability conditions, for the purpose of this paper, we restrict ourselves to moduli spaces of O'Grady type as follows.

\begin{definition}
Let $f: M_{X_1,\sigma_1}(\vv_1) \dashrightarrow M_{X_2,\sigma_2}(\vv_2)$ be a birational map between two moduli spaces of O'Grady type. Let $M_{X_i,\sigma_i}(\vv_{i,p}) \subset \sym^2 M_{X_i,\sigma_i}(\vv_{i,p}) \subset M_{X_i,\sigma_i}(\vv_i)$ be the standard stratification \eqref{eqn:stratification} for $i=1,2$, where $\vv_{i,p}$ denotes the primitive part of $\vv_i$. If $f$ is defined on the generic point of each stratum in $M_{X_1,\sigma_1}(\vv_1)$, and takes each such generic point to the generic point of the corresponding stratum in $M_{X_2,\sigma_2}(\vv_2)$, then we say that $f$ is a \emph{stratum preserving} birational map. 
\end{definition}

Derived (anti-)equivalences are a very natural and particularly rich resource of stratum preserving birational maps. The following lemma gives a criterion for such a derived (anti-)equivalence to induce a stratum preserving birational map. 

\begin{lemma}
\label{lem:stratum}
Let $\Phi: \cD(X_1) \to \cD(X_2)$ be a derived (anti-)equivalence and let $M_{X_1,\sigma_1}(\vv_1)$ and $M_{X_2,\sigma_2}(\vv_2)$ be two moduli spaces of O'Grady type. Assume that $\Phi$ induces a birational map $\Phi_*: M_{X_1,\sigma_1}(\vv_1) \dashrightarrow M_{X_2,\sigma_2}(\vv_2)$. Then $\Phi_*$ is stratum preserving if and only if the following condition holds: there exist a $\sigma_1$-stable object $E$ of class $\vv_1$ and a $\sigma_1$-stable object $E_p$ of class $\vv_{1,p}$, such that $\Phi(E)$ and $\Phi(E_p)$ are $\sigma_2$-stable objects of classes $\vv_2$ and $\vv_{2,p}$ respectively.
\end{lemma}

\begin{proof}
The necessity is part of the definition of $f$ being stratum preserving, so we only discuss sufficiency. By the openness of stability in \cite[Theorem 4.2]{BM13} (which was originally proved in \cite{Toda08}), the assumptions imply that the induced birational map $\Phi_*$ takes the generic points of the open stratum $M_{X_1,\sigma_1}^{\mathrm{st}}(\vv_1)$ and the closed stratum $M_{X_1,\sigma_1}(\vv_{1,p})$ to the generic points of corresponding strata in $M_{X_2,\sigma_2}(\vv_2)$. It remains to show that $\Phi_*$ takes the generic point of the singular locus $\sym^2 M_{X_1,\sigma_1}(\vv_{1,p})$ to the generic point of $\sym^2 M_{X_2,\sigma_2}(\vv_{2,p})$. In fact, a generic point $E_s$ in $\sym^2 M_{X_1,\sigma_1}(\vv_{1,p})$ can be represented by any extension of two generic stable objects of class $\vv_{1,p}$. Since $\Phi$ is a derived (anti-)equivalence, it preserves extensions (or switches the direction). Hence $\Phi(E_s)$ is again the extension of two generic stable objects of class $\vv_{2,p}$, which represents a generic point in $\sym^2 M_{X_2,\sigma_2}(\vv_{2,p})$, as desired. 
\end{proof}

A big advantage of stratum preserving birational maps is that they behave very much like birational maps between smooth symplectic varieties. For example, the following proposition generalises a classical result about a birational map between two $K$-trivial smooth varieties, for instance, in \cite[Proposition 21.6]{GHJ03}. 

\begin{proposition}
\label{prop:codim-two}
Let $f: M_{X_1,\sigma_1}(\vv_1) \dashrightarrow M_{X_2,\sigma_2}(\vv_2)$ be a stratum preserving birational map between two moduli spaces of O'Grady type which is induced by a derived (anti-)equivalence $\Phi:\cD(X_1) \to \cD(X_2)$. Furthermore, assume that there exists an open subset $U \subset M_{X_1,\sigma_1}^{\mathrm{st}}(\vv_1)$ with complement of codimension at least two, such that the restriction $f|_U$ is an injective morphism $f|_U: U \to M_{X_2,\sigma_2}^{\mathrm{st}}(\vv_2)$. Then $f(U)$ has complement of codimension at least two in $M_{X_2,\sigma_2}^{\mathrm{st}}(\vv_2)$.
\end{proposition}

\begin{proof}
Since $f$ is a stratum preserving birational map, there exists an open subset $U_s \subset \sym^2 M_{X_1,\sigma_1}(\vv_{1,p})$, such that $f|_{U_s}: U_s \to \sym^2 M_{X_2,\sigma_2}(\vv_{2,p})$ is an injective morphism. Now we take $Z_1$ to be the closure of $M_{X_1,\sigma_1}^{\mathrm{st}}(\vv_1) \backslash U$ in $M_{X_1,\sigma_1}(\vv_1)$, $Z_2$ to be $\sym^2 M_{X_1,\sigma_1}(\vv_{1,p}) \backslash U_s$, and $Z_3$ to be the closed stratum $M_{X_1,\sigma_1}(\vv_{1,p})$. We consider $V = M_{X_1,\sigma_1}(\vv_1) \backslash (Z_1 \cup Z_2 \cup Z_3)$. Then it is easy to see that $V$ is an open subset of $M_{X_1,\sigma_1}(\vv_1)$. Moreover, $V$ is the union of $U$ and an open subset of $U_s$, hence has a complement of codimension at least two, and the restriction $f|_V$ is an injective morphism $f|_V: V \to M_{X_2,\sigma_2}(\vv_2)$. Since $f$ is induced by a derived (anti-)equivalence $\Phi$, we see that $\Phi^{-1}$ defines an inverse of $f$ on $f(V)$, and therefore $f$ is an isomorphism from $V$ to its image $f(V)$. Note that since $V$ has no intersection with the closed stratum $M_{X_1,\sigma_1}(\vv_{1,p})$, $f(V)$ also has no intersection with the closed stratum $M_{X_2,\sigma_2}(\vv_{2,p})$.

For $i=1,2$, we write $\pi_i: \tilde{M}_i \to M_{X_i,\sigma_i}(\vv_i)$ for the symplectic resolution constructed in \cite{OG99a,LS06}. The construction there implies that both $\pi_1: \pi_1^{-1}(V) \to V$ and $\pi_2: \pi_2^{-1}(f(V)) \to f(V)$ are exactly the blowups of the singular loci. Therefore, the isomorphism $f: V \to f(V)$ induces another isomorphism $\tilde{f}: \pi_1^{-1}(V) \to \pi_2^{-1}(f(V))$. In particular, it is a birational map $\tilde{f}: \tilde{M}_1 \dashrightarrow \tilde{M}_2$. 

We claim that $\pi_1^{-1}(V) \subset \tilde{M}_1$ is an open subset with complement of codimension at least two. On one hand, from the construction of $V$ we observe that the complement of $V$ in $M_{X_1,\sigma_1}^{\mathrm{st}}(\vv_1)$ has codimension at least two. Together with the fact that $\pi_1: \pi_1^{-1}(M_{X_1,\sigma_1}^{\mathrm{st}}(\vv_1)) \to M_{X_1,\sigma_1}^{\mathrm{st}}(\vv_1)$ is an isomorphism, we conclude that the complement of $\pi_1^{-1}(V)$ in $\pi_1^{-1}(M_{X_1,\sigma_1}^{\mathrm{st}}(\vv_1))$ also has codimension at least two. On the other hand, since $V$ contains an open subset of the singular locus $\sym^2 M_{X_1,\sigma_1}(\vv_{1,p})$, we obtain that $\pi_1^{-1}(V)$ contains an open subset of the unique exceptional divisor. Therefore $\pi_1^{-1}(V)$ has a complement of codimension two in $\tilde{M}_1$. 

Now we can apply \cite[Proposition 21.6]{GHJ03} to the birational map $\tilde{f}: \tilde{M}_1 \dashrightarrow \tilde{M}_2$, and conclude that $\pi_2^{-1}(f(V)) = \tilde{f}(\pi_1^{-1}(V)) \subset \tilde{M}_2$ has complement of codimension at least two. This further implies $f(V) \subset M_{X_2,\sigma_2}(\vv_2)$ also has complement of codimension at least two. Therefore $f(U)$, as the intersection of $f(V)$ with the open stratum $M_{X_2,\sigma_2}^{\mathrm{st}}(\vv_2)$, has complement of codimension at least two as well.
\end{proof}

\subsection{Mukai morphisms are isomorphisms}

The Mukai morphism plays an essential role in \cite{BM13}. A classical theorem \cite[Theorem 3.6]{BM13}, originally proved in \cite{Muk87,Yos01}, shows that for a smooth moduli space $M_\sigma(\vv)$ of stable objects on a K3 surface $X$ with $\vv^2>0$, the Mukai morphism induced by a (quasi-)universal family is in fact a Hodge isometry between the orthogonal complement of $\vv$ in the total cohomology $\vv^{\perp,\mathrm{tr}} \subset H^*(X, \bZ)$ and $H^2(M_\sigma(\vv), \bZ)$. By restricting on the algebraic components on both sides, we get an isometry between the orthogonal complement of $\vv$ in the algebraic cohomology $\vv^\perp \subset H^*_{\mathrm{alg}}(X, \bZ)$ and $\ns(M_\sigma(\vv))$. 

Perego and Rapagnetta generalised this classical result to generic Gieseker moduli spaces of O'Grady type in \cite[Theorem 1.7]{PeRa10a}. Here we will follow their approach to generalise the same result further to generic Bridgeland moduli spaces of O'Grady type. This is the content of the following

\begin{theorem}
\label{thm:hodge-isom}
Let $\vv$ be a Mukai vector of O'Grady type and $\sigma \in \stab(X)$ be a generic stability condition with respect to $\vv$. Let $M=M_\sigma(\vv)$ be the moduli space of $\sigma$-semistable objects of class $\vv$ on $X$ and $\pi: \tilde{M} \to M$ be its symplectic resolution. Then we have
\begin{enumerate}
\item The pullback map $\pi^*: H^2(M,\bZ) \to H^2(\tilde{M},\bZ)$ is injective and compatible with the (mixed) Hodge structures. In particular, the Hodge structure on $H^2(M,\bZ)$ is pure of weight two and the restriction of the Beauville-Bogomolov quadratic form $\tilde{q}(-,-)$ on $H^2(\tilde{M},\bZ)$ defines a quadratic form $q(-,-)$ on $H^2(M,\bZ)$; \label{item:hodge1}
\item There exists a well-defined Mukai morphism $\theta_\sigma^{\mathrm{tr}}: \vv^{\perp,\mathrm{tr}} \to H^2(M,\bZ)$ induced by the (quasi-)universal family over $M_{\sigma}^{\mathrm{st}}(\vv)$, which is a Hodge isometry. \label{item:hodge2}
\end{enumerate}
\end{theorem}

The proof of Theorem \ref{thm:hodge-isom} is postponed to the next section. We continue our discussion with the following interesting consequence, which will be very important later.

\begin{corollary}
\label{cor:algebraic-mukai}
Under the assumptions of the Theorem \ref{thm:hodge-isom}, we have
\begin{enumerate}
\item Lefschetz $(1,1)$ theorem holds for $M$. That is, $\ns(M)=H^{1,1}(M, \bZ)$; \label{item:am1}
\item The restriction of the pullback map $\pi$ on $\ns(X)$ is an injective map $\pi^*:\ns(M) \to \ns(\tilde{M})$, which is compatible with the Beauville-Bogomolov pairings $q(-,-)$ and $\tilde{q}(-,-)$; \label{item:am2}
\item The restriction of the Mukai morphism $\theta_\sigma^{\mathrm{tr}}$ on the algebraic Mukai lattice is an isometry $\theta_\sigma: \vv^\perp \to \ns(M)$. In particular, $q(-,-)$ is a non-degenerate pairing on $\ns(M)$ with signature $(1,\rho(X))$. \label{item:am3}
\end{enumerate}
\end{corollary}

\begin{proof}
For simplicity, we still denote $M_\sigma(\vv)$ by $M$. By Theorem \ref{thm:hodge-isom}, the Hodge structure on $H^2(M, \bZ)$ is pure of weight two, and $\pi^*$ preserves the Hodge structure. Therefore, for any class $\alpha \in H^{1,1}(M, \bZ)$, we have $\pi^*\alpha \in H^{1,1}(\tilde{M},\bZ)$ and hence $\pi^*\alpha=c_1(\tilde{L})$ for some line bundle $\tilde{L}$ on $\tilde{M}$. By O'Grady's construction of the symplectic resolution $\tilde{M}$ in \cite{OG99a}, we know that a generic fibre of $\pi$ within the exceptional divisor is a smooth rational curve. Let $C$ be such a rational curve, then $\tilde{L} \cdot C=\pi^*\alpha \cdot [C]=\alpha \cdot \pi_*[C]=0$, which implies that the restriction of $\tilde{L}$ on $C$ is trivial. Since $M$ is normal by \cite[Theorem 3.10]{BM13} (or originally \cite[Theorem 4.4]{KaLeSo06a}), we must have $\tilde{L}=\pi^*L$ for some line bundle $L$ on $M$. Therefore we have $\alpha=c_1(L) \in \ns(M)$ and the Lefschetz $(1,1)$ theorem is true for moduli spaces $M$ of O'Grady type.

By taking the $(1,1)$ components on both sides of the map $\pi^*$ between the second cohomology groups, we get the map between the N\'{e}ron-Severi lattices. Similarly, we can take the $(1,1)$ components on both sides of the Hodge isometry $\theta_\sigma^{\mathrm{tr}}$ to get the desired isometry $\theta_\sigma$.
\end{proof}

\begin{remark}
We briefly describe how the map $\theta_\sigma^{\mathrm{tr}}$ will be constructed in the proof of Theorem \ref{thm:hodge-isom}. Indeed, we will see that it is the unique lift of $\Phi^{\cE}$ along $i^*$ in the diagram
\begin{equation*}
\xymatrix{
\vv^{\perp, \mathrm{tr}} \ar@{-->}[rr]^-{\theta_\sigma^{\mathrm{tr}}} \ar[drr]_{\Phi^{\cE}} & & H^2(M_\sigma(\vv), \bZ) \ar[d]^{i^*} \\
 & & H^2(M_\sigma^{\mathrm{st}}(\vv), \bQ),
}
\end{equation*}
where $\Phi^{\cE}$ is the classical Mukai morphism induced by the (quasi-)universal family $\cE$ on the stable locus $M_\sigma^{\mathrm{st}}(\vv)$ of the moduli space, and $i^*$ is the pullback along an open embedding. The map $\theta_\sigma$, which is the restriction of $\theta_\sigma^\mathrm{tr}$, is sometimes also referred to as the Mukai morphism in the literature. However, we prefer to call it the \emph{algebraic Mukai morphism}, to distinguish it from the Mukai morphism $\theta_\sigma^{\mathrm{tr}}$, which includes the $(2,0)$ and $(0,2)$ components on both sides.
\end{remark}

\begin{remark}
We would also like to point out that, although $\theta_\sigma^{\mathrm{tr}}$ and $\theta_\sigma$ a priori depend on the choice of the generic stability condition $\sigma$, they in fact only depend on the choice of the open chamber $\cC \in \stab(X)$ containing $\sigma$. This is because the moduli space is the same for all interior points of $\cC$. Therefore, in \cite{BM12,BM13}, the Mukai morphism and algebraic Mukai morphism are sometimes also denoted by $\theta_\cC^{\mathrm{tr}}$ and $\theta_\cC$ respectively.
\end{remark}

We conclude this section by briefly mentioning various cones of divisors on $M=M_\sigma(\vv)$. The ample cone $\amp(M)$, big cone $\bigcone(M)$ and movable cone $\mov(M)$ are all well-defined. Due to the existence of the symplectic resolution and Corollary \ref{cor:algebraic-mukai}, we can also define the positive cone of $M$. We will show that the following definition justifies the name.

\begin{definition}
The cone $(\pi^*)^{-1}(\pos(\tilde{M})) \subset \ns(M)$ is called the positive cone of $M$, and is denoted by $\pos(M)$.
\end{definition}

We can see from the following proposition that the notion is reasonably defined and agrees with our intuition.

\begin{proposition}
\label{prop:amp-pos}
The positive cone $\pos(M)$ is one of the two components of $\{ \alpha \in \ns(M) : q(\alpha,\alpha)>0 \}$ and contains the ample cone $\amp(M)$. 
\end{proposition}

\begin{proof}
In fact, by Corollary \ref{cor:algebraic-mukai}\eqref{item:am3}, we know that the cone $\{ \alpha \in \ns(M) : \alpha^2>0 \}$ has two components. Together with the map $\pi^*$ in Corollary \ref{cor:algebraic-mukai}\eqref{item:am2}, we know they are precisely the restrictions of the two components of $\{ \tilde{\alpha} \in \ns(\tilde{M}) : \tilde{q}(\tilde{\alpha}, \tilde{\alpha})>0 \}$ to $\ns(M)$, one of which is $\pos(\tilde{M})$. This proves the first statement. 

If there is no ample class then there is nothing to prove for the second statement. Otherwise, take any ample class $\alpha \in \amp(M)$ and note that $\tilde{\alpha}:=\pi^*\alpha$ is nef and big on $\tilde{M}$. By \cite[Proposition 2.61]{KoMo98a},  this implies $\int_{\tilde{M}}\tilde{\alpha}^{10}>0$. Now the Beauville-Fujiki relation \cite[Proposition 23.14]{GHJ03} implies $\tilde{q}(\tilde{\alpha}, \tilde{\alpha})>0$. Thus we have $\tilde{\alpha} \in \pos(\tilde{M})$ and hence $\alpha \in \pos(M)$. 
\end{proof}

\begin{remark}
A priori, the ample cone $\amp(M)$ could be empty. However, it is proved in \cite[Theorem 1.3]{BM12} that $M$ always carries ample classes. Therefore, between the two components of the cone of square positive classes on $M$, $\pos(M)$ can be simply identified as the one which contains $\amp(M)$.
\end{remark}

\section{Proof of Theorem \ref{thm:hodge-isom}}\label{mainpf}

This whole section is devoted to the proof of Theorem \ref{thm:hodge-isom}. We start with some lemmas about deformations of twisted K3 surfaces and local existence of relative moduli spaces of twisted sheaves. The proof of Theorem \ref{thm:hodge-isom} will follow after these preparations. As in \cite{BM12,BM13}, when talking about twisted sheaves, we always assume that we have a fixed B-field lift of the Brauer class, which was introduced in \cite{HS05}. 

\subsection{Deformations of twisted polarised K3 surfaces}

In this subsection we study deformations of a polarised K3 surface which carry a non-trivial Brauer class with a B-field lift. Our main results here are Proposition \ref{prop:deform-no-mukai} and Proposition \ref{prop:deform-with-mukai}. Roughly speaking, up to changing the B-field by an integral class, we can always deform a twisted polarised K3 surface to an untwisted polarised K3 surface in the period domain. Moreover, if there is a Mukai vector on the initial K3 surface which remains algebraic under deformations, and the initial polarisation is generic with respect to this Mukai vector, then the deformation can be made so that the polarisation is generic with respect to the corresponding Mukai vector on each fibre. Moreover, Lemma \ref{lem:technical} shows how to find such an integral class so as to amend a given B-field and make the above deformations possible.

We briefly recall the necessary notions required for the following discussion. The (cohomological) \emph{Brauer group} of a K3 surface $X$ is the torsion part of the cohomology group $H^2(X, \cO^*_X)$ in the analytic (or \'etale) topology. A \emph{twisted K3 surface} is a K3 surface $X$ equipped with a Brauer class $\alpha$. Using the exponential sequence, we can always find a rational class $B\in H^2(X, \bQ)$, such that its $(0,2)$-component maps to $\alpha$ under the exponential map, i.e. $\exp(B^{0,2})=\alpha$. We call such a rational class $B$ a \emph{rational B-field} lift of the Brauer class $\alpha$. Note that the B-field lift of any given Brauer class $\alpha$ is not unique.

For every $\alpha$-twisted sheaf $E$, a twisted Chern character of $E$, and hence a twisted Mukai vector of $E$, is defined in \cite[Proposition 1.2]{HS05}, which depends on the choice of the $B$-field lift $B$ of $\alpha$. The construction there guarantees that it is a $B$-twisted \emph{algebraic class}, i.e. a class in the $B$-twisted algebraic cohomology group $H^*_{\mathrm{alg}}(X,B,\bZ):=(\exp(B)\cdot H^*_{\mathrm{alg}}(X, \bQ)) \cap H^*(X,\bZ)$, as defined in \cite[Remark 1.3]{HS05}.

We introduce the following notion for simplicity of presentation: assume we have a family of K3 surfaces $\cX \to S$, and each fibre $X_s$ over the point $s \in S$ is equipped with a B-field $B_s \in H^2(X_s, \bQ)$. If these B-fields form a section of the local system over $S$ with fibres given by $H^2(X_s, \bQ)$, then we say the family of B-fields is \emph{locally constant} over $S$. 

The following lemma shows the locally trivial extension of some cohomology classes could remain algebraic on each fibre of a deformation.

% \begin{definition}
% \label{def:deformable-mukai}
% Let $(X,H)$ be a polarised K3 surface and $B$ a rational B-field lift of a certain Brauer class on $X$. We say that $\vv \in H^*(X, \bQ)$ is a \emph{deformable} $B$-twisted Mukai vector if it can be written in the form of $\vv=(r, qh+rB, a)$ for some positive integer $r$, rational number $q$ and integer $a$, where $h=c_1(H)$.
% \end{definition}

% The importance of this notion lies in the following lemma, which shows the locally trivial extension of a deformable twisted Mukai vector is algebraic on each fibre of a deformation. 

\begin{lemma}
\label{lem:deformable-mukai}
Let $(X,H)$ be a polarised K3 surface and $B$ a rational B-field lift of a certain Brauer class on $X$. Let $\vv \in H^*(X, \bZ)$ be of the form $\vv=(r, qh+rB, a)$ for some positive integer $r$, rational number $q$ and integer $a$, where $h=c_1(H)$. Then $\vv \in H^*_{\mathrm{alg}}(X,B,\bZ)$, i.e. $\vv$ is a B-twisted algebraic class. 

Moreover, for any flat deformation of the polarised K3 surface $(X,H)$ with locally constant B-fields extending $B$ on $X$, $\vv$ also extends to a locally constant section, such that we get a twisted algebraic class on each fibre.
\end{lemma}

\begin{proof}
A simple computation shows the component of $\exp(-B)\cdot\vv$ in $H^2(X,\bQ)$ is $qh$, which is a $(1,1)$-class. Hence we conclude that $\vv \in H^*_{\mathrm{alg}}(X, B, \bZ)$ is an integral $B$-twisted algebraic class on $X$. 

Assume we have a deformation over $S$, such that for every $s \in S$, the K3 surface $X_s$ comes with an ample line bundle $H_s$ and a B-field $B_s$, which are both locally constant classes. Then the locally trivial extension of $\vv$ over each fibre $X_s$ is given by $\vv_s=(r, qh_s+rB_s, a) \in H^*(X_s,\bZ)$ where $h_s=c_1(H_s) \in H^2(X_s, \bZ)$. The same computation shows that it is a $B_s$-twisted algebraic class.
\end{proof}

The above lemma leads to the following definition.

\begin{definition}
\label{def:deformable-mukai}
A class $\vv \in H^*(X, \bZ)$ satisfying the assumption of Lemma \ref{lem:deformable-mukai} is called a \emph{deformable} $B$-twisted Mukai vector on $X$.
\end{definition}

The following lemma justifies the universality of this notion. 

\begin{lemma}
\label{lem:change-deform}
Let $(X,H)$ be a polarised K3 surface, $B$ a rational B-field, and $\vv$ a $B$-twisted Mukai vector with its degree zero component $r>0$. If $H$ is generic with respect to $\vv$, then the moduli space $M_{X,B,H}(\vv)$ of $B$-twisted $H$-semistable sheaves with Mukai vector $\vv$ is always isomorphic to a moduli space $M_{X,B,H'}(\vv')$ where $\vv'$ is deformable, and $H'$ is generic with respect to $\vv'$.
\end{lemma}

\begin{proof}
This is also classical (see, for instance, the proof of \cite[Theorem 6.2.5]{HuLe10a}). We write $\vv=(r,c,a)$ and $c_1(H)=h$. By Lemma \ref{lem:change-mukai}, we can replace the Mukai vector $\vv$ by $\vv'=\vv \cdot \exp(mh)$ for any $m\in \bZ$, without changing the moduli space. Moreover, $H$ is still generic with respect to $\vv'$. Note that $\vv'$ is $B$-twisted, therefore the degree two component of $\vv'\cdot \exp(-B)$ is a rational $(1,1)$-class. A simple calculation shows that this class is $c+rmh-rB$. When $m \gg 0$, $c+rmh-rB$ is an ample class and lies in the same chamber as $h$. We fix such an $m$ and write $h'$ for the primitive integral class on the ray generated by $c+rmh-rB$ in the ample cone. We denote the corresponding ample line bundle $H'$, and write $c+rmh-rB=qh'$ for some $q \in \bQ$. Then $\vv'=(r, qh'+rB, a')$ for some $a' \in \bZ$ is a deformable Mukai vector on the polarised K3 surface $(X, H')$. The construction guarantees that $H'$ lies in the interior of a chamber, and hence is generic.
\end{proof}

We recall the following fact from linear algebra, which will be used in the proof of Lemma \ref{lem:technical}.

\begin{lemma}
\label{lem:linear-algebra}
Let $V$ be a real vector space equipped with a (possibly degenerate) real-valued symmetric bilinear pairing, whose signature is $(n_+, n_-, n_0)$ (for the positive definite, negative definite, and isotropic parts respectively). Let $V'$ be a codimension one linear subspace of $V$ equipped with induced pairing, with signature $(n'_+, n'_-, n'_0)$. Then we have $n_+-1 \leq n'_+ \leq n_+$, $n_--1 \leq n'_- \leq n_-$, and $n_0-1 \leq n'_0 \leq n_0+1$. 

The same statement holds for rational vector spaces with rational-valued pairings. 
\end{lemma}

\begin{proof}
The proof is completely elementary by looking at the symmetric matrix representing the symmetric bilinear pairing. We leave it to the reader. 
\end{proof}

\begin{remark}
If the symmetric bilinear pairing is non-degenerate, i.e. $n_0=0$, then we will also write its signature as $(n_+,n_-)$ by abuse of notation.
\end{remark}

We are now ready to prove a technical lemma, which will be used to deal with the subtle potential issues of ampleness caused by $(-2)$-classes in the deformation, and the genericness of the polarisations. We point out that, despite of the words appearing in the statement, this lemma has nothing to do with ampleness of $h$ or $B$ being a B-field. We state it in this way just to indicate the situation in which we apply it. The proof of this lemma only contains elementary lattice theoretic arguments. 

\begin{lemma}
\label{lem:technical}
Let $h \in H^2(X, \bZ)$ be an ample class, and $B \in H^2(X, \bQ)$ be a B-field. Fix an arbitrary positive integer $N_0$. Then there exists $B' \in B + H^2(X, \bZ)$ such that $B'^2<0$, and for every non-zero class $g \in \spanning_\bQ \{ h, B' \} \cap H^2(X, \bZ) \cap h^\perp$, we have $g^2<-N_0$.
\end{lemma}

\begin{proof}
Without loss of generality, we can assume $h$ to be a primitive class. Otherwise we can replace $h$ by the primitive class in the ray generated by $h$, which is still an ample class. Moreover, let $n$ by the smallest positive integer such that $nB \in H^2(X, \bZ)$. Note that $n$ is in fact the order of the Brauer class represented by $B$. 

Since $h$ is primitive, the exact sequence of lattices
$$0 \to \bZ h \to H^2(X, \bZ) \to H^2(X, \bZ)/\bZ h \to 0$$
splits non-canonically. If we choose a lift of $H^2(X, \bZ)/\bZ h$ in the ambient lattice $H^2(X, \bZ)$, say $M$, then $H^2(X, \bZ)= \bZ h \oplus M$. We can write $nB = \alpha_Bh+\beta_Bm_B$ under this decomposition, where $\beta_B$ is a non-negative integer and $m_B \in M$ is a primitive class. Moreover, it is easy to see that, for any primitive class $p_M \in M$, the lattice $\spanning_\bQ \{ h, p_M \} \cap H^2(X, \bZ)$ has an integral basis given by $\{h, p_M\}$. 

We do the same thing for the second time. Since $m_B$ is primitive, the exact sequence of lattices
$$0 \to \bZ m_B \to M \to M/\bZ m_B \to 0$$
splits, and we can choose a lift $L$ of $M/\bZ m_B$ in $M$ and write $M=\bZ m_B \oplus L$. Again, for any primitive class $p_L \in L$, the lattice $\spanning_\bQ \{ m_B, p_L \} \cap M$ has an integral basis given by $\{m_B, p_L\}$. In particular, the divisibility of any integral linear combination $\gamma_1 m_B + \gamma_2 p_L$ is the highest common factor of $\gamma_1$ and $\gamma_2$. 

Now we take $K := \{ h, m_B \}^\perp \cap L$ be the sublattice of $H^2(X, \bZ)$ containing all lattice points in $L$ which are perpendicular to both $h$ and $m_B$ under the Poincar\'{e} pairing. In particular, these classes are perpendicular to $nB$. Note that $K \subset L$ has corank at most $2$, therefore $K \subset H^2(X, \bZ)$ has corank at most $4$. Since the Poincar\'{e} pairing on $H^2(X, \bZ)$ has signature $(3, 19)$ (see e.g. \cite[Proposition VIII.3.2]{BHPV}), we can use Lemma \ref{lem:linear-algebra} repeatedly to see that the restriction of the Poincar\'{e} pairing to $K$ has negative signature at least $15$. That is, $K$ contains a negative definite sublattice of rank at least $15$. Therefore, for any $N \gg 0$, we can find a primitive integral class $c \in K$ such that $c^2 < -N$. 

Now let us consider the rank two lattice $\Lambda_N = \spanning_\bQ \{ h, c+B \} \cap H^2(X, \bZ)$. We will show that for $N$ sufficiently large, it does not contain any $(-2)$-classes which are perpendicular to $h$. 

We first try to find an integral basis for $\Lambda_N$. Note that $nc+nB = \alpha_B h + ( nc + \beta_B m_B )$ where $nc + \beta_B m_B \in M$. Since $c \in L$ is primitive, the divisibility $k_N$ of the class $nc + \beta_B m_B$ is the highest common factor of $n$ and $\beta_B$. In particular, $1 \leq k_N \leq n$. We write $nc + \beta_B m_B = k_N p_N$ where $p_N$ is a primitive class in $M$. By the discussion above, this shows that $ \{ h, p_N \} $ is an integral basis for $\Lambda_N$. Therefore, we just need to show that for each pair of integers $(x,y)$ satisfying $h \cdot (xh+yp_N) =0$, we have $(xh+yp_N)^2 < -N_0$ for any given $N_0$.

We give an estimate of the intersection matrix of the lattice $\Lambda_N$ with integral basis $\{ h, p_N \}$. Since $h$ is a fixed class, $h^2$ is a fixed number. And we have
$$h \cdot p_N = h \cdot \frac{1}{k_N}(nc+\beta_Bm_B)=\frac{\beta_B}{k_N}(h \cdot m_B).$$
Since $1 \leq k_N \leq n$, we see that $\left| h \cdot p_N \right| \leq \beta_B(h \cdot m_B)$ where the right hand side is independent of $N$. Moreover
$$p_N^2=\frac{1}{k_N^2}(nc+\beta_B m_B)^2=\frac{1}{k_N^2}(n^2c^2+\beta_B^2 m_B^2).$$
Since $1 \leq k_N \leq n$ and $c^2 < -N$, we see that when $N \gg 0$ is large and positive, $p_N^2 \ll 0$ is large and negative. 

Now assume we have a non-zero vector $xh+yp_N \in \Lambda_N$ with $h \cdot (xh+yp_N) =0$. Then we can assume $y \neq 0$ since otherwise the vector is zero. Thus we have
\begin{align*}
(xh + yp_N)^2 &= yp_N \cdot (xh+yp_N) = y(x h\cdot p_N + y p_N^2) \\
&= y(x h\cdot p_N + y \frac{(h\cdot p_N)^2}{h^2} + y (p_N^2 - \frac{(h\cdot p_N)^2}{h^2}))\\
&= y^2 (p_N^2 - \frac{(h\cdot p_N)^2}{h^2}).
\end{align*}
From the above discussion, we see that $\frac{(h\cdot p_N)^2}{h^2}$ is bounded for all choices of $N$, while $p_N^2 \ll 0$ when $N \gg 0$. Since $y^2$ is a positive integer, the above computation shows that $(xh + yp_N)^2 \ll 0$, hence is smaller than any number $-N_0$ we started with.

Therefore, in the statement of the proposition, we can simply choose $B'=c+B$. Note that the choice of $c$ makes $c\cdot B=0$, hence we also have $B'^2=c^2+B^2 \ll 0$ for a sufficiently large choice of $N$. 
\end{proof}

The following proposition shows how to deform a twisted polarised K3 surface to an untwisted polarised K3 surface.

\begin{proposition}
\label{prop:deform-no-mukai}
Let $(X, \alpha)$ be a twisted projective K3 surface with an ample line bundle $H$, and $B \in H^2(X, \bQ)$ be a B-field lift of $\alpha$. We denote $h=c_1(H)$ and assume that $B^2 < 0$ and the lattice $\spanning_\bQ \{ h, B \} \cap H^2(X, \bZ) \cap h^\perp$ does not contain any $(-2)$-classes. Then there exists a family of twisted projective K3 surfaces $(\cX, \cH) \to S$ over an integral base $S$, with ample line bundle $H_s$ and locally constant B-field $B_s$ on the fibre $X_s$ of the family at each point $s \in S$, such that at $s_0 \in S$, we have $(X_{s_0}, H_{s_0}, B_{s_0}) = (X, H, B)$, and at $s_1 \in S$, we have $B_{s_1} \in H^{1,1}(X_{s_1}, \bQ)$. In particular, it represents a trivial Brauer class on $X_{s_1}$. 
\end{proposition}

\begin{proof}
We can choose $S$ to be a single point if $B$ is of type $(1,1)$. Otherwise, we study the period domain for the deformation of polarised K3 surfaces. Consider the lattice $\Lambda=\{h\}^{\perp}\subset H^2(X, \bZ)$, where $h=c_1(H)\in H^2(X, \bZ)$. The associated period domain is:
\begin{equation*}
\cD=\{ \sigma\in\bP(\Lambda\otimes \bC) : \sigma^2=0, \sigma\bar{\sigma}>0 \}.
\end{equation*}
The period domain of marked polarised K3 surfaces is obtained by removing hyperplanes of the form $\delta^\perp$ from $\cD$ for all $(-2)$-classes $\delta$. Since $\Lambda$ has signature $(2, 19)$, $\cD$ has two connected components, which can be identified by complex conjugation. We denote (the marking) $[H^{2,0}(X)]$ by $\sigma_0$ and observe that $\sigma_0\in \cD$. 

In order to deform $X$ so that the B-field becomes of type $(1,1)$, we consider $\Lambda'=\{h, B\}^{\perp}\subset H^2(X, \bZ)$ and set
\begin{equation*}
\cD'=\{ \sigma\in\bP(\Lambda'\otimes \bC) : \sigma^2=0, \sigma\bar{\sigma}>0 \}.
\end{equation*}
Notice that $\cD'$ is in fact a hyperplane section of $\cD$. The condition $B^2<0$ guarantees that $\Lambda'$ has signature $(2, 18)$. Therefore, $\cD'$ is non-empty and has two components, which are contained in the two components of $\cD$ respectively. The assumption also guarantees that $\cD'$ is not contained in the hyperplane $\delta^\perp$ for any $(-2)$-class $\delta \in \Lambda$. Therefore any period point in $\cD'$ away from these hyperplanes represents a deformation of $X$ on which $B$ has type $(1,1)$. Let $S \subset \cD$ be any integral curve joining such a period point and $\sigma_0$. It parametrises a family of projective K3 surfaces after removing at most locally finitely many closed points. 
\end{proof}

Next we deal with the issue of the genericness of polarisations. We remind the reader that for any (twisted) Mukai vector $\vv$ on a K3 surface with rank $r$, its \emph{discriminant} is defined as $\Delta(\vv)=\vv^2+2r^2$; see \cite[Section 3.4]{HuLe10a} or \cite[Definition 3.2.1.1]{Lie07}. 

\begin{lemma}
\label{lem:generic-criterion}
Let $(X,H)$ be a polarised K3 surface with a rational B-field $B$. Let $\vv$ be a $B$-twisted algebraic class with its degree zero component $r>0$. If there is no class $\xi \in H^{1,1}(X, \bZ)$ satisfying both $\xi \cdot H = 0$ and $-\frac{r^2}{4}\Delta{(\vv)} \leq \xi^2 <0$, then $H$ is generic with respect to $\vv$. 
\end{lemma}

\begin{proof}
The proof is contained in \cite[Theorem 4.C.3]{HuLe10a}, and so we only point out what change has to be made to adapt it in the twisted case. Using the notation there, $\xi=r\cdot c^B_1(F')-r'\cdot c^B_1(F)$ is a class in $H^{1,1}(X,B,\bZ)$. However, we realise that $c^B_1(F)-rB$ is the $(1,1)$-component of $\exp(-B)\vv$, hence is a class in $H^{1,1}(X,\bQ)$. Similarly, $c^B_1(F')-r'B$ is also a class in $H^{1,1}(X,\bQ)$. We note that a different way to write $\xi$ is $\xi=r\cdot (c^B_1(F')-r'B) - r'\cdot (c^B_1(F)-rB)$, hence $\xi$ is in fact also an untwisted $(1,1)$ class. So we only need to check classes in $H^{1,1}(X,\bZ)$ instead of $H^{1,1}(X,B,\bZ)$. On the other hand, we realise that the Bogomolov inequality also holds in twisted case by \cite[Proposition 3.2.3.13]{Lie07}. Therefore by the same argument as in the proof of \cite[Theorem 4.C.3]{HuLe10a}, if $H$ is not $\vv$-generic, then there is such a $\xi$ satisfying both $\xi \cdot H=0$ and $-\frac{r^2}{4}\Delta{(\vv)} \leq \xi^2 <0$. The proof is finished by contradiction. 
\end{proof}

The following proposition shows how to keep the genericness of polarisations with respect to a locally constant section of deformable Mukai vectors. 

\begin{proposition}
\label{prop:deform-with-mukai}
Assume we are in the situation of Proposition \ref{prop:deform-no-mukai}. We also assume that $\vv=(r,qh+rB,a)$ is a deformable Mukai vector with $\vv^2>0$, and $H$ is generic with respect to $\vv$. Furthermore suppose that for each non-zero class $g \in \spanning_\bQ\{h,B\} \cap H^2(X,\bZ) \cap h^\perp$ we have $g^2<-\frac{r^2}{4}\Delta(\vv)$. Then, in the family constructed in Proposition \ref{prop:deform-no-mukai}, $\vv$ extends to a locally constant family of deformable Mukai vectors over $S$ with its value $\vv_s$ at the point $s \in S$. Moreover, for every $s \in S$, $H_s$ is generic with respect to $\vv_s$. 
\end{proposition}

\begin{proof}
By Lemma \ref{lem:deformable-mukai}, it is clear that $\vv$ can always be extended to a locally constant family of deformable Mukai vectors as long as $h$ and $B$ extend. Therefore, we just need to show that, for a generic period point $\sigma \in \cD'$ as constructed in the proof of Proposition \ref{prop:deform-no-mukai}, the class $h$ is generic with respect to the class $\vv$ on the projective K3 surface corresponding to $\sigma$.

We use the criterion in Lemma \ref{lem:generic-criterion} to determine the genericness. That is, in the period domain $\cD$ constructed in the proof of Lemma \ref{lem:deformable-mukai}, we need to remove hyperplane sections of the form $\xi^\perp$ for every class $\xi$ satisfying both $\xi \cdot H = 0$ and $-\frac{r^2}{4}\Delta{(\vv)} \leq \xi^2 <0$, so that every period point in the remaining part corresponds to a projective K3 surface on which the ample line bundle corresponding to the class $h$ is generic with respect to the locally constant extension of $\vv$. 

By \cite[Corollary 4.2]{PeRa10a}, we conclude that these hyperplane sections to be removed from are locally finite in $\cD$. However, we also have to show that $\cD'$ is not contained in any of these hyperplanes, so that we can still find period points to make the B-field have type $(1,1)$ after removing these hyperplane sections. In fact, if $\cD' \subset \xi^\perp$ for some $\xi \in H^2(X, \bZ)$, then $\xi \in \spanning_\bQ\{h,B\} \cap H^2(X,\bZ)$. By assumption, if such a class satisfies further $\xi \cdot h=0$, then $g^2<-\frac{r^2}{4}\Delta(\vv)$. Hence $\xi^\perp$ is not one of the hyperplane sections to be removed. 
\end{proof}

\subsection{Local existence of relative twisted moduli spaces and compatibility}

In this subsection, we show that for a family of polarised K3 surfaces with locally constant B-fields and deformable Mukai vectors, we can always construct relative moduli spaces of twisted sheaves in an (\'{e}tale or analytic) neighbourhood of an arbitrary point. Moreover, this relative moduli space carries a (quasi-)universal family on its smooth locus. However, this construction would involve a choice of a uniform \v{C}ech 2-cocycle representation of the B-fields over that neighbourhood, hence there is in general no way to glue these (quasi-)universal families together. We will nevertheless show that on any fixed polarised K3 surface, for different choices of the B-field lift of the same Mukai vector, or for different choices of the \v{C}ech 2-cocycle representations of the same B-field, the moduli spaces are always non-canonically isomorphic to one another. As we will see later, this is already sufficient for the property of Mukai morphisms being isometries to hold or fail simultaneously over the entire base (if connected).

The materials presented in this subsection are probably well-known to experts. However we couldn't find any reference in which these results are explicitly written down. Hence we state these results with full proofs as below. The first result proves the local existence of relative moduli spaces of semistable sheaves.\footnote{We thank K\={o}ta Yoshioka for informing us that he was able to obtain a global relative twisted moduli space without universal families a long time ago, in the category of algebraic (or analytic) spaces by gluing along an \'{e}tale (or analytic) cover of the base.}

\begin{proposition}
\label{prop:local-deformation}
Let $(\cX, \cH) \to S$ be a family of projective K3 surface equipped with a locally trivial family of B-fields, denoted by $B_s$ on each fibre $X_s$. We also write $h_s=c_1(H_s) \in H^2(X_s, \bZ)$. Let $\vv_s=(r, qh_s+rB_s, a)\in H^*(X_s, \bZ)$ be a locally constant family of deformable Mukai vectors with $\vv_s^2>0$. Then for each point $s_0 \in S$, there is an (\'{e}tale or analytic) open neighbourhood $U_{s_0} \subset S$, over which there is a relative moduli space $\cM$, whose fibre $M_s$ at any point $s \in U_{s_0}$ is the moduli space of $B_s$-twisted $H_s$-semistable sheaves of $B_s$-twisted Mukai vector $\vv_s$, with a (quasi-)universal family on the smooth locus of $\cM$. 
\end{proposition}

\begin{proof}
We first show that, for every fixed $s_0 \in S$, over an open neighbourhood of $s_0$, the B-fields have a ``uniform" \v{C}ech 2-cocycle representation. We write $p: \cX \to S$ for the family, and start with finitely many (\'{e}tale or analytic) open subsets of the total space $\cX$, which cover the fibre $X_{s_0}$. We label them by $\{V_i\}$ for $i$ in a finite index set. The union $\cup_iV_i$ is an open subset of the total space $\cX$ containing the fibre $X_{s_0}$. Therefore $p(\cX \backslash \cup_iV_i)$ is a closed subset of $S$ missing the point $s_0$. We denote $S \backslash p(\cX \backslash \cup_iV_i)$ by $U_{s_0}$, then for every $s \in U_{s_0}$, $\{ V_i \cap X_s \}$ is a finite (\'{e}tale or analytic) cover of the fibre $X_s$. Using these open covers, we realise that the \v{C}ech complex
$$ \oplus_i\Gamma(V_i, \bZ) \to \oplus_{i,j}\Gamma(V_i \cap V_j, \bZ) \to \oplus_{i,j,k}\Gamma(V_i \cap V_j \cap V_k, \bZ) \to \cdots$$
computes the integral cohomology $H^*(X_s , \bZ)$ for all $s \in U_{s_0}$ simultaneously. In fact, its restriction on each fibre $X_s$ for $s \in U_{s_0}$ is precisely the \v{C}ech complex on that fibre. In other words, any \v{C}ech 2-cocycle representation of the B-field $B_{s_0}$ on a single fibre $X_{s_0}$ using the open cover $\{V_i \cap X_s\}$ can be canonically extended over $U_{s_0}$, so that the extension simultaneously represents the B-fields $B_s$ for all $s \in B_{s_0}$. 

From this point on, everything becomes classical. We can write down a moduli functor for families of twisted semistable sheaves on fibres of $p^{-1}(U_{s_0}) \to U_{s_0}$. Note that for each such family we have a global \v{C}ech 2-cocycle representation of the locally constant B-field. The GIT construction of the relative moduli space in the untwisted situation (see, for instance, \cite[Theorem 4.3.7]{HuLe10a}) applies in the twisted situation and gives a relative moduli space over $U_{s_0}$, whose fibre over each point $s \in U_{s_0}$ is the moduli space of $B_s$-twisted $H_s$-semistable sheaves of class $\vv_s$ over $X_s$. And a (quasi-)universal family exists on the smooth locus of the relative moduli space.
\end{proof}

\begin{remark}
For the construction of the associated family of Azumaya algebras and the connection with Simpson's stability, we refer the reader to \cite{Yos06} and \cite{Sim94}. Note that the relative moduli space constructed in Proposition \ref{prop:local-deformation} is in general not smooth over the parameter space. In particular, in the case of relative moduli spaces of O'Grady type, which is the most interesting case in this paper, the relative moduli space is not smooth over the parameter space.
\end{remark}

The following three lemmas handle the compatibility issues caused by change of \v{C}ech 2-cocycle representation of a fixed B-field, change of the B-field lift of a fixed Brauer class, and change of the Mukai vector by tensoring with multiples of the polarisation.

\begin{lemma}
\label{lem:change-cocycle}
Let $(X,H)$ be a polarised K3 surface, $B$ be a rational B-field on $X$, and $\vv$ be a $B$-twisted Mukai vector of O'Grady type. Let $B^1$ and $B^2$ be two different \v{C}ech 2-cocycle representations of the same B-field $B$, and denote the moduli space of $B^1$-twisted (or $B^2$-twisted) $H$-semistable sheaves with twisted Mukai vector $\vv$ by $M_{X,B^1,H}(\vv)$ (or $M_{X,B^2,H}(\vv)$). Then we can choose an isomorphism $f: M_{X,B^1,H}(\vv) \to M_{X,B^2,H}(\vv)$, such that the Mukai morphisms $\theta^{\mathrm{tr}}_{B^1,H}$ and $\theta^{\mathrm{tr}}_{B^2,H}$ are identified. More precisely, the following diagram commutes
\begin{equation}
\label{eqn:diagram1}
\xymatrix{
\vv^{\perp,\mathrm{tr}} \ar[rr]^-{\theta^{\mathrm{tr}}_{B^1,H}} \ar@{=}[d]  & & H^2(M_{X,B^1,H}(\vv), \bZ) \\
\vv^{\perp,\mathrm{tr}} \ar[rr]_-{\theta^{\mathrm{tr}}_{B^2,H}} & & H^2(M_{X,B^2,H}(\vv), \bZ) \ar[u]^{f^*}_{\cong}.
}
\end{equation}
In other words, the Mukai morphism is independent of the choice of the \v{C}ech 2-cocycle representation of the B-field $B$. In particular, $\theta^{\mathrm{tr}}_{B^1,H}$ is an isometry if and only if $\theta^{\mathrm{tr}}_{B^2,H}$ is an isometry.
\end{lemma}

\begin{proof}
This is probably well-known, but we nevertheless give a proof. We first show that there is an equivalence of the abelian categories $F: \textbf{Coh}(X,B^1) \to \textbf{Coh}(X,B^2)$. Without loss of generality, we can assume that the open cover underlying the two \v{C}ech 2-cocycle representations $B^1$ and $B^2$ are the same. Otherwise, we can always find a common refinement, and restriction is a canonical way to obtain transition functions for the refined open cover from $B^1$ (or $B^2$). Now we assume $\{ U_i \}$ is the open cover for both $B^1$ and $B^2$, then $B^1$ (or $B^2$) can be represented by $\{ B^1_{ijk} \in \Gamma(U_{ijk},\bQ) \}$ (or $\{ B^2_{ijk} \in \Gamma(U_{ijk},\bQ) \}$ respectively). Since they represent the same class in $H^2(X, \bQ)$, the difference $\{ B^2_{ijk}-B^1_{ijk} \in \Gamma(U_{ijk},\bQ) \}$ is a coboundary. That is, there exists a 1-cocycle $\{ \gamma_{ij} \in \Gamma(U_{ij},\bQ) \}$, such that $B^2_{ijk}-B^1_{ijk}=\gamma_{ij}+\gamma_{jk}+\gamma_{ki}$. 

With the above preparation, we can construct the functor $F$ in the following way: for every $E \in \textbf{Coh}(X,B^1)$ given by $\{ E_i, \varphi_{ij} \}$, where $\varphi_{ij} \in \Gamma(U_{ij}, \cO^*)$, we define $F(E)$ to be $\{ E_i, \varphi_{ij}\cdot \exp(\gamma_{ij}) \}$. A direct computation shows that $F(E) \in \textbf{Coh}(X, B^2)$. Note that this operation is also invertible and produces $F^{-1}$. Therefore $F: \textbf{Coh}(X,B^1) \to \textbf{Coh}(X,B^2)$ is an equivalence of categories (in fact in our case both compositions of $F$ and $F^{-1}$ are the identity functors). It is also easy to see that $F$ preserves the abelian structures on the two categories. 

We also observe that the functor $F$ preserves twisted Chern characters. For this we follow the notations and arguments in proof of \cite[Proposition 1.2]{HS05}. Note that the $B^1$-twisted Chern character of $E$ is defined to be the (untwisted) Chern character of (the $\cC^{\infty}$ sheaf) $E_{B^1}= \{ E_i, \varphi'_{ij}=\varphi_{ij} \cdot \exp(a_{ij}) \}$ for some $\cC^{\infty}$-coboundary representation $\{-a_{ij}\}$ of $B^1$. We realise that $\{-a_{ij}+\gamma_{ij}\}$ becomes a $\cC^{\infty}$-coboundary representation of $B^2$. Therefore, the $B^2$-twisted Chern character of $F(E)$ can be defined to be the (untwisted) Chern character of (the $\cC^{\infty}$ sheaf) $F(E)_{B^2}= \{ E_i, \varphi''_{ij} \}$, where $\varphi''_{ij}=\varphi_{ij} \cdot \exp(\gamma_{ij}) \cdot \exp(a_{ij}-\gamma_{ij}) = \varphi_{ij} \cdot \exp(a_{ij})$. Therefore $E_{B^1}$ and $F(E)_{B^2}$ are in fact the same sheaf, and the twisted Chern characters $\chern^{B^1}(E)=\chern^{B^2}(F(E))$. In other words, the functor $F$ preserves twisted Chern characters. 

Since $F$ preserves twisted Chern characters, it also preserves (twisted) Euler characteristics and (twisted) Hilbert polynomials, and therefore $H$-stability. It is also straightforward to see that the construction of $F$ can be carried out in families of $B^1$-twisted sheaves. In particular, $F$ induces a natural transformation between the moduli functors for $B^1$- and $B^2$-twisted $H$-semistable sheaves of Mukai vector $\vv$, and hence induces an isomorphism on the moduli spaces $f: M_{X,B^1,H}(\vv) \to M_{X,B^2,H}(\vv)$. Moreover, if $\cE$ is a (quasi-)universal family of $B^1$-twisted $H$-stable sheaves of Mukai vector $\vv$ on the smooth locus of the moduli space $M_{X,B^1,H}^{\mathrm{st}}(\vv)$, then $F(\cE)$ is a (quasi-)universal family of $B^2$-twisted $H$-stable sheaves of Mukai vector $\vv$ on the smooth locus of the moduli space $M_{X,B^2,H}^{\mathrm{st}}(\vv)$. The above argument shows that the two families $\cE$ and $F(\cE)$ have the same twisted Chern character. In particular, the cohomological integral functors $\Phi^{\cE}$ and $\Phi^{F(\cE)}$ given by the two families are the same. More precisely, we have the commutative diagram
\begin{equation}
\label{eqn:diagram2}
\xymatrix{
\vv^{\perp,\mathrm{tr}} \ar[rr]^-{\Phi^{\cE}} \ar@{=}[d]  & & H^2(M_{X,B^1,H}^{\mathrm{st}}(\vv), \bQ) \\
\vv^{\perp,\mathrm{tr}} \ar[rr]_-{\Phi^{F(\cE)}} & & H^2(M_{X,B^2,H}^{\mathrm{st}}(\vv), \bQ) \ar[u]^{f^*}_{\cong}.
}
\end{equation}
By the construction of the Mukai morphism in \cite[Section 3]{PeRa10a}, we know that $\Phi^{\cE}$ (resp. $\Phi^{F(\cE)}$) uniquely determines $\theta^{\mathrm{tr}}_{B^1,H}$ (resp. $\theta^{\mathrm{tr}}_{B^2,H}$), and the commutativity of \eqref{eqn:diagram2} implies the commutativity of \eqref{eqn:diagram1} (see also \cite[Lemma 3.11]{PeRa10a}). Therefore $\theta^{\mathrm{tr}}_{B^1,H}$ is an isomorphism of lattices if and only if $\theta^{\mathrm{tr}}_{B^2,H}$ is an isomorphism of lattices, and by \cite[Lemma 3.12]{PeRa10a}, $\theta^{\mathrm{tr}}_{B^1,H}$ is an isometry if and only if $\theta^{\mathrm{tr}}_{B^2,H}$ is an isometry. 
\end{proof}

\begin{lemma}
\label{lem:change-B}
Let $(X,H)$ be a polarised K3 surface, $B$ be a rational B-field on $X$, and $\vv$ be a $B$-twisted Mukai vector of O'Grady type. Pick any $c\in H^2(X, \bQ)$ which is a B-field lift of the zero Brauer class. We denote the moduli space of $B$-twisted (resp. $(B+c)$-twisted) $H$-semistable sheaves with twisted Mukai vector $\vv$ by $M_{X,B,H}(\vv)$ (resp. $M_{X,B+c,H}(\vv)$). Then we can choose an isomorphism $f: M_{X,B,H}(\vv) \to M_{X,B+c,H}(\vv )$, such that the Mukai morphisms $\theta^{\mathrm{tr}}_{B,H}$ and $\theta^{\mathrm{tr}}_{B+c,H}$ are identified. More precisely, the following diagram commutes
\begin{equation}
\label{eqn:diagram3}
\xymatrix{
\vv^{\perp,\mathrm{tr}} \ar[rr]^-{\theta^{\mathrm{tr}}_{B,H}} \ar@{=}[d]  & & H^2(M_{X,B,H}(\vv), \bZ) \\
\vv^{\perp,\mathrm{tr}} \ar[rr]_-{\theta^{\mathrm{tr}}_{B+c,H}} & & H^2(M_{X,B+c,H}(\vv), \bZ) \ar[u]^{f^*}_{\cong}.
}
\end{equation}
In other words, the Mukai morphism is independent of the choice of the B-field lift of the Brauer class up to a rational $(1,1)$ class. In particular, $\theta^{\mathrm{tr}}_{B,H}$ is an isometry if and only if $\theta^{\mathrm{tr}}_{B+c,H}$ is an isometry.
\end{lemma}

We point out that the condition $c$ is a B-field lift for the zero Brauer class contains many interesting special cases. The cases that we need are $c \in H^2(X, \bZ)$ or $c \in H^{1,1}(X,\bZ)$. 

\begin{proof}
We follow the same idea as in proof of Lemma \ref{lem:change-cocycle}. We fix an open cover $\{ U_i \}$ on which both $B$ and $c$ has a \v{C}ech 2-cocycle representation, say $B=\{ B_{ijk} \in \Gamma(U_{ijk}, \bQ) \}$ and $c=\{ c_{ijk} \in \Gamma(U_{ijk}, \bQ) \}$. Since $c$ lifts the zero Brauer class, we have $\exp(c) =0 \in H^2(X, \cO^*)$. In other words, we can find a 1-cocycle $\{ \gamma_{ij} \in \Gamma(U_{ij}, \cO^*) \}$, such that $\exp(c_{ijk})=\gamma_{ij}\cdot\gamma_{jk}\cdot\gamma_{ki}$. We define the functor $F: \textbf{Coh}(X,B) \to \textbf{Coh}(X,B+c)$ as follows: for every $E \in \textbf{Coh}(X,B)$ which can be represented by $E=\{ E_i, \varphi_{ij} \}$, we define $F(E)=\{ E_i, \varphi_{ij}\cdot \gamma_{ij} \}$. It is easy to check that $F(E) \in \textbf{Coh}(X, B+c)$. It is easy to see that $F$ is an equivalence of categories, and preserves the abelian structure. We also observe that the same construction can be done in families.

However, similar to the computation in the proof of Lemma \ref{lem:change-cocycle}, we observe that $\chern^B(E)=\chern^{B+c}(F(E))$. That is, $F$ preserves twisted Chern character and hence preserves the twisted Euler characteristic and $H$-stability. This implies that $F$ induces an isomorphism of moduli spaces $f: M_{X,B,H}(\vv) \to M_{X,B+c,H}(\vv)$. Moreover, if $\cE$ is a (quasi-)universal family over $M_{X,B,H}^{\mathrm{st}}(\vv)$, then $F(\cE)$ is a (quasi-)universal family over $M_{X,B+c,H}^{\mathrm{st}}(\vv)$, with $\chern^B(\cE)=\chern^{B+c}(F(\cE))$. In particular, the cohomological integral functors $\Phi^\cE$ and $\Phi^{F(\cE)}$ given by the two families are the same. Hence the following diagram commutes
\begin{equation}
\label{eqn:diagram4}
\xymatrix{
\vv^{\perp,\mathrm{tr}} \ar[rr]^-{\Phi^{\cE}} \ar@{=}[d]  & & H^2(M_{X,B,H}^{\mathrm{st}}(\vv), \bQ) \\
\vv^{\perp,\mathrm{tr}} \ar[rr]_-{\Phi^{F(\cE)}} & & H^2(M_{X,B+c,H}^{\mathrm{st}}(\vv), \bQ) \ar[u]^{f^*}_{\cong}.
}
\end{equation}
Using a similar reasoning as in the proof of Lemma \ref{lem:change-cocycle}, we conclude that the commutativity of \eqref{eqn:diagram4} implies the commutativity of \eqref{eqn:diagram3}. In particular, $\theta^{\mathrm{tr}}_{B,H}$ is an isometry if and only if $\theta^{\mathrm{tr}}_{B+c,H}$ is an isometry.
\end{proof}

\begin{lemma}
\label{lem:change-mukai}
Let $(X,H)$ be a polarised K3 surface, $B$ be a rational B-field on $X$, and $\vv$ be a $B$-twisted Mukai vector of O'Grady type. We write $h=c_1(H)$. For any $m\in \bZ$, there is a canonical isomorphism $f: M_{X,B,H}(\vv) \to M_{X,B,H}(\vv \cdot \exp(mh))$, such that the following diagram commutes
\begin{equation}
\label{eqn:diagram5}
\xymatrix{
\vv^{\perp,\mathrm{tr}} \ar[rr]^-{\theta^{\mathrm{tr}}_{B,H,\vv}} \ar[d]_{(-) \cdot \exp(mh)}^{\cong}  & & H^2(M_{X,B,H}(\vv), \bZ) \\
(\vv \cdot \exp(mh))^{\perp,\mathrm{tr}} \ar[rr]_-{\theta^{\mathrm{tr}}_{B,H,\vv\cdot\exp(mh)}} & & H^2(M_{X,B,H}(\vv \cdot \exp(mh)), \bZ) \ar[u]^{f^*}_{\cong}.
}
\end{equation}
In particular, $\theta^{\mathrm{tr}}_{B,H,\vv}$ is an isometry if and only if $\theta^{\mathrm{tr}}_{B,H,\vv\cdot\exp(mh)}$ is an isometry.
\end{lemma}

\begin{proof}
This is classical. We state a proof using the language similar to the above lemmas. We observe that $F:\textbf{Coh}(X,B) \to \textbf{Coh}(X,B)$ with the assignment $F(E)=E\otimes H^{\otimes m}$ is an equivalence of abelian categories and is well defined for families. Moreover the functor $F$ preserves $H$-stability. Therefore it induces an isomorphism of moduli spaces $f: M_{X,B,H}(\vv) \to M_{X,B,H}(\vv \cdot \exp(mh))$, and takes a (quasi-)universal family $\cE$ on $M_{X,B,H}^{\mathrm{st}}(\vv)$ to a (quasi-)universal family $F(\cE)$ on $M_{X,B,H}^{\mathrm{st}}(\vv \cdot \exp(mh))$. We can also write down a diagram as follows.
\begin{equation}
\label{eqn:diagram6}
\xymatrix{
\vv^{\perp,\mathrm{tr}} \ar[rr]^-{\Phi^\cE} \ar[d]_{(-) \cdot \exp(mh)}^{\cong}  & & H^2(M_{X,B,H}^{\mathrm{st}}(\vv), \bQ) \\
(\vv \cdot \exp(mh))^{\perp,\mathrm{tr}} \ar[rr]_-{\Phi^{F(\cE)}} & & H^2(M_{X,B,H}^{\mathrm{st}}(\vv \cdot \exp(mh)), \bQ) \ar[u]^{f^*}_{\cong}.
}
\end{equation}
We just point out that the extra factor $\exp(mh)$ in the twisted Mukai vector and the extra factor $\exp(mh)$ coming from the twisted Chern character of the (quasi-)universal family cancel out, thanks to the dual operation we have in the Mukai version of cohomological integral functor (see \cite[Definition 6.1.12]{HuLe10a}, or \cite[Section 1.2]{Yos01}, or \cite[Section 3.2]{PeRa10a}). This is why the commutativity of the diagram holds. By the same reasoning as in Lemma \ref{lem:change-cocycle}, this implies the commutativity of \eqref{eqn:diagram5}. 
\end{proof}

\subsection{Main proof and consequence}

We are now ready to state the proof of Theorem \ref{thm:hodge-isom}. The most difficult part in the proof is to show that the Mukai morphism is an isometry. We outline the idea used in this part of proof. The principal idea has already been established in \cite{PeRa10a} where the case of Gieseker moduli spaces is proven. Perego and Rapagnetta proved their result by showing that one particular moduli space has this property, and all the other Gieseker moduli spaces can be related to it by Fourier-Mukai transforms and deformations, under which this property is stable. Our approach for any Bridgeland moduli space comes in two steps. We first use a Fourier-Mukai transform constructed in \cite[Lemma 7.3]{BM12} to relate such a Bridgeland moduli space to a twisted Gieseker moduli space and then use the deformations constructed in Propositions \ref{prop:deform-no-mukai} and \ref{prop:deform-with-mukai} to relate the twisted Gieseker moduli space to an untwisted Gieseker moduli space. Since the property of the Mukai morphism being an isometry is stable under both operations, and is already true for the untwisted Gieseker moduli space, we conclude it is true for the Bridgeland moduli space. A complete proof goes as follows.

\begin{proof}[Proof of Theorem \ref{thm:hodge-isom}]
We observe that the statement \eqref{item:hodge1} follows from the very general result \cite[Lemma 3.1]{PeRa10a}. The pure weight two Hodge structure on $H^2(M,\bZ)$ is described in \cite[Definition 3.4]{PeRa10a} and the induced quadratic form is described in \cite[Definition 3.5]{PeRa10a}. 

From now on we focus on statement \eqref{item:hodge2}. We first show that the Mukai morphism is always well-defined. Note that using the (quasi-)universal family on the stable locus $M^{\mathrm{st}}$ in $M$, there is always a well-defined map $\theta^{\mathrm{st}}_{\vv}: \vv^{\perp,\mathrm{tr}} \to H^2(M^{\mathrm{st}},\bQ)$. It is just the problem whether this map can be lifted to $H^2(M,\bZ)$. 

If $M$ is the moduli space of $B$-twisted $H$-semistable sheaves with Mukai vector $\vv$, where $B$ is a rational B-field, $\vv$ is a Mukai vector of O'Grady type, and $H$ is generic with respect to $\vv$, then the moduli space $M$ admits a GIT construction, and the construction in \cite[Section 3.2]{PeRa10a} can be applied without change. Nevertheless, we point out that the argument there includes a step which requires the local existence of relative moduli spaces, which is guaranteed by Lemma \ref{lem:change-deform} and Proposition \ref{prop:local-deformation}.

To show the general case, we use \cite[Lemma 7.3]{BM12} and the discussion thereafter. We can always find a derived equivalence in the form of a Fourier-Mukai transform $\Phi: \cD(X) \to \cD(Y,B)$ for a $B$-twisted projective K3 surface $Y$, such that the image $\Phi(\sigma)$ of the generic stability condition $\sigma$ lies in the Gieseker chamber of $\stab(Y,B)$, hence is represented by an ample line bundle $H$, which is generic with respect to the Mukai vector $\Phi(\vv)$. In particular, $\Phi$ induces an isomorphism of the moduli spaces $\Phi: M_{X,\sigma}(\vv) \to M_{Y,B,H}(-\Phi(\vv))$. Moreover, every $H$-semistable object representing a point in $M_{Y,B,H}(\Phi(\vv))$ is locally free, hence $\Phi(\vv)$ has positive rank. We consider the following diagram
\begin{equation*}
\xymatrix{
\vv^{\perp,\mathrm{tr}} \ar[rr]^{\Phi}_{\cong} \ar[d]_{\theta^{\mathrm{st}}_{\vv}} \ar@/_25mm/@{-->}[dd]_{\theta^{\mathrm{tr}}_{\vv}} & & (\Phi(\vv))^{\perp,\mathrm{tr}} \ar[d]^{\theta^{\mathrm{st}}_{\Phi(\vv)}} \ar@/^25mm/[dd]^{\theta^{\mathrm{tr}}_{\Phi(\vv)}} \\
H^2(M_{X,\sigma}^{\mathrm{st}}(\vv),\bQ) \ar@{=}[rr]^{\Phi^*} & & H^2(M_{Y,B,H}^{\mathrm{st}}(\Phi(\vv)),\bQ)   \\
H^2(M_{X,\sigma}(\vv),\bZ) \ar[u]^{i^*} \ar@{=}[rr]^{\Phi^*} & &  H^2(M_{Y,B,H}(\Phi(\vv)),\bZ) \ar[u]_{i^*}
}
\end{equation*}
By \cite[Proposition 2.4]{Yos01}\footnote{Although this proposition is stated for moduli spaces of sheaves, the calculation in its proof is purely cohomological, which works for moduli spaces of complexes literally.}, we obtain that the upper square commutes, and the first row is an isometry. It is also easy to see the bottom square also commutes, since both vertical maps in it are functorial and both horizontal maps are induced by the isomorphism of the moduli spaces. By the above discussion, there is a lift $\theta^{\mathrm{tr}}_{\Phi(\vv)}$ for the maps in the right column. Therefore a lift $\theta^{\mathrm{tr}}_{\vv}$ for the left column also exists. This shows that the Mukai morphism $\theta^{\mathrm{tr}}_{\vv}$ is always well-defined for any Bridgeland moduli space of O'Grady type.

Now we show that $\theta^{\mathrm{tr}}_{\vv}$ is an isometry. From the above diagram and \cite[Lemma 3.12]{PeRa10a}, it suffices to show it only for every twisted Gieseker moduli space $M_{X,B,H}(\vv)$ of O'Grady type. By Lemma \ref{lem:change-deform}, we can change $\vv$ and $H$ if necessary, so that $\vv$ is deformable. By Lemma \ref{lem:technical}, we can add an integral class to $B$ if necessary, so that $B^2<0$, and for every non-zero class $g \in \spanning_\bQ \{ h, B \} \cap H^2(X, \bZ) \cap h^\perp$, we have $g^2<\min\{-2,-\frac{r^2}{4}\Delta{(\vv)}\}$, where $h=c_1(H)$ and $r$ is the degree zero component of $\vv$. By the above discussion, $r$ is a positive integer. By Lemmas \ref{lem:change-cocycle}, \ref{lem:change-B}, and \ref{lem:change-mukai}, we observe that the above adjustments do not affect the property that $\theta^{\mathrm{tr}}_{\vv}$ is an isometry. That is, it suffices to prove the property of $\theta^{\mathrm{tr}}_{\vv}$ being an isometry under these additional assumptions.

The key step here, is to deform such a moduli space of twisted sheaves to an untwisted moduli space. By Proposition \ref{prop:deform-no-mukai}, we can find a family of polarised K3 surfaces $(\cX,\cH) \to S$ over an integral base $S$, with fibre $(X_s,H_s)$ over any point $s \in S$. Moreover, the family comes equipped with a locally constant family of B-fields with $B_s \in H^2(X_s, \bQ)$ and such that $(X_{s_0},H_{s_0},B_{s_0})=(X,H,B)$ and $(X_{s_1},H_{s_1},B_{s_1})$ satisfies $B_{s_1} \in H^{1,1}(X_{s_1},\bQ)$ for some $s_0, s_1 \in S$. By Lemma \ref{lem:deformable-mukai}, the Mukai vector $\vv$ on the fibre over $s_0$ extends to a locally constant section over $S$, such that its value $\vv_s$ at every point $s\in S$ is still deformable, hence $\vv_s \in H^*_{\mathrm{alg}}(X_s, B_s, \bZ)$. By Proposition \ref{prop:deform-with-mukai}, we can further assume that $H_s$ is generic with respect to $\vv_s$ for every point $s \in S$. 

We now define $T:= \{ s \in S : \theta^{\mathrm{tr}}_{B_s,H_s,\vv} \text{ is an isometry} \}$ and claim that $T$ is both open and closed in $S$. To show it is open, we consider a point $t_0 \in T$. By Proposition \ref{prop:local-deformation}, we can find an (\'{e}tale or analytic) open subset $U_{t_0} \subset S$ containing $t_0$, such that there is a relative moduli space $\cM$ over $U_{t_0}$, whose fibre at each $s \in U_{t_0}$ is the moduli space $M_{X,B_s,H_s}(\vv)$. As observed in \cite[Proof of Theorem 1.7]{PeRa10a}, the isomorphisms of lattices and integral bilinear forms are both discrete properties and therefore have to remain constant in families. Since the isometry holds at the point $t_0$, it has to hold for every point $s \in U_{t_0}$, hence $U_{t_0} \subset T$. Since $t_0$ is an arbitrary point of $T$, we conclude that $T$ is open. To show $T$ is closed, we take an arbitrary point $t_0 \in (S \backslash T)$. Then, again by Proposition \ref{prop:local-deformation}, a relative moduli space exists over an open neighbourhood $U_{t_0}$ of $t_0$. Since $\theta_{\vv}^{\mathrm{tr}}$ fails to be an isometry at $t_0$, it must fail to be an isometry for every point of $U_{t_0}$. Hence $U_{t_0} \subset (S \backslash T)$ and we conclude that $T$ is closed. 

Since $S$ is connected, we either have $T=S$ or $T= \emptyset$. However, by the construction of $S$, there is a point $s_1 \in S$ at which $B_{s_1} \in H^{1,1}(X_s,\bQ)$ and hence represents the zero Brauer class. It has been proved in \cite[Theorem 1.7]{PeRa10a} that the Mukai morphism is an isometry for untwisted moduli spaces of O'Grady type. By Lemma \ref{lem:change-B}, this implies that Mukai morphism $\theta^{\mathrm{tr}}_{B_{s_1},H_{s_1}}$ is an isometry. Hence we must have $T=S$. In particular, the Mukai morphism $\theta^{\mathrm{tr}}_{B_{s_0},H_{s_0}}$ is an isometry, which is the arbitrary twisted Gieseker moduli space of O'Grady type we started with. This finishes the proof that $\theta^{\mathrm{tr}}_\sigma$ is an isometry.

Finally, to show that $\theta^{\mathrm{tr}}_\sigma$ preserves Hodge structure we observe that step 4 in \cite[Proof of Theorem 1.7]{PeRa10a} can be applied without change in our situation. And we are done.
\end{proof}

\begin{remark}
After writing this proof down, it was pointed out to us that one can use an observation of Yoshioka to reduce all of the previous arguments to the untwisted case. Indeed, let $\Phi:\cD(X) \to \cD(Y,B)$ be an equivalence inducing an isomorphism $M_{X,\sigma}(\vv) \to M_{Y,B,H}(-\Phi(\vv))$ as above. Set $\ww_0 := \Phi((1,0,1))^\vee$ and $\ww_1 := \Phi((0,0,1))^\vee$, where $\vee$ means the dual, and observe that for a general $H'$ which is sufficiently close to $H$, we have $Z := M_{Y,-B,H'}(\ww_1)\simeq X$ which is an untwisted K3 surface, since $\langle \ww_0,\ww_1\rangle = -1$. Now, if we let $\cF$ be a universal family and $\Psi : \cD(Y,B) \to \cD(Z)$ be a twisted Fourier-Mukai transform such that $\Psi(\cO_y) = \cF_{Z\times\{y\}}$ then we can apply \cite[Theorem 1.7 and Remark 1.6]{Yos06} to $E(nH')$ for $E \in M_{Y,B,H'}(-\Phi(\vv))$ and $n \gg 0$ to reduce to the untwisted case.
\end{remark}

In fact, the above proof yields the following interesting consequence.

\begin{corollary}
\label{cor:deformation-equivalent}
Let $\vv$ be a Mukai vector of O'Grady type and $\sigma \in \stab(X)$ be any generic stability condition with respect to $\vv$. Then the symplectic resolution of $M_\sigma(\vv)$ is deformation equivalent to the irreducible holomorphic symplectic manifold constructed by O'Grady in \cite{OG99a}.
\end{corollary}

\begin{proof}
We use the same Fourier-Mukai transform and deformation argument as in the proof of Theorem \ref{thm:hodge-isom}. Applying \cite[Lemma 7.3]{BM12} we obtain an isomorphism $\Phi: M_{X,\sigma}(\vv) \to M_{Y,B,H}(-\Phi(\vv))$. In particular, their symplectic resolutions are isomorphic as well. By Propositions \ref{prop:deform-no-mukai} and \ref{prop:deform-with-mukai}, we can deform the underlying twisted polarised K3 surface to an untwisted K3 surface, and the moduli space of twisted semistable sheaves deforms along with the underlying K3 surface locally near every point. Now we can apply \cite[Proposition 2.17]{PeRa10a} and conclude that the fibrewise resolutions of this equisingular family of moduli spaces are also deformation equivalent. Although these families of singular moduli spaces do not necessarily glue into a global family due to the relevance of B-field and its \v{C}ech 2-cocycle representation, the singular moduli spaces in different local families over the same twisted K3 surface are always isomorphic due to Lemmas \ref{lem:change-cocycle} and \ref{lem:change-B}. Therefore, they have isomorphic resolutions. Thus, we can conclude that the symplectic resolution of $M_\sigma(\vv)$ is deformation equivalent to that of a Gieseker moduli space of untwisted sheaves, which, by \cite[Theorem 1.6]{PeRa10a}, is deformation equivalent to the holomorphic symplectic manifold constructed by O'Grady in \cite{OG99a}. 
\end{proof}

\begin{remark}
We have learned from \cite[Theorem 6.2]{KaLeSo06a} and \cite[Theorem 1.6]{PeRa10a} that the only deformation type of irreducible holomorphic symplectic manifold that we can obtain by resolving singular Gieseker moduli spaces of semistable sheaves on K3 surfaces is the one constructed by O'Grady in \cite{OG99a}. Our Corollary \ref{cor:deformation-equivalent} shows that, even if we enlarge our scope to resolutions of moduli spaces of Bridgeland semistable objects on K3 surfaces, we still only get the same deformation type. 
\end{remark}

\begin{remark}
There is another more straightforward proof of Corollary \ref{cor:deformation-equivalent}, which does not need to go through twisted Gieseker moduli spaces and their deformations. However, it relies on a result which we will prove later, and a very deep theorem of Huybrechts. In fact, by Theorem \ref{thm:identify-cones}, we know that $M_\sigma(\vv)$ is birationally equivalent to a Gieseker moduli space $M_H(\vv)$ for some ample line bundle $H$ which is generic with respect to $\vv$. Therefore, their symplectic resolutions are also birationally equivalent. By \cite[Theorem 2.5]{Huy03}, these two symplectic resolutions are deformation equivalent. By \cite[Theorem 1.6]{PeRa10a}, they are further deformation equivalent to the irreducible holomorphic symplectic manifold constructed by O'Grady in \cite{OG99a}. We remind the reader that this shorter proof does not result in a circular argument because the proof of Theorem \ref{thm:identify-cones} does not use Corollary \ref{cor:deformation-equivalent}.
\end{remark}

\section{The Local \bmm Map}\label{localbmmap}

The most essential ingredient in \cite{BM12, BM13} is a linearisation map from $\stab(X)$ to N\'{e}ron-Severi group of the moduli space for any primitive Mukai vector $\vv \in H^*_{\mathrm{alg}}(X, \bZ)$. In this paper we refer to it as the \emph{\bmm map}. Note that the construction of \bmm map depends on the moduli space, therefore one can define a \bmm map for each chamber in $\stab(X)$, which we will call a \emph{local \bmm map}. And we are primarily only interested in the restriction of the map in that chamber (or its closure). These local \bmm maps are already sufficient for the purpose of proving the projectivity of Bridgeland moduli spaces as in \cite{BM12}. However, in \cite{BM13}, the authors managed to glue the maps defined on various chambers together and obtained a \emph{global \bmm map}, which reveals nicely the relation between wall crossings on the stability manifold and birational geometry of the Bridgeland moduli spaces. 

In this section, we will follow \cite{BM12} and generalise the definition of the local \bmm map to any non-primitive Mukai vector $\vv$ of O'Grady type, which relies greatly on the isomorphism $\theta_\sigma$ proved in Corollary \ref{cor:algebraic-mukai}\eqref{item:am3}. By using the local \bmm map, we will follow the approach in \cite{BM12} to show the positivity of some determinant line bundles on the singular moduli spaces of O'Grady type. We will discuss the global \bmm map in a later section.

\subsection{Construction of the local \bmm map}

The local \bmm map was first defined in \cite[Section 3 and 4]{BM12}, and plays a vital role in \cite{BM12,BM13}. This map establishes a bridge between the two central spaces in question: the stability manifold of the K3 surface, and the N\'eron-Severi group of the moduli space. 

There are mainly two different descriptions of the map. The first description, as defined in \cite[Section 3]{BM12}, offers the perfect point of view to understand the positivity theorem \cite[Theorem 1.1]{BM12}. While the second description, defined in \cite[Section 4]{BM12}, was used more often thereafter (for instance in \cite[Section 10]{BM13}). The two descriptions are equivalent by \cite[Proposition 4.4]{BM12}. 

Here we briefly recall the second description of the local \bmm map, which emphasises the role of the algebraic Mukai morphism $\theta_\cC$ from Corollary \ref{cor:algebraic-mukai}, and is more convenient to use for our purposes. Although it could be described in a more general situation, we only restrict ourselves to the case of non-primitive Mukai vectors $\vv$ of O'Grady type. In short, fix such a Mukai vector $\vv$ and a chamber $\cC \in \stab(X)$. Then the composition of the following three maps is called the local \bmm map for the chamber $\cC$ with respect to the Mukai vector $\vv$ and is denoted by $\ell_\cC$.

\begin{equation*}
\stab(X) \xra{\cZ} H^*_{\mathrm{alg}}(X, \bZ)\otimes \bC \xra{\cI} \vv^\perp \xra{\theta_\cC} \ns(M).
\end{equation*}

We briefly describe each map in the above composition. The first map is simply a forgetful map. For any $\sigma =(Z, \cP) \in \stab(X)$, with the central charge $Z(-)=(\Omega, -)$, we have $\cZ(\sigma)=\Omega \in H^*_{\mathrm{alg}}(X, \bZ)\otimes \bC$ which only remembers the central charge. An important feature is that $\cZ$ is a covering map on an open subset of the target, which was proved in \cite[Section 8]{Bri08} (see also \cite[Theorem 2.10]{BM13}). Note that $\cZ$ does not depend on $\vv$ or $\cC$. 

The second map forgets even more. For any $\Omega \in H^*_{\mathrm{alg}}(X, \bZ)\otimes \bC$, we have $\cI(\Omega)=\im \frac{\Omega}{-(\Omega, \vv)}$. In particular, when $(\Omega, \vv)=-1$, this map simply takes the imaginary part of $\Omega$. This map does not depend on $\cC$ either. 

Finally, the third map $\theta_\cC$ is the algebraic Mukai morphism described in Corollary \ref{cor:algebraic-mukai}. In particular, it is an isometry. Note that it is the only map among the three which depends on the choice of the chamber $\cC$. 

We fix some notations for later convenience (and to keep consistent with the notations in \cite{BM12,BM13}). We write $w_\sigma := (\cI \circ \cZ) (\sigma)$ for the image of any $\sigma \in \stab(X)$ under the composition of first two maps. The composition of all three maps is denoted by $\ell_\cC = \theta_\cC \circ \cI \circ \cZ$. Although $\ell_\cC$ is defined for the whole $\stab(X)$, we are mainly interested in its behaviour on the chamber $\cC$ itself (or rather, the closure of $\cC$). When the chamber $\cC$ is clear from the context, for any $\sigma$ in the interior of $\cC$, we also write $\ell_\sigma := \ell_\cC(\sigma)$, hence $\theta_\cC(w_\sigma)=\ell_\sigma$. And similarly, for any $\sigma_0$ on the boundary of $\cC$, we also write $\ell_{\sigma_0} := \ell_\cC(\sigma_0)$. 

\subsection{Positivity via the local \bmm map}

The projectivity of moduli spaces is one of the main results in \cite{BM12}. In \cite[Corollary 7.5]{BM12}, it is proved that $\ell_\sigma$ is ample for a generic $\sigma$ in the case of primitive Mukai vectors. However, for non-primitive Mukai vectors, the authors used a rather indirect approach to prove the projectivity, due to the lack of Yoshioka's theorem \cite[Theorem 6.10]{BM12} in such cases. As a result, it is proved that $\ell_\sigma$ is ample for $\sigma$ in a dense subset of $\cC$. However, our Theorem \ref{thm:hodge-isom} allows us to mimic the proof of \cite[Corollary 7.5]{BM12} and show the ampleness of $\ell_\sigma$ for all $\sigma \in \cC$.

\begin{proposition}
\label{prop:ampleness}
Let $\vv \in H^*_{\mathrm{alg}}(X,\bZ)$ be a Mukai vector of O'Grady type and $\cC \subset \stab(X)$ an open chamber with respect to $\vv$. Then the image $\ell_\sigma$ under the local \bmm map defined by the chamber $\cC$ is ample for all $\sigma \in \cC$.
\end{proposition}

\begin{proof}
As in Theorem \ref{thm:hodge-isom}, we write $\pi: \tilde{M} \to M$ for the symplectic resolution of $M$. By \cite[Theorem 1.1]{BM12}, we know that $\ell_\sigma$ is nef on $M$ and therefore its pullback $\pi^*\ell_\sigma$ is nef on $\tilde{M}$. Moreover, by Corollary \ref{cor:algebraic-mukai}, we have $\tilde{q}(\pi^*\ell_\sigma) = q(\ell_\sigma) = w_\sigma^2>0$, where the last inequality follows from \cite[Theorem 1.1]{Bri08}. Now by the Beauville-Fujiki relation \cite[Proposition 23.14]{GHJ03}, we know that the top self-intersection number of $\pi^*\ell_\sigma$ on $\tilde{M}$ is positive. By the bigness criterion \cite[Proposition 2.61]{KoMo98a}, we see that $\pi^*\ell_\sigma$ is big and nef on $\tilde{M}$. Furthermore, the base point free theorem \cite[Theorem 3.3]{KoMo98a} tells us that $\pi^*\ell_\sigma^{\otimes b}$ is globally generated for $b \gg 0$. 

Now we reduce  all the above statements from $\tilde{M}$ to $M$. Since $M$ has rational singularities, we have $H^0(\tilde{M}, \pi^*\ell_\sigma^{\otimes b})=H^0(M, \ell_\sigma^{\otimes b})$ for any $b$, i.e. the global sections of $\pi^*\ell_\sigma^{\otimes b}$ are precisely those obtained by pulling back global sections of $\ell_\sigma^{\otimes b}$ along $\pi$. Hence we conclude that $\ell_\sigma^{\otimes b}$ is also globally generated. 

Finally, by \cite[Theorem 1.1]{BM12} we see that $\ell_\sigma^{\otimes b}$ is positive on all curves in $M$. But any globally generated line bundle that is positive on all curves is ample; see \cite[Proposition 6.3]{Hu99a}. 
\end{proof}

By a similar argument, we can immediately deduce a weaker property for stability conditions on the wall. The following proposition shows that in the O'Grady situation, $\ell_{\sigma_0}$ for any $\sigma_0$ on the boundary of $\cC$ is nef, big, and semiample. Recall that a line bundle is semiample if a certain multiple of it is base point free. In particular, this result shows that the contraction morphisms $\pi_\pm$ in \cite[Theorem 1.4]{BM12} are still well-defined in the O'Grady situation and are birational morphisms. We will generalise this result to arbitrary Mukai vectors in Proposition \ref{prop:semiampleness} from a completely different point of view. 

\begin{proposition}
\label{old-semiampleness}
Let $\vv \in H^*_{\mathrm{alg}}(X,\bZ)$ be a Mukai vector of O'Grady type and $\cC \subset \stab(X)$ an open chamber with respect to $\vv$. For any stability condition $\sigma_0$ in the boundary of the chamber $\cC$, its image $\ell_{\sigma_0}$ under the local \bmm map defined by the chamber $\cC$ is big, nef, and semiample on $M_{\cC}(\vv)$.

In particular, $\ell_{\sigma_0}$ induces a birational morphism which contracts curves in $M_{\cC}(\vv)$ parametrising S-equivalent objects with respect to the stability condition $\sigma_0$. 
\end{proposition}
\begin{proof}
In fact, the first two paragraphs in the proof of Proposition \ref{prop:ampleness} also work here without any change, which prove the first statement. The second statement is an immediate consequence of \cite[Theorem 1.1]{BM12}.
\end{proof}

\section{Classification of Walls: Results}\label{classificationresults}

In this section, we study wall crossing on the stability manifold for non-primitive Mukai vectors. More precisely, we prove that the criteria in \cite[Theorem 5.7]{BM13} still gives a complete classification of walls in the O'Grady situation. Moreover, we will prove that \cite[Theorem 1.1]{BM13} also holds in the O'Grady situation, which allows us to glue the local \bmm maps and study the birational geometry of the singular moduli spaces via wall crossings. Due to the technicality in the proofs, we will only state our results in this section, while leaving all proofs for next section. 

\subsection{Hyperbolic lattice associated to a wall}
One of the main tools in \cite{BM13} is a rank two hyperbolic lattice associated to any wall. More precisely, if we fix a Mukai vector $\vv \in H^*_{\mathrm{alg}}(X,\bZ)$ with $\vv^2>0$, and a wall $\cW$ of the chamber decomposition with respect to $\vv$, then we can define $\cH_\cW \subset H^*_{\mathrm{alg}}(X,\bZ)$ to be the set of classes
$$\ww \in \cH_\cW \quad \Leftrightarrow \quad \im \frac{Z(\ww)}{Z(\vv)} = 0 \quad \text{for all} \quad \sigma = (Z, \cP) \in \cW.$$
By \cite[Proposition 5.1]{BM13}, which also works when $\vv$ is not primitive, the set $\cH_\cW$ is a primitive rank two hyperbolic lattice. Conversely, given a primitive rank two hyperbolic sublattice $\cH \subset H^*_{\mathrm{alg}}(X,\bZ)$ containing $\vv$, we can define a potential wall $\cW$ in $\stab(X)$ to be a connected component of the real codimension one submanifold of stability conditions $\sigma = (Z, \cP)$ which satisfy the condition that $Z(\cH)$ is contained in a line. Notice that every (potential) wall is associated to a unique rank two hyperbolic lattice $\cH$ whereas each lattice may give rise to many (potential) walls. Roughly speaking, whilst crossing a wall $\cW$, all the relevant Mukai vectors lie in the same hyperbolic lattice $\cH_\cW$. This simple observation reduces the analysis of wall-crossing to elementary lattice-theoretic computations.

We also need to recall the names of some special classes in $\cH$. A class $\uu \in \cH$ is called a \emph{spherical class} if $\uu^2=-2$, or an \emph{isotropic class} if $\uu^2=0$. We say a hyperbolic lattice $\cH$ is an \emph{isotropic lattice} if $\cH$ contains at least one non-zero isotropic class. We say a (potential) wall $\cW$ is an \emph{isotropic wall} if its associated hyperbolic lattice $\cH$ is an isotropic lattice. 

After establishing the relation between walls and rank two hyperbolic sublattices, we review the notions of a positive cone and an effective cone. We assume we have a potential wall $\cW$ and its associated rank two hyperbolic lattice $\cH$. The \emph{positive cone} $P_\cH$ is a cone in $\cH \otimes \bR$, which is generated by integral classes $\uu \in \cH$, with $\uu^2 \geq 0$ and $(\vv, \uu)>0$. We call any integral class in $P_\cH$ a \emph{positive class}. Note that the positive cone $P_\cH$ only depends on the choice of $\cH$, not on the choice of the wall associated to $\cH$.

In comparison, the \emph{effective cone} $C_\cW$ depends on the choice of a potential wall. It is a cone in $\cH \otimes \bR$, which is generated by integral classes $\uu \in \cH$, with $\uu^2 \geq -2$ and $\re \frac{Z(u)}{Z(v)} > 0$ for a fixed $\sigma=(Z, \cP) \in \cW$. This notion is well-defined. Indeed, by \cite[Proposition 5.5]{BM13}, this cone does not depend on the choice of $\sigma \in \cW$. We call any integral class in $C_\cW$ an \emph{effective class}. We also point out that for different walls $\cW$ associated to the same hyperbolic lattice $\cH$, the effective cone $C_\cW$ might differ by spherical classes. However, they all contain the same positive cone $P_\cH$. 

We also need to make the same genericity assumption as in \cite[Remark 5.6]{BM13}.

\subsection{First classification theorem}

We will classify all the walls in the stability manifold from two points of view. Note that various types of walls have been defined in \cite[Section 8]{BM12} and \cite[Definition 2.20]{BM13}. From now on, we always denote a generic stability condition on a wall $\cW$ by $\sigma_0$, while two generic stability conditions on two different sides of the $\cW$ by $\sigma_+$ and $\sigma_-$. The chambers which contain $\sigma_+$ and $\sigma_-$ are denoted by $\cC_+$ and $\cC_-$ respectively. Moreover, we always assume that $\sigma_+$ and $\sigma_-$ are sufficiently close to the wall, whose precise meaning is contained in \cite[Proof of Proposition 5.1]{BM13}.

We first classify walls according to whether there exists any $\sigma_0$-stable objects of class $\vv$. We say a (potential) wall associated to the class $\vv$ is \emph{totally semistable}, if there is no $\sigma_0$-stable objects of class $\vv$ for a generic stability condition $\sigma_0 \in \cW$. Otherwise, $\cW$ is not totally semistable. In other words, the (potential) wall $\cW$ is totally semistable if and only $M_{\sigma_+}(\vv)$ and $M_{\sigma_-}(\vv)$ are completely disjoint from each other, and $\cW$ is not totally semistable if and only if they contain a common open dense subset $M_{\sigma_0}(\vv)$. There is a third way to describe these walls, which reveals more on the reason why the name is obtained. The wall $\cW$ is totally semistable if every stable $\sigma_+$-stable object becomes strictly $\sigma_0$-semistable, which is equivalent to the condition that every stable $\sigma_-$-stable object is strictly $\sigma_0$-semistable. Otherwise $\cW$ is not totally semistable. 

We are ready to state the following numerical criterion for a wall to be totally semistable for an arbitrary class $\vv \in \cH$ with $\vv^2>0$, regardless of whether it is primitive or has O'Grady type. The theorem generalises the first half of \cite[Theorem 5.7]{BM13}.

\begin{theorem}
\label{thm:1st-classification}
Let $\vv \in \cH$ be a positive class with $\vv^2>0$, and $\cW$ be a potential wall for $\vv$. Then $\cW$ is a totally semistable wall for $\vv$ if and only if either of the following two conditions holds
\begin{itemize}\itemindent=20pt
\item[(TS1)] there exists an effective spherical class $\ss \in \cH$ with $(\ss, \vv)<0$;
\item[(TS2)] there exists an isotropic class $\ww \in \cH$ with $(\ww, \vv)=1$.
\end{itemize}
\end{theorem}

In particular, if $\vv$ is a non-primitive class, then (TS2) cannot happen and (TS1) is the only possibility for a totally semistable wall. As an application of the proof of this theorem, we also obtain the bigness of $\ell_{\sigma_0}$ for an arbitrary class $\vv$, which generalises \cite[Theorem 1.4(a)]{BM12} and Proposition \ref{old-semiampleness}. 

\begin{proposition}
\label{prop:semiampleness}
Let $\vv \in \cH$ be a positive class with $\vv^2>0$, and $\cW$ be a potential wall for $\vv$. For a generic $\sigma_0 \in \cW$, its image $\ell_{\cC_\pm}(\sigma_0)$ under the local \bmm map with respect to the chamber $\cC_\pm$ induces a birational morphism \[\pi_{\pm}: M_{\sigma_\pm}(\vv) \to \bar{M}_\pm\] which contracts curves in $M_{\sigma_\pm}(\vv)$ parametrising S-equivalent objects under the stability condition $\sigma_0$, where $\bar{M}_\pm$ is the image of $\pi_{\pm}$ in $M_{\sigma_0}(\vv)$.
\end{proposition}

\subsection{Second classification theorem}

The second classification result relies on Proposition \ref{prop:semiampleness} (or Proposition \ref{old-semiampleness}). Since $\pi_+: M_{\sigma_+}(\vv) \to \bar{M}_+$ is a birational morphism, it could be one of three possible types: a divisorial contraction, a small contraction, or an isomorphism. In fact, in the next theorem, we will see that $\pi_+$ and $\pi_-$ always correspond to the same contraction type because the numerical criteria do not tell the difference of two sides of $\cW$. So it doesn't matter whether the definition is made for $\pi_+$ or $\pi_-$.

In the first two cases, since $\pi_+$ and $\pi_-$ induce actual contractions, the two moduli spaces $M_{\sigma_+}(\vv)$ and $M_{\sigma_-}(\vv)$ are not the same and hence $\cW$ is a genuine wall which we call a \emph{divisorial wall} and a \emph{flopping wall} respectively; the reason for the name will be explained in Section \ref{classificationproofs}. If $\pi_+$ and $\pi_-$ are both isomorphisms, we will see that $M_{\sigma_+}(\vv)$ and $M_{\sigma_-}(\vv)$ are isomorphic. By the first classification Theorem \ref{thm:1st-classification}, they could be either disjoint from each other, in which case we call $\cW$ a \emph{fake wall}, or identical with each other, in which case we say that $\cW$ is not a wall. We refer the reader to \cite[Section 8]{BM12} and \cite[Definition 2.20]{BM13} for the names of these walls. 

With these notions at hand, we can now state the second classification result which provides a numerical criterion for each of the three types of contraction discussed above. This generalises the second half of \cite[Theorem 5.7]{BM13}, but we can only prove it in the O'Grady situation:

\begin{theorem}
\label{thm:2nd-classification}
Let $\vv \in \cH$ be a positive class with $\vv=2\vv_p$, where $\vv_p^2=2$, and $\cW$ be a potential wall for $\vv$. Then $\cW$ can be one of the following three types:
\begin{itemize}
\item $\cW$ is a wall inducing a divisorial contraction if and only if at least one of the following two conditions holds:
\begin{itemize}\itemindent=30pt
\item[(BN)] there exists a spherical class $\ss$ with $(\ss, \vv)=0$, or
\item[(LGU)] there exists an isotropic class $\ww$ with $(\ww, \vv)=2$.
\end{itemize}
Moreover, the condition (LGU) necessarily implies the condition (BN).
\item Otherwise, $\cW$ is a wall inducing a small contraction if and only if the following condition holds:
\begin{itemize}\itemindent=30pt
\item[(SC)] there exists a spherical class $\ss$ with $0 < (\ss, \vv) \leq \frac{\vv^2}{2}=4$.
\end{itemize}
\item In all other cases, $\cW$ is either a totally semistable wall, or not a wall. 
\end{itemize}
\end{theorem}

We discuss briefly the differences between \cite[Theorem 5.7]{BM13} in the primitive case and our classification in the O'Grady setting. The main difference shows up in the condition required for a divisorial wall. On the one hand, the non-primitivity of our Mukai vectors excludes the possibility of having a Hilbert-Chow wall. On the other hand, we will prove in Lemma \ref{lem:dc0} that (LGU) always implies (BN) for moduli spaces of O'Grady type. Therefore, (BN) is the unique condition that is required to characterise a divisorial wall. In other words, every divisorial wall is a wall of Brill-Noether type, while it is also a wall of Li-Gieseker-Uhlenbeck type if both (BN) and (LGU) are satisfied at the same time. However, there is a subtle difference between the cases (BN) and (LGU) which can be seen by looking at the HN filtration of the generic semistable object along the divisor; see Remark \ref{LGUBNdiff} for more details. Another difference between our Theorem \ref{thm:2nd-classification} and \cite[Theorem 5.7]{BM13} is that there are no flopping walls arising from a decomposition $\vv=\aa+\bb$ into positive classes $\aa$ and $\bb$. Indeed, this case cannot happen because the O'Grady spaces have such small dimension; see Lemma \ref{noSC1walls} for more details. 

The above classification shows us how a moduli space could change when the stability condition approaches a wall. Although the contraction induced by a wall can be arbitrarily misbehaved, we can always find a birational map, which identifies open subsets in the moduli spaces $M_{\sigma_+}(\vv)$ and $M_{\sigma_-}(\vv)$ with complements of codimension at least two. By composing these birational maps, we can obtain the same conclusion for any pair of generic stability conditions. This is the content of the following result, which generalises \cite[Theorem 1.1]{BM13}. 

\begin{theorem}
\label{thm:identify-cones}
Let $\vv \in \cH$ be a positive class with $\vv=2\vv_p$, where $\vv_p^2=2$, and $\cW$ a potential wall for $\vv$. Then there exists a derived (anti-)autoequivalence $\Phi$ of $\cD(X)$, which induces a birational map $\Phi_*: M_{\sigma_+}(\vv) \dashrightarrow M_{\sigma_-}(\vv)$, such that 
\begin{itemize}
\item $\Phi_*$ is stratum preserving, and
\item one can find open subsets $U_\pm \subset M_{\sigma_\pm}(\vv)$, with complements of codimension at least two, and $\Phi_*: U_+ \to U_-$ an isomorphism, i.e. $\Phi_*$ is an isomorphism in codimension one. 
\end{itemize}
In particular, $\Phi_*$ identifies the N\'{e}ron-Severi lattices $\ns(M_{\sigma_+}(\vv)) \cong \ns(M_{\sigma_-}(\vv))$, as well as the movable cones $\mov(M_{\sigma_+}(\vv)) \cong \mov(M_{\sigma_-}(\vv))$ and big cones $\bigcone(M_{\sigma_+}(\vv)) \cong \bigcone(M_{\sigma_-}(\vv))$. 

More generally, the same result is true for any two generic stability conditions $\sigma, \tau \in \stab(X)$.
\end{theorem}

\begin{remark}
Let us take this opportunity to clarify why an autoequivalence $\Phi$ of $\cD(X)$, which sends every $E\in U_+\subset M_{\sigma_+}(\vv)$ to an object $\Phi(E)\in U_-\subset M_{\sigma_-}(\vv)$, induces a birational map $\Phi_*$ defined as a morphism on $U_+$. If we let $\mathfrak{M}_{\sigma_\pm}(\vv)$ denote the moduli stacks of $\sigma_\pm$-semistable objects then, by the GIT construction in \cite{BM12}, we have classifying morphisms $f_\pm:\mathfrak{M}_{\sigma_\pm}(\vv)\to M_{\sigma_\pm}(\vv)$ together with the fact that the moduli spaces $M_{\sigma_\pm}(\vv)$ universally corepresent the moduli stacks $\mathfrak{M}_{\sigma_\pm}(\vv)$. Now, the autoequivalence $\Phi$ defines a morphism from an open substack of $\mathfrak{M}_{\sigma_+}(\vv)$ to an open substack of $\mathfrak{M}_{\sigma_-}(\vv)$, which is in fact an isomorphism (because the functor $\Phi$ has an inverse). Note that these two open substacks are in fact preimages of open subschemes of the corresponding moduli spaces (along $f_+$ and $f_-$); the reason is that if $\Phi(E)$ is $\sigma_-$-semistable, then every object in the same $\sigma_+$-S-equivalence class of $E$ is also mapped to a $\sigma_-$-semistable object. Finally, the isomorphism between the two open substacks descends to an isomorphism between the two open subschemes by the universal corepresentability of $M_{\sigma_\pm}(\vv)$. In particular, if $U_\pm\subset M_{\sigma_\pm}(\vv)$ is an open subset then $U_\pm$ universally corepresents its preimage $f_\pm^{-1}(U_{\pm})$ since universal corepresentability is preserved under arbitrary base change; see \cite[Definition 2.2.1 \& Theorem 4.3.4]{HuLe10a} for more details.
\end{remark}

\begin{remark}
\label{rmk:unified-notation}
This theorem suggests that we can use the notation $\ns (M(\vv))$ for the N\'{e}ron-Severi lattice of the moduli space without specifying the generic stability condition (or rather its chamber). Similarly, we will also use the notations $\mov(M(\vv))$ and $\bigcone(M(\vv))$ for the movable cone and big cone of the moduli space associated to any chamber in $\stab(X)$. 
\end{remark}

The fact that all generic moduli spaces have canonically identified N\'{e}ron-Severi groups shows that, all chamberwise local \bmm maps have the same target and hence could possibly be glued together into a global \bmm map, which will reveal the power of wall crossings in the study of birational geometry of the moduli spaces; see Section \ref{globalbmmap}.

\section{Classification of Walls: Proofs}\label{classificationproofs}

In this section, we prove all the results which were stated in Section \ref{classificationresults}. First, we fix notation. Throughout, we set $\vv=m\vv_p \in H^*(X, \bZ)$ to be a Mukai vector, for some positive integer $m$, and $\vv_p$ a primitive vector with $\vv_p^2>0$. Although we will mainly focus on moduli spaces of O'Grady type, where $m=2$ and $\vv_p^2=2$, some of our results are actually true for arbitrary $\vv$. Let $\cH \subset H^*(X, \bZ)$ be a primitive hyperbolic rank two sublattice containing $\vv$, and $\cW \subset \stab(X)$ a potential wall associated to $\cH$. A generic stability condition on $\cW$ is denoted by $\sigma_0$, while generic stability conditions on the two sides of the wall are denoted by $\sigma_+$ and $\sigma_-$. The phase functions of the central charges of $\sigma_+, \sigma_0, \sigma_-$ are denoted by $\phi_+, \phi_0, \phi_-$ respectively. We also use $\pi_\pm: M_{\sigma_\pm}(\vv) \to \bar{M}_\pm$ for the morphism induced by $\ell_{\sigma_0}$ on the generic moduli spaces $M_{\sigma_\pm}(\vv)$. Finally, we point out that, in this section, for simplicity of notations, any morphism (or birational map) between two moduli spaces induced by a derived (anti-)autoequivalence $\Phi$ of $\cD(X)$ is still denoted by $\Phi$, by abuse of notation. 

We should point out that a large portion of our proofs are, in fact, already contained in \cite{BM13}. Therefore, to avoid repetition we will reference the results and proofs of \cite{BM13} freely and only point out where the differences are and what additional arguments (if any) need to be added. In particular, our focus will be on explaining the extra effort required to deal with the problems caused by the presence of the singular locus. Moreover, all results that we prove for $\sigma_+$ also hold for $\sigma_-$ with identical proof, which we will not explicitly mention in every statement. 

\subsection{Totally semistable walls}

In this subsection we prove the criterion stated in Theorem \ref{thm:1st-classification} for a wall to be totally semistable. We will also show that $\pi_\pm$ are always birational morphisms. We prove these results for arbitrary Mukai vectors $\vv$ with $\vv^2>0$. The notion of the minimal class in a $G_\cH$-orbit will be frequently used; for its definition, we refer the reader to \cite[Proposition and Definition 6.6]{BM13}. 

We start by proving some lemmas. The first two deal with non-minimal and minimal classes respectively, whilst the third one establishes a bridge between these two kinds of classes when they lie in the same $G_\cH$-orbit. This will be used later to deduce results for non-minimal classes from the existing results for minimal classes. After that, we will prove Theorem \ref{thm:1st-classification} and Proposition \ref{prop:semiampleness}.

\begin{lemma}
\label{lem:ts1}
If $\vv$ is non-minimal, then there is no $\sigma_0$-stable objects of class $\vv$, which means $\cW$ is a totally semistable wall. 
\end{lemma}

\begin{proof}
This is the first statement in \cite[Proposition 6.8]{BM13}. Note that the proof there works for both primitive and non-primitive classes. 
\end{proof}

\begin{lemma}
\label{lem:ts2}
Let $\vv$ be a minimal, non-primitive class. Then there exist $\sigma_0$-stable objects of class $\vv$.
\end{lemma}

\begin{proof}
We write $\vv=m\vv_p$ where $m>1$ and $\vv_p$ is a minimal primitive class. There are two cases to consider. 

If there is no isotropic class $\ww$ with $(\ww, \vv_p)=1$, then by \cite[Theorem 5.7]{BM13}, there exist $\sigma_0$-stable objects of class $\vv_p$. Therefore by \cite[Lemma 2.16]{BM13}, there exist $\sigma_0$-stable objects of class $\vv$. 

If instead there is an isotropic class $\ww$ with $(\ww, \vv_p)=1$, then by \cite[Proposition 8.2]{BM13} and the discussion above it (see also \cite{LQ11, Lo12}), $\ell_{\sigma_0}$ induces a Li-Gieseker-Uhlenbeck morphism. Since $(\ww, \vv)=m>1$, a generic stable object in $M_{\sigma_\pm}(\vv)$ corresponds to a locally free stable sheaf in a Gieseker moduli space, hence remains $\sigma_0$-stable. 
\end{proof}

\begin{lemma}
\label{lem:ts3}
Let $\vv$ be any class not satisfying (TS2), and $\vv_0$ be the minimal class in the $G_\cH$-orbit of $\vv$. Then there exists a derived autoequivalence $\Phi_+$ of $\cD(X)$ defined as a composition of spherical twists, such that for every $\sigma_0$-stable object $E$ of class $\vv_0$, $\Phi_+(E)$ is a $\sigma_+$-stable object of class $\vv$. Moreover, for any other $\sigma_0$-stable object $E'$ of class $\vv_0$, $\Phi_+(E)$ and $\Phi_+(E')$ are not S-equivalent with respect to $\sigma_0$.
\end{lemma}

\begin{proof}
By \cite[Theorem 5.7]{BM13} (in primitive case) or Lemma \ref{lem:ts2} (in non-primitive case), there always exist $\sigma_0$-stable objects of class $\vv_0$. The first statement is contained in \cite[Proposition 6.8]{BM13}, and the second statement is contained in the proof of \cite[Corollary 7.3]{BM13}. 
\end{proof}

Now we are ready to prove the criterion for totally semistable walls, and show that the contraction induced by $\ell_{\sigma_0}$ on $M_{\sigma_\pm}(\vv)$ are always birational for an arbitrary Mukai vector $\vv$, which generalises Proposition \ref{old-semiampleness}.

\begin{proof}[Proof of Theorem \ref{thm:1st-classification}]
If $\vv$ is a primitive class, the statement is proved in \cite[Theorem 5.7]{BM13}. Now we assume $\vv$ is non-primitive. Note that in this case (TS2) cannot happen, hence we only need to prove that (TS1) is both sufficient and necessary for $\cW$ to be a totally semistable wall. The sufficiency is the content of Lemma \ref{lem:ts1}, and Lemma \ref{lem:ts2} proves the necessity by contradiction. 
\end{proof}

\begin{proof}[Proof of Proposition \ref{prop:semiampleness}]
If $\vv$ is a primitive class, the statement is \cite[Theorem 1.4(a)]{BM12}. Now we assume $\vv$ is non-primitive. Note that by \cite[Theorem 1.1]{BM12}, $\ell_{\sigma}$ is nef on $M_{\sigma_+}(\vv)$, and contracts curves which generically parametrise S-equivalent objects with respect to $\sigma_0$; see Proposition \ref{old-semiampleness}. Therefore, it suffices to show that there exists a dense open subset $U \subset M_{\sigma_0}(\vv)$ in which no curve is contracted by $\ell_{\sigma_0}$. If $\vv$ is minimal, we simply take $U$ to be the open subset of $\sigma_0$-stable objects. By Lemma \ref{lem:ts2}, $U$ is non-empty and hence dense in $M_{\sigma_0}(\vv)$. If $\vv$ is non-minimal, let $\vv_0$ be the corresponding minimal class with an open subset $U_0 \subset M_{\sigma_+}(\vv_0)$ of $\sigma_0$-stable objects. Then, by Lemma \ref{lem:ts3}, $U=\Phi(U_0)$ is non-empty and does not contain any curve which can be contracted by $\ell_{\sigma_0}$.
\end{proof}

By virtue of Proposition \ref{prop:semiampleness}, the contractions $\pi_\pm: M_{\sigma_\pm}(\vv) \to \bar{M}_\pm$ induced by $\ell_{\sigma_0}$ are birational morphisms, and we will deal with three mutually exclusive cases: a divisorial contraction, a small contraction, or no contraction at all, i.e. an isomorphism. We will prove the numerical criterion for each case, and show that $M_{\sigma_+}(\vv)$ and $M_{\sigma_-}(\vv)$ are always birational and isomorphic in codimension one. 

We also point out that from now on, we always restrict ourselves to the O'Grady situation. That is, $\vv=2\vv_p$ where $\vv_p$ is a primitive Mukai vector with $\vv_p^2=2$. We do this because many of the following proofs will rely heavily on the existence of a symplectic resolution. 

\subsection{Walls inducing divisorial contractions}\label{divcontractions}

In this subsection we prove Theorems \ref{thm:2nd-classification} and \ref{thm:identify-cones} in the case that $\ell_{\sigma_0}$ induces a divisorial contraction. We remind readers that the singular locus in the O'Grady moduli spaces has codimension two. Therefore, the contraction of a divisor is a phenomenon which happens in the smooth locus. This is the reason why most of the arguments in \cite{BM13} still apply in the O'Grady situation. 

This subsection consists mainly of a series of lemmas; each of which will become an integral part in the proof of Theorems \ref{thm:2nd-classification} and \ref{thm:identify-cones}. At the end of this subsection, we summarise all the results with Proposition \ref{prop:divisorial-contra}. We start with the following special feature of classes of O'Grady type, which will be used later to simplify a few proofs. 

\begin{lemma}
\label{lem:dc0}
Let $\vv$ be any class of O'Grady type. If (LGU) holds for $\vv$, then (BN) also holds for $\vv$. Moreover, $\vv$ is a minimal class, and we can choose the class $\ww$ in (LGU) so that $M_{\sigma_0}(\ww)=M_{\sigma_0}^{\mathrm{st}}(\ww)$.
\end{lemma}

\begin{proof}
By (LGU) we have $(\ww, \vv_p)=1$ which implies $\ww$ is a primitive class. Using the notation in \cite[Lemma 8.1]{BM13}, we have either $(\ww_0, \vv_p)=1$ or $(\ww_1, \vv_p)=1$. Suppose $(\ww_0, \vv_p)=1$. Then we have $(\vv_p-\ww_0)^2=\vv_p^2-2=0$ which can only happen when $\vv_p-\ww_0$ is a multiple of $\ww_0$ or $\ww_1$. If it is a multiple of $\ww_0$, then so is $\vv_p$, which contradicts $\vv_p^2=2$. Thus, we have $\vv_p-\ww_0=k\ww_1$ for some non-negative integer $k$ and hence $2=\vv_p^2=(\ww_0+k\ww_1)^2=2k(\ww_0, \ww_1)$ which implies $k=1$ and $(\ww_0,\ww_1)=1$. Swapping subscripts shows that starting from $(\ww_1, \vv_p)=1$ also yields $\vv_p=\ww_0+\ww_1$ and $(\ww_0,\ww_1)=1$. Therefore, we always have $(\vv_p, \ww_0)=(\vv_p, \ww_1)=1$.

Now, the unique effective spherical class is given, up to sign, by $\ss = \ww_0 - \ww_1$. In particular, $(\ss,\vv)=2(\ss,\vv_p)=0$ and we see that (BN) must also hold. We also have $M_{\sigma_0}^{\mathrm{st}}(\ww_0)=M_{\sigma_0}(\ww_0)$ and $M_{\sigma_0}^{\mathrm{st}}(\ww_1)=\emptyset$, where $\ww_1=\ww_0+(\ss, \ww_0)\ss$, and so we can choose the class $\ww$ in (LGU) to be $\ww_0$; see \cite[Proposition 6.3 and Lemma 8.1]{BM13} for more details.
\end{proof}

\begin{remark}\label{LGUBNdiff}
In the case when (LGU), and hence (BN), is satisfied, the HN filtration along the divisor does not quite behave like the normal BN case. Indeed, if we look at the arguments of \cite[proof of Theorem 1.1, p.40]{BM13} then we see that the divisor of semistable objects is still the BN-divisor, but the HN filtration of the generic semistable object does not have the normal BN-form, i.e. it does not contain a spherical factor.
\end{remark}

\begin{lemma}
\label{lem:dc1}
Let $\vv$ be a minimal class of O'Grady type. If (BN) does not hold, then the set of $\sigma_0$-stable objects in $M_{\sigma_+}(\vv)$ has complement of codimension at least two.
\end{lemma}

\begin{proof}
If $\cH$ does not contain any isotropic vector, we follow the proof of \cite[Lemma 7.2]{BM13}. Otherwise, by Lemma \ref{lem:dc0}, (LGU) does not hold either. Moreover there is no isotropic class $\ww$ with $(\ww, \vv)=1$ since $\vv$ is non-primitive. Therefore, we can follow steps 1 and 2 of the proof of \cite[Proposition 8.6]{BM13} to get the conclusion. The only change that we have to make in both proofs is that the application of their Theorem 3.8 has to be replaced by our Proposition \ref{prop:symp-divisor}.
\end{proof}

\begin{lemma}
\label{lem:dc2}
Let $\vv$ be a non-minimal class of O'Grady type. If (BN) does not hold, then there is an open subset of $M_{\sigma_+}(\vv)$ with complement of codimension at least two, on which no curve is contracted by $\ell_{\sigma_0}$. 
\end{lemma}

\begin{proof}
We write $\vv_0$ for the minimal class in the $G_\cH$-orbit containing $\vv$, then (BN) does not hold for $\vv_0$. By Lemma \ref{lem:dc1}, there exists an open subset $U_0 \subset M_{\sigma_+}(\vv_0)$ with complement of codimension at least two, which parametrises objects remaining stable under $\sigma_0$. By Lemma \ref{lem:ts3}, there is a derived autoequivalence $\Phi_+$ of $\cD(X)$ which induces an injective morphism $\Phi_+: U_0 \to M_{\sigma_+}^{\mathrm{st}}(\vv)$, and the image $\Phi_+(U_0)$ does not contain any curve that generically parametrises S-equivalent objects under $\sigma_0$. Therefore, no curve in $\Phi_+(U_0)$ will be contracted by $\ell_{\sigma_0}$. It remains to show $\Phi_+(U_0)$ has complement of codimension two in $M_{\sigma_+}(\vv)$. 

Since $\vv_0$ is a minimal class and (LGU) does not hold (by Lemma \ref{lem:dc0}), Theorem \ref{thm:1st-classification} tells us that $\cW$ is not a totally semistable wall for $\vv_{0,p}$, i.e. the effective primitive class proportional to $\vv_0$. By Lemma \ref{lem:ts3}, the same derived auto-equivalence $\Phi_+$ takes every $\sigma_0$-stable object of class $\vv_{0,p}$ to a $\sigma_+$-stable object of class $\vv_p$, i.e. the effective primitive class proportional to $\vv$. Now by Lemma \ref{lem:stratum}, $\Phi_+: M_{\sigma_+}(\vv_0) \dashrightarrow M_{\sigma_+}(\vv)$ is a stratum preserving birational map and so we can apply Proposition \ref{prop:codim-two} to conclude that $\Phi_+(U_0)$ has complement of codimension two in $M_{\sigma_+}(\vv)$. 
\end{proof}

\begin{lemma}
\label{lem:dc3}
Let $\vv$ be any class of O'Grady type and assume that (BN) holds for $\vv$ while (LGU) does not hold for $\vv$. Then $\pi_+$ is a divisorial contraction of Brill-Noether type. 
\end{lemma}

\begin{proof}
We first consider the case that $\cH$ contains no isotropic vectors. The divisor contracted by $\pi_+$ is constructed in \cite[Lemma 7.4]{BM13} when $\vv$ is a minimal class, and in \cite[Lemma 7.5]{BM13} when $\vv$ is a non-minimal class. Next, we consider the case that $\cH$ is isotropic. By \cite[Proposition 6.3]{BM13}, there is only one effective spherical class in $\cH$, hence $\vv$ is necessarily minimal. The divisor contracted by $\pi_+$ is constructed in \cite[Lemma 8.8]{BM13}, which also shows the contraction has Brill-Noether type.
\end{proof}

\begin{lemma}
\label{lem:dc4}
Let $\vv$ be any class of O'Grady type and assume that (BN) holds for $\vv$ while (LGU) does not hold. Then there is a derived (anti-)autoequivalence $\Phi$ inducing a stratum preserving birational map $\Phi: M_{\sigma_+}(\vv) \dashrightarrow M_{\sigma_-}(\vv)$, which is an isomorphism in codimension one. 
\end{lemma}

\begin{proof}
By Lemma \ref{lem:dc3} we know that $\pi_+$ is a contraction of Brill-Noether type under the given assumption. If $\vv$ is a minimal class, then we can follow the proof of \cite[Theorem 1.1(b)]{BM13} in the Brill-Noether case and obtain a spherical twist $\Phi$ of $\cD(X)$, inducing a birational map $\Phi: M_{\sigma_+}(\vv) \dashrightarrow M_{\sigma_-}(\vv)$, which is an isomorphism in codimension one. It remains to show that it is stratum preserving. Note that under the given assumption, for the primitive class $\vv_p$ (which is the primitive part of $\vv$), the wall $\cW$ also induces a Brill-Noether divisorial contraction and is not totally semistable by \cite[Theorem 5.7]{BM13}. We apply the same part of proof in \cite[Theorem 1.1(b)]{BM13} to conclude that the restriction of $\Phi$ on the moduli spaces with primitive classes $\Phi_p: M_{\sigma_+}(\vv_p) \dashrightarrow M_{\sigma_-}(\vv_p)$ is also a birational map. By Lemma \ref{lem:stratum}, we conclude that $\Phi$ is a stratum preserving birational map between moduli spaces with non-primitive classes. 

Now we consider the case that $\vv$ is not a minimal class. We still write $\vv_0$ for the minimal class in the $G_\cH$-orbit of $\vv$. The above construction gives $\Phi_0: M_{\sigma_+}(\vv_0) \dashrightarrow M_{\sigma_-}(\vv_0)$. By \cite[Lemma 7.5]{BM13}, there is a composition of spherical twists, say $\Phi_+$, inducing a birational map $\Phi_+: M_{\sigma_+}(\vv_0) \dashrightarrow M_{\sigma_+}(\vv)$. Moreover, since $\cW$ is not a totally semistable wall for the primitive class $\vv_{0,p}$ (which is the primitive part of $\vv_0$), we can apply \cite[Proposition 6.8]{BM13} to see that $\Phi_+$ also induces a birational map $\Phi_{+,p}: M_{\sigma_+}(\vv_{0,p}) \dashrightarrow M_{\sigma_+}(\vv_p)$ and Lemma \ref{lem:stratum} shows that $\Phi_+$ is a stratum preserving birational map. By \cite[Lemma 7.5]{BM13}, there exists an open subset $U_+ \subset M_{\sigma_+}(\vv_0)$ with complement of codimension at least two, such that $\Phi_+$ is an injective morphism on $U_+$ and so by Proposition \ref{prop:codim-two}, we see that $\Phi_+$ is an isomorphism in codimension one. 

In the same way we can construct $\Phi_-: M_{\sigma_-}(\vv_0) \dashrightarrow M_{\sigma_-}(\vv)$. Then the composition $\Phi_- \circ \Phi_0 \circ \Phi_+^{-1}: M_{\sigma_+}(\vv) \dashrightarrow M_{\sigma_-}(\vv)$ is a stratum preserving birational map which is isomorphic in codimension one, as desired. 
\end{proof}

\begin{lemma}
\label{lem:dc5}
Let $\vv$ be any class of O'Grady type and assume that (LGU) holds for $\vv$. Then $\pi_+$ is a divisorial contraction of Li-Gieseker-Uhlenbeck type.
\end{lemma}

\begin{proof}
By Lemma \ref{lem:dc0}, $\vv$ is a minimal class and we can assume $M_{\sigma_0}(\ww)=M_{\sigma_0}^{\mathrm{st}}(\ww)$ for the class $\ww$ in (LGU). We can simply follow the proof of \cite[Lemma 8.7]{BM13} to get the divisor contracted by $\pi_+$, and it is clear from there that the contraction is of Li-Gieseker-Uhlenbeck type.
\end{proof}

\begin{lemma}
\label{lem:dc6}
Let $\vv$ be any class of O'Grady type and assume (LGU) holds for $\vv$. Then there is a derived (anti-)autoequivalence $\Phi$ of $\cD(X)$ inducing a stratum preserving birational map $\Phi: M_{\sigma_+}(\vv) \dashrightarrow M_{\sigma_-}(\vv)$, which is an isomorphism in codimension one. 
\end{lemma}

\begin{proof}
The proof of this lemma is similar to that of Lemma \ref{lem:dc4}. However, in this situation, while $\cW$ is a Li-Gieseker-Uhlenbeck type divisorial wall for the class $\vv$, it is a Hilbert-Chow type totally semistable divisorial wall for the class $\vv_p$. By Lemma \ref{lem:dc0}, $\vv$ and hence $\vv_p$, are minimal classes. The proof of \cite[Theorem 1.1(b)]{BM13} in the Li-Gieseker-Uhlenbeck case shows that a derived anti-autoequivalence $\Phi$ (which is in fact a spherical twist) induces a birational map $\Phi: M_{\sigma_+}(\vv) \dashrightarrow M_{\sigma_-}(\vv)$ which is isomorphic in codimension one. The same proof as in the Hilbert-Chow case shows that $\Phi$ also induces a birational map $\Phi_p: M_{\sigma_+}(\vv_p) \dashrightarrow M_{\sigma_-}(\vv_p)$. Therefore by Lemma \ref{lem:stratum}, $\Phi$ is also a stratum preserving birational map.
\end{proof}

We conclude the discussion about divisorial walls by summarising all the above lemmas and proving the following proposition, which is nothing but the collection of statements about divisorial contractions in Theorems \ref{thm:2nd-classification} and \ref{thm:identify-cones}.

\begin{proposition}
\label{prop:divisorial-contra}
The potential wall $\cW$ associated to a rank two hyperbolic lattice $\cH \subset H^*_{\mathrm{alg}}(X,\bZ)$ containing $\vv$ induces a divisorial contraction if and only if at least one of the following two conditions holds
\begin{itemize}\itemindent=30pt
\item[(BN)] there exists a spherical class $\ss$ with $(\ss, \vv)=0$, or
\item[(LGU)] there exists an isotropic class $\ww$ with $(\ww, \vv)=2$.
\end{itemize}
Moreover, the condition (LGU) necessarily implies the condition (BN).

In either case, we can find an (anti-)autoequivalence $\Phi$ of $\cD(X)$, which induces a birational map $\Phi: M_{\sigma_+}(\vv) \dashrightarrow M_{\sigma_-}(\vv)$. It is stratum preserving and an isomorphism in codimension one.
\end{proposition}

\begin{proof}
Lemma \ref{lem:dc0} shows that (LGU) implies (BN). Therefore it suffices to show $\cW$ induces a divisorial contraction if and only if (BN) holds. The necessity of (BN) is proved in Lemma \ref{lem:dc1} for minimal classes and in Lemma \ref{lem:dc2} for non-minimal classes by contradiction. The sufficiency of (BN), as well as the type of contraction, are the contents of Lemmas \ref{lem:dc3} and \ref{lem:dc5}. Finally, the birational map $\Phi$ is constructed in Lemma \ref{lem:dc4} for Brill-Noether contractions and in Lemma \ref{lem:dc6} for Li-Gieseker-Uhlenbeck contractions.
\end{proof}

\subsection{Walls inducing small contractions or no contractions}\label{small-contractions}

In this subsection, we prove Theorems \ref{thm:2nd-classification} and \ref{thm:identify-cones} in the case that $\cW$ induces a small contraction or no contraction at all. We point out that most of the results in this subsection work for all effective classes $\vv$ of divisibility two, because the existence of a symplectic resolution will not be used in the proof. 

As before, this subsection consists mainly of a series of lemmas; each of which will become an integral part in the proof of Theorems \ref{thm:2nd-classification} and \ref{thm:identify-cones}. All the results will be summarised in Propositions \ref{prop:small-contra} and \ref{prop:no-contra} at the end. These two propositions, together with Proposition \ref{prop:divisorial-contra}, cover all the cases in Theorems \ref{thm:2nd-classification} and \ref{thm:identify-cones}.

\begin{lemma}\label{noSC1walls}
Let $M=M_\sigma(\vv)$ be a moduli space of O'Grady type. Then there are no small contractions arising from a decomposition of $\vv$ into the sum $\aa+\bb$ of two positive classes.
\end{lemma}

\begin{proof} 
Suppose we have such a decomposition for $\vv$. Then \[8=\vv^2=\aa^2+2(\aa,\bb)+\bb^2\geq2(\aa,\bb)\implies(\aa,\bb)\leq4.\] If we suppose further that $\cW$ does not induce a divisorial contraction, then just as in the proof of \cite[Proposition 9.1]{BM13}, we can assume that $(\aa,\bb)>2$. This leaves $(\aa,\bb)=3$ or 4 as the only cases. Next, let us observe that we also have the following equation: 
\[8=\vv^2=(\vv,\aa+\bb)=(\vv,\aa)+(\vv,\bb)= 2(\vv_p,\aa)+2(\vv_p,\bb).\] That is, $(\vv,\aa)$ and $(\vv,\bb)$ must be even integers. In particular, the following inequality \[(\vv,\aa)=(\aa+\bb,\aa)=\aa^2+(\aa,\bb)\geq (\aa,\bb)>2\] implies that $(\vv,\aa)=4$; and hence $(\vv,\bb)=4$ as well. In other words, we are left with two cases when $4=(\vv,\aa)=\aa^2+(\aa,\bb)$:
\begin{itemize}
\item If $(\aa,\bb)=3$ then $\aa^2=1$; contradicting the fact that the lattice is even.
\item If $(\aa,\bb)=4$ then $\aa^2=0$ which implies $\ss:=\vv_p-\aa$ is a spherical class with $(\ss,\vv)=0$. Hence, by Proposition \ref{prop:divisorial-contra}, we must be on a BN-wall which contradicts our assumption that $\cW$ does not induce a divisorial contraction.
\end{itemize}
Thus, writing $\vv$ as the sum of two positive classes $\aa$ and $\bb$ never gives rise to a flopping wall.
\end{proof}

\begin{remark}
Notice that Lemma \ref{noSC1walls} makes our classification of small contractions considerably easier than that of \cite{BM13}. In particular, for a two-divisible Mukai vector, the parallelogram described in \cite[Lemma 9.3]{BM13} always contains an extra interior lattice point $\vv_p$. Even though it is possible to modify the proof of \cite[Lemma 9.3]{BM13} in order to deal with two-divisible Mukai vectors, we will not need it here when $\vv_p^2=2$.
\end{remark}

\begin{lemma}
\label{lem:sc01}
Let $\vv$ be a minimal class of O'Grady type, and assume that $\cW$ does not induce a divisorial contraction for $\vv$. If (SC) does not hold for $\vv$, then $\cW$ is not a wall for $\vv$.
\end{lemma}

\begin{proof}
We use the decomposition of the moduli space $M_{\sigma_\pm}(\vv) = M_{\sigma_\pm}^{\mathrm{st}}(\vv) \sqcup \sym^2 M_{\sigma_\pm}(\vv_p)$. Notice that $\vv_p$ is minimal since $\vv$ is, and so by Lemma \ref{noSC1walls} and \cite[Proposition 9.4]{BM13}, we know that $\cW$ is not a wall for $\vv_p$. In particular, we can identify the strictly semistable loci $\sym^2 M_{\sigma_+}(\vv_p) = \sym^2 M_{\sigma_-}(\vv_p)$. For the stable loci, we observe that the proof of \cite[Proposition 9.4]{BM13} can be applied without change to show that every $\sigma_+$-stable object of class $\vv$ is also $\sigma_0$-stable, and hence $\sigma_-$-stable. (In fact, the only extra point is that any $\sigma_-$-unstable object of class $\vv$ cannot have two HN-factors of class $\vv_p$.) The same is true for $\sigma_-$-stable objects. Therefore, we can also identify the stable loci $M_{\sigma_+}^{\mathrm{st}}(\vv) \cong M_{\sigma_-}^{\mathrm{st}}(\vv)$ and conclude that $\cW$ is not a wall for $\vv$. 
\end{proof}

\begin{lemma}
\label{lem:sc02}
Let $\vv$ be a non-minimal class of O'Grady type, and assume that $\cW$ does not induce a divisorial contraction for $\vv$. If (SC) does not hold for $\vv$, then $\cW$ is a fake wall for $\vv$. Moreover, there exists a derived autoequivalence $\Phi$ of $\cD(X)$, which induces an isomorphism $\Phi: M_{\sigma_+}(\vv) \cong M_{\sigma_-}(\vv)$.
\end{lemma}

\begin{proof}
Let $\vv_0$ be the minimal class in the same $G_\cH$-orbit as $\vv$. Then as in the proof of \cite[Proposition 9.4]{BM13}, we consider the derived autoequivalence $\Phi_+$ of $\cD(X)$ (which is in fact a composition of spherical twists) constructed in \cite[Proposition 6.8]{BM13}. 

On the one hand, \cite[Proposition 9.4]{BM13} shows that $\Phi_+$ induces an isomorphism $\Phi_{+,p}: M_{\sigma_+}(\vv_{0,p}) \xra{\sim} M_{\sigma_+}(\vv_p)$ on the moduli spaces with primitive classes. Since $\Phi_+$ preserves extensions, it also induces an isomorphism $\Phi_+: \sym^2 M_{\sigma_+}(\vv_{0,p}) \xra{\sim} \sym^2 M_{\sigma_+}(\vv_p)$ on the strictly semistable loci of the moduli spaces with non-primitive classes. Moreover, since the S-equivalence relation with respect to $\sigma_0$ is trivial on $M_{\sigma_+}(\vv_{0,p})$, it is also trivial on $M_{\sigma_+}(\vv_p)$, and hence on $\sym^2 M_{\sigma_+}(\vv_p)$, which implies that no curve in $\sym^2 M_{\sigma_+}(\vv_p)$ is contracted by $\pi_+$.

On the other hand, Lemma \ref{lem:sc01} ensures that every $\sigma_+$-stable object of class $\vv_0$ is also $\sigma_0$-stable and so by Lemma \ref{lem:ts3}, we have an injective morphism $\Phi_+: M_{\sigma_+}^{\mathrm{st}}(\vv_0) \to M_{\sigma_+}^{\mathrm{st}}(\vv)$ whose image does not contain any curve contracted by $\pi_+$. Now we combine the stable and strictly semistable loci to get an injective morphism $\Phi_+: M_{\sigma_+}(\vv_0) \to M_{\sigma_+}(\vv)$. Notice that $\Phi_+$ is an isomorphism from $M_{\sigma_+}(\vv_0)$ to its image in $M_{\sigma_+}(\vv)$ because the inverse derived autoequivalence $\Phi_+^{-1}$ induces an inverse of the morphism. Finally, since both moduli spaces are irreducible projective varieties of the same dimension, $\Phi_+$ is in fact an isomorphism between the two moduli spaces. Moreover, by the above discussion, no curve in $M_{\sigma_+}(\vv)$ is contracted by $\pi_+$. By Theorem \ref{thm:1st-classification}, we know that $\cW$ is totally semistable for $\vv$ and so it must be a fake wall.

Similarly we have a derived autoequivalence $\Phi_-$ of $\cD(X)$ inducing an isomorphism $\Phi_-: M_{\sigma_-}(\vv_0) \xra{\sim}M_{\sigma_-}(\vv)$. Since $\cW$ is not a wall for $\vv_0$, the composition $\Phi_- \circ \Phi_+^{-1}: M_{\sigma_+}(\vv) \xra{\sim} M_{\sigma_-}(\vv)$ gives the desired isomorphism between the two moduli spaces. 
\end{proof}

\begin{lemma}
\label{lem:sc12}
Let $\vv$ be a Mukai vector of O'Grady type and assume that $\cW$ does not induce a divisorial contraction for $\vv$. If $\vv$ (or rather $\vv_p$) satisfies the following boundary condition:
\begin{itemize}\itemindent=30pt
\item[(BC)] there exists a spherical class $\ss \in \cH$ with $0 < (\ss, \vv_p) \leq \frac{\vv_p^2}{2}$
\end{itemize}
then $\cW$ induces a small contraction for $\vv$. 
\end{lemma}

\begin{proof}
First let us observe that a similar argument to the one used in the proof of Lemma \ref{noSC1walls} shows that no small contractions can come from a decomposition of $\vv_p$ into positive classes. Now, if (BC) holds, then by \cite[Proposition 9.1]{BM13}, $\cW$ induces a (divisorial or small) contraction for the primitive class $\vv_p$. This implies $\cW$ induces a contraction in the strictly semistable locus $\sym^2 M_{\sigma_+}(\vv_p) \subset M_{\sigma_+}(\vv)$, which is a small contraction considered in $M_{\sigma_+}(\vv)$. 
\end{proof}

\begin{lemma}
\label{lem:sc14}
Let $\vv$ be a Mukai vector of O'Grady type. Assume that $\cW$ does not induce a divisorial contraction for $\vv$, and condition (BC) does not hold for $\vv$. If there exists an effective spherical class $\ss$ with $0<(\ss, \vv)\leq \frac{\vv^2}{2}$, then $\cW$ induces a small contraction for $\vv$.
\end{lemma}

\begin{proof}
We first observe that since (BC) does not hold, we in fact have $(\ss, \vv)=2(\ss, \vv_p)>\vv_p^2$. Next, following the notation in \cite[Proof of Proposition 9.1]{BM13}, we just need to show that $\ext^1(\tilde{S},F)>\ext^1(E_1,E_2)$. This will ensure that all objects obtained by taking extensions of $\tilde{S}$ and $F$ do not fall in a single S-equivalence class with respect to $\sigma_+$. To show the inequality, we observe that $\ext^1(\tilde{S},F) \geq (\ss, \vv-\ss) = (\ss, \vv)+2 > \vv_p^2+2$, while $\ext^1(E_1,E_2) = \vv_p^2+2$ or $\vv_p^2$, depending on whether $E_1$ and $E_2$ are isomorphic or not. In either case, the inequality is satisfied and therefore, the contraction is justified. 
\end{proof}

\begin{lemma}
\label{lem:sc15}
Let $\vv$ be a Mukai vector of O'Grady type. Assume that $\cW$ does not induce a divisorial contraction for $\vv$, and condition (BC) does not hold for $\vv$. If there exists a non-effective spherical class $\ss$ with $0<(\ss, \vv)\leq \frac{\vv^2}{2}$, then $\cW$ induces a small contraction for $\vv$.
\end{lemma}

\begin{proof}
As in Lemma \ref{lem:sc14}, the failure of (BC) implies that $(\ss, \vv)>\vv_p^2$. We set $\tt=-\ss$ (which is effective) and follow the third and fourth paragraphs in the proof of \cite[Proposition 9.1]{BM13}. There are two steps in the proof which require further justification. 

As before, we need to show that all the extensions constructed from a fixed $F$ and $\tilde{T}$ as in \cite[Proposition 9.1]{BM13} cannot lie in the same S-equivalence class with respect to $\sigma_+$ by a dimension comparison. The proof of \cite[Proposition 9.1]{BM13} shows that $(\tt, \vv')=(\ss, \vv)+2$ and all extensions constructed there are parametrised by a Grassmannian with $\dim \gr ((\ss, \vv)+1, \ext^1(\tilde{T},F)) \geq \dim \gr ((\ss, \vv)+1, (\ss, \vv)+2) = (\ss, \vv)+1 > \vv_p^2+1$. On the other hand, the dimension of the S-equivalence class of extensions of $\sigma_+$-stable objects $E_1$ and $E_2$ of class $\vv_p$ is given by $\ext^1(E_1,E_2)-1=\vv_p^2+1$ or $\vv_p^2-1$, depending on whether $E_1$ and $E_2$ are isomorphic or not. In either case, we have the desired inequality and the construction in \cite[Proposition 9.1]{BM13} yields a proper contraction under the given assumption.
\end{proof}

\begin{lemma}
\label{lem:sc21}
Let $\vv$ be any class of O'Grady type and assume that $\cW$ induces a small contraction for $\vv$. Then there is a derived autoequivalence $\Phi$ of $\cD(X)$, which induces a stratum preserving birational map $\Phi: M_{\sigma_+}(\vv) \dashrightarrow M_{\sigma_-}(\vv)$, which is an isomorphism in codimension one. 
\end{lemma}

\begin{proof}
When $\vv$ is a minimal class, we claim that we can simply take $\Phi$ to be the identity functor on $\cD(X)$. Since $\cW$ induces a small contraction, Lemma \ref{lem:dc1} tells us that $M_{\sigma_+}(\vv)$ and $M_{\sigma_-}(\vv)$ contain a common open subset $M_{\sigma_0}^{\mathrm{st}}(\vv)$, with complement of codimension at least two in either moduli space. Therefore, the identify functor induces a birational map from $M_{\sigma_+}(\vv)$ to $M_{\sigma_-}(\vv)$ which is an isomorphism in codimension one. Notice that the assumptions imply that $\cW$ is not a totally semistable wall for the primitive class $\vv_p$ and so the identity functor also induces a birational map between $M_{\sigma_+}(\vv_p)$ and $M_{\sigma_-}(\vv_p)$. By Lemma \ref{lem:stratum}, we conclude that the birational map induced by the identify functor is stratum preserving.

As usual, when $\vv$ is not a minimal class, we denote the corresponding minimal class by $\vv_0$. Let $\Phi_0: M_{\sigma_+}(\vv_0) \dashrightarrow M_{\sigma_-}(\vv_0)$ be the birational map induced by the identity functor on $\cD(X)$. By the discussion above, we know that it is stratum preserving and an isomorphism in codimension one. Now we take $\Phi_+$ to be the composition of spherical twists constructed in \cite[Proposition 6.8]{BM13}. By Lemma \ref{lem:ts3}, $\Phi_+$ induces a birational map $\Phi_+: M_{\sigma_+}(\vv_0) \dashrightarrow M_{\sigma_+}(\vv)$, such that its restriction on the open subset $M_{\sigma_0}^{\mathrm{st}}(\vv_0)$ is an injective morphism. Moreover, since $\cW$ is not totally semistable for $\vv_{0,p}$, $\Phi_+$ also induces a birational morphism $\Phi_{+,p}: M_{\sigma_+}(\vv_{0,p}) \dashrightarrow M_{\sigma_+}(\vv_p)$. Thus, by Lemma \ref{lem:stratum}, the birational map induced by $\Phi_+$ is stratum preserving. Since $M_{\sigma_0}^{\mathrm{st}}(\vv_0)$ has a complement of codimension at least two in $M_{\sigma_+}(\vv_0)$, we conclude that $\Phi_+$ is an isomorphism in codimension one by Proposition \ref{prop:codim-two}. We can follow the same procedure to get a derived autoequivalence $\Phi_-$ which induces a birational map $\Phi_-: M_{\sigma_-}(\vv_0) \dashrightarrow M_{\sigma_-}(\vv)$ satisfying both requirements. Then $\Phi_- \circ \Phi_0 \circ \Phi_+^{-1}$ is the autoequivalence which induces the desired birational map.
\end{proof}

We conclude the discussion by summarising all the lemmas above with the following two propositions; which verify the collection of statements about small contractions and no contractions in Theorems \ref{thm:2nd-classification} and \ref{thm:identify-cones}.

\begin{proposition}
\label{prop:small-contra}
The set $\cW$ is a wall inducing a small contraction if and only if it does not induce any divisorial contraction, and the following condition holds:
\begin{itemize}\itemindent=30pt
\item[(SC)] there exists a spherical class $\ss$ with $0 < (\ss, \vv) \leq \frac{\vv^2}{2}$.
\end{itemize}
Moreover, we can find an autoequivalence $\Phi$ of $\cD(X)$, which induces a birational map $\Phi: M_{\sigma_+}(\vv) \dashrightarrow M_{\sigma_-}(\vv)$. It is stratum preserving and an isomorphism in codimension one.
\end{proposition}

\begin{proof}
The necessity of (SC) is proved in Lemma \ref{lem:sc01} for minimal classes and in Lemma \ref{lem:sc02} for non-minimal classes by contradiction. The sufficiency is proved in Lemma \ref{lem:sc12} when the extra condition (BC) holds and Lemmas \ref{lem:sc14} and \ref{lem:sc15} when (BC) does not hold and (SC) holds. Finally, the birational map is constructed in Lemma \ref{lem:sc21}.
\end{proof}

\begin{proposition}
\label{prop:no-contra}
In the case that $\cW$ does not induce any contractions at all, we can find an autoequivalence $\Phi$ of $\cD(X)$, which induces an isomorphism $\Phi: M_{\sigma_+}(\vv) \to M_{\sigma_-}(\vv)$. In particular, $\Phi$ is just the identity functor when $\vv$ is a minimal class. In other words, $\cW$ is a fake wall if $\vv$ is non-minimal, or is not a wall if $\vv$ is minimal. 
\end{proposition}

\begin{proof}
This is contained in Lemma \ref{lem:sc01} for minimal classes and Lemma \ref{lem:sc02} for non-minimal classes.
\end{proof}

To conclude, we observe that all the cases in Theorems \ref{thm:2nd-classification} and \ref{thm:identify-cones} have been covered by Proposition \ref{prop:divisorial-contra} (in the case of divisorial contractions), Proposition \ref{prop:small-contra} (in the case of small contractions), and Proposition \ref{prop:no-contra} (in the case of no contractions at all). 

\section{The Global \bmm Map}\label{globalbmmap}

In this section we introduce the global \bmm map and show how it can be used to study the birational geometry of the singular moduli spaces. We start by discussing its construction.

\subsection{Construction of the global \bmm map}

We use our Theorems \ref{thm:2nd-classification} and \ref{thm:identify-cones} to glue the local \bmm maps defined on chambers to a global map on $\stab(X)$. The key step is to prove a compatibility result between any two local \bmm maps $\ell_{\cC_+}$ and $\ell_{\cC_-}$ which are defined on adjacent chambers $\cC_+$ and $\cC_-$ separated by a wall $\cW$. The following is a generalisation of \cite[Lemma 10.1]{BM13}. 

\begin{lemma}
\label{lem:boundary-values}
The maps $\ell_{\cC_+}$ and $\ell_{\cC_-}$ agree on the wall $\cW$. More precisely, we have
\begin{itemize}
\item If $\cW$ induces a divisorial contraction, then the analytic continuations of $\ell_{\cC_+}$ and $\ell_{\cC_-}$ differ by the reflection of $\ns (M_{\sigma_+}(\vv))$ (or $\ns (M_{\sigma_-}(\vv))$) at the divisor $D$ contracted by $\ell_{\sigma_0}$;
\item In all other cases, the analytic continuations of $\ell_{\cC_+}$ and $\ell_{\cC_-}$ agree with each other.
\end{itemize}
\end{lemma}

\begin{proof}
Although the Mukai vector $\vv$ is not primitive in our situation, a (quasi-)universal family still exists on the stable locus $M_{\sigma_+}^{\mathrm{st}}(\vv)$ (and $M_{\sigma_+}^{\mathrm{st}}(\vv)$) which has a complement of codimension two. Therefore, the proof for primitive Mukai vectors in \cite[Lemma 10.1]{BM13} still works without any changes. In fact, our situation is even easier since Hilbert-Chow wall do not exist any more. 
\end{proof}

As pointed out in \cite{BM13}, we conclude from this lemma that the moduli spaces $M_{\sigma_\pm}(\vv)$ for the two adjacent chambers are isomorphic when $\cW$ induces a divisorial contraction or no contraction. If $\cW$ induces a small contraction, then these moduli spaces differ by a flop; this is why these walls are called \emph{flopping walls} in \cite{BM12,BM13}.

\begin{remark}
\label{rmk:global-bmm}
By Theorem \ref{thm:identify-cones} and Lemma \ref{lem:boundary-values}, the local \bmm map $\ell_\cC: \cC \to M_\cC(\vv)$ defined on each chamber $\cC \subset \stab(X)$ can be glued together to give a continuous map on $\stab(X)$. By Remark \ref{rmk:unified-notation}, we can denote the global \bmm map by $\ell: \stab(X) \to \ns (M(\vv))$.
\end{remark}

\subsection{Birational geometry via the global \bmm map}

We follow the approach in \cite[Section 10]{BM13} to study the global properties of the \bmm map. This map allows us to prove a precise relationship between the birational geometry of the moduli spaces and wall-crossing in the stability manifold. Before we state the main theorem of the paper, we need three more lemmas; the first of which partly describes the image of the global \bmm map. 

\begin{lemma}
\label{lem:image-bmm}
The image of the global \bmm map is contained in $\bigcone(M(\vv))\, \cap\, \mov(M(\vv))$.
\end{lemma}

\begin{proof}
Let $\sigma \in \stab(X)$ be a generic stability condition with respect to $\vv$. Then Proposition \ref{prop:ampleness} shows that $\ell_\sigma$ is ample on $M_\sigma(\vv)$ and is therefore a big and movable class. By Proposition \ref{old-semiampleness} or Proposition \ref{prop:semiampleness}, $\ell_{\sigma_0}$ is also a big and movable class, because it induces a birational morphism $\pi_+: M_{\sigma_+}(\vv) \to \bar{M}_+$. 
\end{proof}

The second lemma tells us that each open chamber in $\ns(M(\vv))$ which represents the ample cone of a certain birational model is either completely contained in the image of $\ell$, or has no intersection with the image of $\ell$ at all.

\begin{lemma}
\label{lem:full-cone}
For any generic stability condition $\sigma \in \stab(X)$, the ample cone $\amp(M_\sigma(\vv))$ of the moduli space $M_\sigma(\vv)$ is contained in the image of the global \bmm map $\ell$.
\end{lemma}

\begin{proof}
By Proposition \ref{prop:ampleness}, we know that $\ell_\sigma \in \amp(M_\sigma(\vv))$. Hence there is at least one point in $\amp(M_\sigma(\vv))$, which lies in the image of $\ell$. For any other point $\alpha \in \amp(M_\sigma(\vv))$ we can appeal to Proposition \ref{prop:amp-pos} to see that $\alpha \in \pos(M_\sigma(\vv))$. Hence, by Corollary \ref{cor:algebraic-mukai}, we have $\alpha=\theta_\sigma(\aa)$ for some class $\aa \in \vv^\perp$ with $\aa^2>0$. Now we use the argument in \cite[Proof of Theorem 1.2 (a)(b)(c)]{BM13} to conclude that $\alpha$ also lies in the image of $\ell$.
\end{proof}

Now we state the third and final lemma. If a wall $\cW$ of a chamber $\cC$ in $\stab(X)$ induces a contraction of the moduli space $M_\cC(\vv)$, then the image $\ell(\cW) \subset \ns(M(\vv))$ of $\cW$ under the global \bmm map $\ell: \stab(X) \to \ns(M(\vv))$ is a wall of the nef cone $\nef(M_\cC(\vv))$. However, if $\cW$ is a fake wall, then its image under the global \bmm map is not a wall in $\ns (M(\vv))$. This is the content of the following result.

\begin{lemma}
\label{lem:fake-replacement}
Let $\sigma_0$ be a generic stability condition on a fake wall $\cW$. Then its image lies in the interior of the nef cone $\nef(M_{\sigma_+}(\vv))$ of the moduli space $M_{\sigma_+}(\vv)$. The same statement is true for $M_{\sigma_-}(\vv)$.
\end{lemma}

\begin{proof}
This is just \cite[Proposition 10.3]{BM13}; whose proof works regardless of whether $\vv$ is primitive or not. 
\end{proof}

We can now state the main theorem of the paper, which crystallises the relation between the birational geometry of singular moduli spaces of O'Grady type and wall crossings in the stability manifold $\stab(X)$. It is a slight generalisation of \cite[Theorem 1.2]{BM13} in the case of singular moduli spaces which admit symplectic resolutions. 

\begin{theorem}
\label{thm:every-model}
Let $\vv$ be a Mukai vector of O'Grady type and $\sigma \in \stab(X)$ a generic stability condition with respect to $\vv$. Then
\begin{enumerate}
\item We have a globally defined continuous \bmm map $ \ell: \stab(X) \to \ns (M_\sigma(\vv)) $, which is independent of the choice of $\sigma$. Moreover, for any generic stability condition $\tau \in \stab(X)$, the moduli space $M_\tau(\vv)$ is the birational model corresponding to $\ell_\tau$. \label{item:main1}
\item If $\cC \subset \stab(X)$ is the open chamber containing $\sigma$, then $\ell(\cC)=\amp(M_\sigma(\vv))$.\label{item:main2}
\item The image of $\ell$ is equal to $\bigcone(M_\sigma(\vv)) \, \cap \, \mov(M_\sigma(\vv))$. In particular, every $K$-trivial $\bQ$-factorial birational model of $M_\sigma(\vv)$ which is isomorphic to $M_\sigma(\vv)$ in codimension $1$ appears as a moduli space $M_\tau(\vv)$ for some generic stability condition $\tau \in \stab(X)$. \label{item:main3}
\end{enumerate}
\end{theorem}

\begin{proof}
All the ingredients of the proof have already been presented above. For \eqref{item:main1}, the existence of the global \bmm map is contained in Remark \ref{rmk:global-bmm}. The ampleness statement in Proposition \ref{prop:ampleness} implies that $M_\tau(\vv)$ is the birational model corresponding to $\ell_\tau$ for any generic $\tau$. 

For \eqref{item:main2}, we first realise that $\ell(\cC)$ has full dimension by Lemma \ref{lem:full-cone} and hence must be an open subset of $\amp(M_\sigma(\vv))$ by Proposition \ref{prop:ampleness}. Now we know that the image of every non-fake wall of $\cC$ has to be a boundary component of $\amp(M_\sigma(\vv))$ whereas Lemma \ref{lem:fake-replacement} ensures that the image of a fake wall of $\cC$ never separates $\amp(M_\sigma(\vv))$ into two parts. Therefore, the image of $\cC$ under $\ell$ is the whole ample cone. 

For the first claim in \eqref{item:main3}, one direction of inclusion is in Lemma \ref{lem:image-bmm}. We just need to show that every open chamber of $\bigcone(M_\sigma(\vv))\, \cap\, \mov(M_\sigma(\vv))$ regarded as the ample cone of some birational model of $M_\sigma(\vv)$ is contained in the image of $\ell$. Assume we have two such ample cones adjacent to each other, one of which lies in the image of $\ell$ but not the other. Then the wall separating the two ample cones must be the image $\ell(\cW)$ of a wall $\cW$ in $\stab(X)$. By Lemma \ref{lem:fake-replacement}, $\cW$ cannot be a fake wall since fake walls are not mapped to walls in the movable cone by $\ell$. By Lemma \ref{lem:boundary-values}, $\cW$ cannot be a flopping wall since the image of $\ell$ would extend to the other side of $\ell(\cW)$. Hence $\cW$ must be a divisorial wall. Then we claim that the image of $\cW$ has reached the boundary of $\mov(M_\sigma(\vv))$. In fact, since $\ell_{\sigma_0}$ induces a divisorial contraction for a generic $\sigma_0 \in \cW$, $\ell_{\sigma_0}$ has degree zero on each curve contained in the fibres of this contraction, which implies that any line bundle on the other side of $\ell(\cW)$ must have negative degree on these curves contained in fibres. Therefore, its base locus contains the entire contracted divisor. The second claim follows immediately from the surjectivity.
\end{proof}

\begin{remark}
Careful readers might have found that our proof of surjectivity of the \bmm map onto the intersection of the big and movable cones is slightly different from that in \cite[Proof of Theorem 1.2 (b)]{BM13}. More precisely, we avoided using the statement that any big and movable class is strictly positive, which is well-known for irreducible holomorphic symplectic manifolds, but not for moduli spaces of O'Grady type. However, using \cite[Theorem 1.1]{Bri08}, our Theorem \ref{thm:every-model} implies that the same statement is still true for singular moduli spaces of O'Grady type.
\end{remark}

\subsection{Torelli theorem for singular moduli spaces of O'Grady type}

In analogy with \cite[Corollary 1.3]{BM13}, we can prove a Torelli-type theorem for the singular moduli spaces of O'Grady type. It gives a Hodge-theoretic criterion for the existence of stratum-preserving birational maps between these moduli spaces. We use the same notation as in \cite{BM13}, which was originally introduced by Mukai: the total cohomology $H^*(X, \bZ)$ of a K3 surface $X$ carries a weight two Hodge structure which is polarised by the Mukai pairing. We write $\vv^{\perp,\mathrm{tr}} \subset H^*(X, \bZ)$ for the orthogonal complement of $\vv$ in the total cohomology.

\renewcommand{\labelenumi}{(\alph{enumi})}

\begin{corollary}
\label{cor:torelli}
Let $X$ and $X'$ be two smooth projective K3 surfaces with $\vv \in H^*_{\mathrm{alg}}(X, \bZ)$ and $\vv' \in H^*_{\mathrm{alg}}(X', \bZ)$ Mukai vectors of O'Grady type. If $\sigma \in \stab(X)$ and $\sigma' \in \stab(X')$ are generic stability conditions with respect to $\vv$ and $\vv'$ respectively, then the following statements are equivalent
\begin{enumerate}
\item There is a stratum preserving birational map $\varphi: M_{X,\sigma}(\vv) \dashrightarrow M_{X',\sigma'}(\vv')$. 
\item The embedding $\vv^{\perp,\mathrm{tr}} \subset H^*(X, \bZ)$ of the integral weight two Hodge structures is isomorphic to the embedding $\vv'^{\perp,\mathrm{tr}} \subset H^*(X', \bZ)$. 
\item There is an (anti-)autoequivalence $\Phi$ from $\cD(X)$ to $\cD(X')$ with $\Phi_*(\vv)=\vv'$. 
\item There is an (anti-)autoequivalence $\Psi$ from $\cD(X)$ to $\cD(X')$ with $\Psi_*(\vv)=\vv'$ which induces a stratum preserving birational map $\psi: M_{X,\sigma}(\vv) \dashrightarrow M_{X',\sigma'}(\vv')$. 
\end{enumerate}
\end{corollary}

\begin{proof}
The proof of \cite[Corollary 1.3]{BM13} already contains most of what we need in our situation. We only point out the differences. For (a) $\Rightarrow$ (b), we observe that $M_{X,\sigma}(\vv_p) \subset M_{X,\sigma}(\vv)$ and $M_{X',\sigma'}(\vv'_p) \subset M_{X',\sigma'}(\vv')$ are the deepest singular strata and so the stratum preserving birational map $\varphi$ in (a) restricts to a birational map $M_{\sigma}(\vv_p) \dashrightarrow M_{\sigma'}(\vv'_p)$. Therefore, by \cite[Corollary 1.3]{BM13}, we can deduce (b). For (b) $\Rightarrow$ (c), the proof is already in \cite[Corollary 1.3]{BM13}. For (c) $\Rightarrow$ (d), we can follow the proof of \cite[Corollary 1.3]{BM13} and replace $\Phi$ if necessary, so that there exists a generic $\tau \in \stab(X')$ with $M_{\sigma}(\vv) \cong M_{\tau}(\vv')$. By Theorem \ref{thm:identify-cones}, we can find an (anti-)autoequivalence $\Phi'$ of $\cD(X')$ which induces a stratum preserving birational map $M_{\tau}(\vv') \dashrightarrow M_{\sigma'}(\vv')$. And the composition $\Phi' \circ \Phi$ does the work. The final part (d) $\Rightarrow$ (a) is trivial. 
\end{proof}

\subsection{Lagrangian Fibrations}

An immediate consequence of the birationality of wall-crossing is the following result:

\begin{theorem}\label{Lagrangian}
Let $M=M_\sigma(\vv)$ be a moduli space of O'Grady type, $\pi:\tilde{M}\to M$ its symplectic resolution and $q(-,-)$ its Beauville-Bogomolov form.  
If there is an integral divisor class $D$ with $q(D)=0$ then there exists a birational moduli space $M'=M_{\sigma'}(\vv')$ of O'Grady type whose resolution $\tilde{M'}$ admits a Lagrangian fibration.
\end{theorem}

\begin{proof}
This follows from the arguments in \cite[Section 11]{BM13}. Indeed, we just have to replace \cite[Theorem 1.2]{BM13} with our Theorem \ref{thm:every-model} in the proof of \cite[Theorem 1.5]{BM13}.
\end{proof}

\begin{remark}
Bayer and Macr\`i are able to make a stronger statement in their situation if the divisor class is also assumed to be nef; see \cite[Conjecture 1.4(b)]{BM13}. However, for an analogous statement to hold true for moduli spaces of O'Grady type, we would need to find a replacement for \cite[Proposition 3.3]{BM13} (which is a summary of Markman's results). In particular, we would need to know that the cone $\mov(M)\cap\pos(M)$ of big and movable divisors on $M$ was equal to the fundamental chamber of the Weyl group action on the positive cone $\pos(M)$ of $M$. 
\end{remark}

\begin{remark}
There might be square-zero classes on $\tilde{M}$ which are not the pullback of square-zero classes on $M$ and so it is not clear whether the converse of Theorem \ref{Lagrangian} is true. In particular, we see no reason why the existence of a Lagrangian fibration $\tilde{M}\to\bP^5$ should guarantee the existence of a square-zero class on $M$.
\end{remark}

\begin{remark}
A natural question at this point is whether one can combine the recent techniques of Bayer-Hassett-Tschinkel \cite{BHT13} and Markman \cite{Mar11,Mar13} with similar arguments of \cite[Section 12]{BM13} to obtain a description of the Mori cone of the symplectic resolution $\tilde{M}$ of a moduli space $M=M_\sigma(\vv)$ of O'Grady type. The authors plan to return to this question in future work.
\end{remark}

\section{Examples of Movable and Nef cones}
In this section, we examine examples of cones of divisors on moduli spaces of O'Grady type.

Let us suppose for simplicity that $\Pic(X)=\bZ[H]$ with $H^2=2d$. We set $\vv=2\vv_p=(2,0,-2)$ and $M:=M_H(\vv)$. Then a basis for $\ns(M)$ is given by 
\[\widetilde{H}=\theta_\sigma(0,-H,0) \quad\trm{and}\quad B=\theta_\sigma(-1,0,-1)\]
where 
$\theta_\sigma:\vv^\perp\xra\sim\ns(M)$ is the isometry from Corollary \ref{cor:algebraic-mukai}(3).

By Theorem \ref{thm:2nd-classification}, divisorial contractions are divided into two cases:

\begin{itemize}
\item[(BN)] there exists a spherical class $\ss$ with $(\ss,\vv)=0$, or 
\item[(LGU)] there exists an isotropic class $\ww$ with $(\ww,\vv)=2$.
\end{itemize}

Just as in \cite[Section 13]{BM13}, we can solve the following system of equations 
\[\{\ss^2=-2, (\ss,\vv)=0\;:\;\ss=(r,cH,s)\}\]
to see that the case of BN-contractions are governed by solutions to the following Pell's equation:
\begin{equation}\label{BNpell}
x^2-dy^2=1\quad\trm{via}\quad\ss=(x,-yH,x).
\end{equation}
Similarly, we can solve
\[\{\ww^2=0, (\ww,\vv)=2\;:\;\ww=(r,cH,s)\}\]
to see that the case of LGU-contractions are governed by solutions to
\begin{equation}\label{LGUpell}
x^2-dy^2=1\quad\trm{via}\quad\ww=\left(\frac{x+1}{2},-\frac{yH}{2},\frac{x-1}{2}\right)\quad\trm{or}\quad\ww=(x+1,-yH,x)
\end{equation}
depending on whether $y$ is even or odd. The two equations determine the movable cone:

\begin{proposition}\label{movcone}
Assume $\Pic(X)=\bZ[H]$ with $H^2=2d$. The movable cone of $M=M_H(2,0,-2)$ has the following form:
\begin{itemize}
\item[(a)] If $d=\frac{k^2}{h^2}$, with $k,h\geq1$, $\gcd(k,h)=1$, then 
\[\mov(M)=\langle\tilde{H},\tilde{H}-\frac{k}{h}B\rangle,\]
where $q(h\tilde{H}-kB)=0$.
\item[(b)] If $d$ is not a perfect square, and \eqref{BNpell} has a solution, then 
\[\mov(M)=\langle\tilde{H},\tilde{H}-d\frac{y_1}{x_1}B\rangle,\]
where $(x_1,y_1)$ is the solution to \eqref{BNpell} with $x_1,y_1>0$, and with smallest possible $x_1$.
\item[(c)] If $d$ is not a perfect square, and \eqref{BNpell} has no solution, then 
\[\mov(M)=\langle\tilde{H},\tilde{H}-d\frac{y'_1}{x'_1}B\rangle,\]
where $(x'_1,y'_1)$ is the solution to \eqref{LGUpell} with smallest possible $\frac{y'_1}{x'_1}>0$.
\end{itemize}
\end{proposition}

\begin{proof}
Setting $n=2$ in the proof of \cite[Proposition 13.1]{BM13} is sufficient here. Indeed, their proof relies on \cite[Theorem 12.3]{BM13} whose proof also goes through in our case when we replace \cite[Theorem 1.2]{BM13} with Theorem \ref{thm:every-model}. 
\end{proof}

\begin{example}
If $d=1$ or 2, then we are in case (a) or (b) of Proposition \ref{movcone} and we have
\[\mov(M)=\langle\tilde{H},\tilde{H}-B\rangle\quad\trm{or}\quad\mov(M)=\langle\tilde{H},\tilde{H}-\frac{4}{3}B\rangle\quad\trm{respectively}.\]
\end{example}

By Theorem \ref{thm:2nd-classification}, we have a flopping wall if and only if:

\begin{itemize}
\item[(SC)] there exists a spherical class $\ss$ with $0<(\ss,\vv)\leq \frac{\vv^2}{2}=4$.
\end{itemize}

Notice that the condition becomes $0<(\ss,\vv)=2(\ss,\vv_p)\leq 4$ and so there is only enough room for two walls of this kind: either $(\ss,\vv)=2$ or 4. Solving the following system of equations 
\[\{\ss^2=-2, (\ss,\vv)=2\trm{ or }4\;:\;\ss=(r,cH,s)\}\]
shows that flopping walls of type (SC2) are governed by solutions of
\begin{equation}\label{SC2pell}
x^2-4dy^2=5\quad\trm{and}\quad x^2-dy^2=2\quad\trm{respectively.}
\end{equation}
The associated spherical classes are $\ss=\left(\frac{x+1}{2},-yH,\frac{x-1}{2}\right)$ and $\ss=(x+1,-yH,x-1)$ respectively.

\begin{lemma}
Let $M=M_H(2,0,-2)$. The nef cone of $M$ has the following form:
\begin{itemize}
\item[(a)] If \eqref{SC2pell} has no solutions, then 
\[\nef(M)=\mov(M).\]
\item[(b)] Otherwise, let $(x_1,y_1)$ be the positive solution of \eqref{SC2pell} with $x_1$ minimal. Then 
\[\nef(M)=\langle\tilde{H},\tilde{H}-2d\frac{y_1}{x_1}B\rangle.\]
where $(x_1,y_1)$ is the solution to \eqref{BNpell} with $x_1,y_1>0$, and with smallest possible $x_1$.
\end{itemize}
\end{lemma}

\begin{proof}
We apply \cite[Theorem 12.1]{BM13} whose proof also works in our case once we replace \cite[Theorems 1.2 and 5.7]{BM13} with Theorems \ref{thm:every-model} and \ref{thm:2nd-classification} respectively. The movable cone and the nef cone agree unless there is a flopping wall, described in Theorem \ref{thm:2nd-classification}. By Lemma \ref{noSC1walls}, we know that the case $\vv = \aa + \bb$ with $\aa, \bb$ positive is impossible. This leaves only the case of a spherical class $\ss$ with $(\ss,\vv) = 2$ or $4$; these exist if and only if \eqref{SC2pell} has a solution.
\end{proof}

\begin{example}\label{exasdeg1and2}
If $d=1$, then $x^2-y^2=2$ has no solutions and the only positive solution of $x^2-4y^2=5$ is given by $(x,y)=(3,1)$ which gives rise to the spherical class $\ss=(2,-H,1)$. In order to compute the generators of the nef cone of $M$, we need an element of $\vv^\perp\cap\ss^\perp$. For example, $(2,-3H,2)$ is perpendicular to both $\vv$ and $\ss$ and can be expressed in our basis for $\ns(M)$ as $3\widetilde{H}-2B$ or equivalently $\widetilde{H}-\frac{2}{3}B$. In particular, we have shown that 
\[\nef(M)=\langle\tilde{H},\tilde{H}-\frac{2}{3}B\rangle.\]
Similarly, if $d=2$ then $x^2-8y^2=5$ has no solutions and the minimal solution of $x^2-2y^2=2$ is given by $(x,y)=(2,1)$. Therefore,
\[\nef(M)=\langle\tilde{H},\tilde{H}-B\rangle.\]
\end{example}

\begin{remark}
Observe that when $(\vv,\ss)=4$, we have another spherical class $\tt:=\vv-\ss$ with $(\ss,\tt)=6$ and so this wall is flopping the natural Lagrangian $\bP^5\simeq\Ext^1(S,T)$ inside $M_{\sigma_0}(\vv)$ 
\[\xymatrix{\bP^{5}\simeq\bP\Ext^1(S,T) \ar@{^{(}->}[r] \ar[d] & M_{\sigma_0}(\vv)\\ M_{\sigma_0}(\vv-\ss)\simeq M_{\sigma_0}(\tt)\simeq\{\pt\}.}\]
However, when $(\vv,\ss)=2$ we have another spherical class given by $\tt:=\vv_p-\ss$ with $(\ss,\tt)=3$. In particular, $(\tt,\vv-\ss)<0$ and so by Theorem \ref{thm:1st-classification}, the wall is totally semistable for $\vv-\ss$; similar arguments show that the wall is totally semistable for $\vv-\tt$ as well. This means the flopping locus cannot be described as a projective bundle.
\end{remark}

\begin{example}
If $X$ is a genus two K3 surface with $\Pic(X)=\bZ[H]$ and $H:=[C]$ the class of a genus two curve $C$, then we can twist our calculations in Example \ref{exasdeg1and2} by $\cO_X(C)$ to get a description for the nef cone of $M_H(2,2H,0)$. Moreover, since $d=1$ in this case, we know that there is only one flopping wall corresponding to the spherical class $\ss=(2,H,1)$. That is, the movable cone has two chambers and hence $M_\sigma(\vv)$ has two `birational models' which are exchanged by the birational map $\Phi_*:M_\sigma(2,2H,0)\dashrightarrow M_{\Phi(\sigma)}(0,2H,-2)\;;\;\cE\mapsto \Phi(\cE)$ induced by the spherical twist around $\cO_X$. 

To see that the two models really are interchanged by the spherical twist $\Phi$, let $\frh$ be the upper half plane, $\beta,\omega\in\ns(X)_\bR$ with $\omega$ ample and consider the open subset
\[\frh^0=\frh\setminus\{\beta+i\omega\in\frh\;:\;\langle\exp(\beta+i\omega),\ss\rangle=0\trm{ when $\ss$ is a $(-2)$-class}\}.\]
If $X$ has Picard number one then $\stab(X)/\tilde{\mrm{GL}_2^+}(\bR)$ can be identified with (the universal cover of) $\frh^0$; see \cite{BB13}. In particular, if $X$ is a generic double cover of $\bP^2$ with $H^2=2$ and $d=1$, then we can set $\beta:=sH$ and $\omega:=tH$ for $(s,t)\in\bR\times\bR_{>0}$ and use the formula in \cite[p.35]{Mea12} 
to see that the flopping wall of $M_\sigma(2,2H,0)$ in $\frh^0$ is given by a semicircle with centre and radius equal to $(-1/2,\sqrt{5}/2)$. Furthermore, the cohomological transform $\Phi^H:H^*(X,\bQ)\to H^*(X,\bQ)$ implies that $\Phi:\sigma_{0,tH}\mapsto\sigma_{0,\frac{1}{t}H}$ and so we see that the non-Gieseker chamber with respect to $\vv=(2,2H,0)$ gets identified with the Gieseker chamber with respect to $\Phi(\vv)=(0,2H,-2)$. If we set $\tilde{H}':=\theta_\sigma(0,-H,-2)$ and $B':=\theta_\sigma(-1,-H,-2)$ (to be the classes $\tilde{H}$ and $B$ twisted by $\cO_X(C)$) then we can illustrate our observations as in Figure \ref{ogrcone}.

\begin{figure}[htbp!]
\begin{center}
\begin{tikzpicture}[scale=1]
\draw [style=dashed,->] (-8.1,0) -- (4.4,0)  node [anchor=south] {$s$};
\draw [->] (0,-0.1) -- (0,5.3) node [anchor=west] {$t$};
\node[anchor=north east] at (0,4) {$1$};
\node[anchor=north] at (4,0) {$1$};
\node[anchor=north] at (2,0) {$0.5$};
\node[anchor=north] at (0,0) {$0$};
\node[anchor=north] at (-2,0) {$-0.5$};
\node[anchor=north] at (-4,0) {$-1$};
\node[anchor=north] at (-6,0) {$-1.5$};
\node[anchor=north] at (-8,0) {$-2$};
\clip (-8.2,-0.3) rectangle (5,5.2);
\draw [color=red,thick] (4,0) -- (4,5);
\draw [thick,domain=0:180] plot ({-2+4.47*cos(\x)}, {4.47*sin(\x)});
\draw [->] (1,3.7)[thick] arc (100:170:0.5);
\node[anchor=west] at (1,3.7) {flop};
\node at (-2,3.5) {$M_\sigma(2,2H,0)$};
\node at (-2,3) {$\|$};
\node at (-2,2.5) {$M_H(0,2H,-2)$};
\node at (2.4,4.8) {$M_H(2,2H,0)$};

\filldraw[blue] (4,4) circle(1.5pt);
\draw [color=blue] (4,4) -- (4,0);
\filldraw[blue] (0,4) circle(1.5pt);
\draw [color=blue] (0,4) -- (0,0);
\filldraw[blue] (-4,4) circle(1.5pt);
\draw [color=blue] (-4,4) -- (-4,0);
\filldraw[blue] (-8,4) circle(1.5pt);
\draw [color=blue] (-8,4) -- (-8,0);

\filldraw[blue] (2,2) circle(1.5pt);
\draw [color=blue] (2,2) -- (2,0);
\filldraw[blue] (-2,2) circle(1.5pt);
\draw [color=blue] (-2,2) -- (-2,0);
\filldraw[blue] (-6,2) circle(1.5pt);
\draw [color=blue] (-6,2) -- (-6,0);

\filldraw[blue] (1.6,0.8) circle(1.5pt);
\draw [color=blue] (1.6,0.8) -- (1.6,0);
\filldraw[blue] (2.4,0.8) circle(1.5pt);
\draw [color=blue] (2.4,0.8) -- (2.4,0);
\filldraw[blue] (-1.6,0.8) circle(1.5pt);
\draw [color=blue] (-1.6,0.8) -- (-1.6,0);
\filldraw[blue] (-2.4,0.8) circle(1.5pt);
\draw [color=blue] (-2.4,0.8) -- (-2.4,0);
\filldraw[blue] (-5.6,0.8) circle(1.5pt);
\draw [color=blue] (-5.6,0.8) -- (-5.6,0);
\filldraw[blue] (-6.4,0.8) circle(1.5pt);
\draw [color=blue] (-6.4,0.8) -- (-6.4,0);

\filldraw[blue] (1.2,0.4) circle(1.5pt);
\draw [color=blue] (1.2,0.4) -- (1.2,0);
\filldraw[blue] (2.8,0.4) circle(1.5pt);
\draw [color=blue] (2.8,0.4) -- (2.8,0);
\filldraw[blue] (-1.2,0.4) circle(1.5pt);
\draw [color=blue] (-1.2,0.4) -- (-1.2,0);
\filldraw[blue] (-2.8,0.4) circle(1.5pt);
\draw [color=blue] (-2.8,0.4) -- (-2.8,0);
\filldraw[blue] (-6.8,0.4) circle(1.5pt);
\draw [color=blue] (-6.8,0.4) -- (-6.8,0);
\filldraw[blue] (-5.2,0.4) circle(1.5pt);
\draw [color=blue] (-5.2,0.4) -- (-5.2,0);
\end{tikzpicture}
\end{center}
\vs

\begin{center}
\begin{tikzpicture}[scale=1]
\draw [->][style=dashed,thick] (0,0) -- (-4,5);
\draw [->][thick] (0,0) -- (0,5);
\draw [->][color=red,thick] (0,0) -- (4,5);
\draw [->](0.33,3)[thick] arc (50:130:0.5);
\node[anchor=north] at (0.35,3) {flop};
\node at (1.8,4.5) {$\nef(M_H(2,2H,0))$};
\node at (-1.8,4.5) {$\nef(M_H(0,2H,-2))$};
\node at (3,3) {$\tilde{H}'$}; \node at (3.1,2) {BN=LGU wall}; 
\node at (-3.3,3) {$\widetilde{H}'-B'$};
\node at (-0.1,1.5) {$\widetilde{H}'-\frac{2}{3}B'$};
\end{tikzpicture}
\end{center}

\[\mov(M_\sigma(2,2H,0))=\nef(M_H(2,2H,0))\sqcup\nef(M_H(0,2H,-2))\]

\caption{The movable cone of $M_\sigma(2,2H,0)$ when $X$ is a genus two K3 surface.}
\label{ogrcone}
\end{figure}

The LGU wall and hence BN wall (Lemma \ref{lem:dc0}) can be realised geometrically by considering the object $\cE:=\Phi(\cI_y(-2C))[1]$ for some point $y\in X$ with $\vv(\cE)=\Phi(1,-2H,2)[1]=(2,2H,1)$ and then splitting off another point $x\in X$. More precisely, turning the torsion sequence $\cF\to \cE\to \cO_x$ 
for $\cF\in M_H(2,2H,0)$ yields
\[\cO_x\to \cF[1]\to \cE[1].\] This is Mukai's morphism $\theta_v$ which contracts the projective space $\bP\Ext^1(\cE[1],\cO_x)\simeq\bP^1$. 
That is, this wall will contract a divisor of non-locally free sheaves $\cF$, and the short exact sequence is (generically) induced by injection into a locally free sheaf $\cE$.
\end{example}

\begin{remark}
The only difference with \cite{BM13} is that we cannot say the extremal ray $\widetilde{H}-B$ gives rise to a Lagrangian fibration because $M_\sigma(0,2H,-2)$ is singular; we would need to precompose with the resolution $\pi:\widetilde{M}\to M$ to get this.
\end{remark}

\begin{remark}
Recall that the six dimensional O'Grady space \cite{OG03} is constructed in a very similar way to his ten dimensional space. Indeed, he considers a moduli space of sheaves on an abelian surface, rather than a K3 surface, and takes the fibre of the Albanese map over zero. Since there are no spherical objects on an abelian surface \cite[Lemma 15.1]{Bri08}, we expect that the analogue of our classification of walls Theorem \ref{thm:2nd-classification} would become a lot simpler in this case. In particular, Yoshioka \cite{Yos12} has already shown that all of the machinery in \cite{BM12,BM13} works on abelian surfaces and so if we assume our arguments also work in this case, then we would have to conclude that there are no flopping walls at all; c.f. \cite[Proposition 4.6]{Mac12}. That is, the movable cone $\mov(M_\sigma(\vv))$ and the nef cone $\nef(M_\sigma(\vv))$ would actually coincide in this case. It seems rather surprising that the six dimensional O'Grady space $M_\sigma(\vv)$ should only have \emph{one} birational model but this remark should only be regarded as speculation.
\end{remark}

\bibliographystyle{alpha}
\bibliography{OGradyRef_16Jan}

\begin{thebibliography}{BHPVdV04}

\bibitem[AB13]{AB13}
Daniele Arcara and Aaron Bertram.
\newblock Bridgeland-stable moduli spaces for {$K$}-trivial surfaces.
\newblock {\em J. Eur. Math. Soc. (JEMS)}, 15(1):1--38, 2013.
\newblock With an appendix by Max Lieblich.
  \href{http://arxiv.org/abs/0708.2247}{arXiv:0708.2247}.

\bibitem[BB13]{BB13}
Arend Bayer and Tom Bridgeland.
\newblock Derived automorphism groups of {K}3 surfaces of {P}icard rank 1.
\newblock {\em Arxiv preprint}, 2013.
\newblock \href{http://arxiv.org/abs/1310.8266}{arXiv:1310.8266}.

\bibitem[BCHM10]{BCHM10}
Caucher Birkar, Paolo Cascini, Christopher~D. Hacon, and James McKernan.
\newblock Existence of minimal models for varieties of log general type.
\newblock {\em J. Amer. Math. Soc.}, 23(2):405--468, 2010.
\newblock \href{http://arxiv.org/abs/math/0610203}{arXiv:math/0610203}.

\bibitem[Bea00]{Bea00}
Arnaud Beauville.
\newblock Symplectic singularities.
\newblock {\em Invent. Math.}, 139(3):541--549, 2000.
\newblock \href{http://arxiv.org/abs/math/9903070}{arXiv:math/9903070}.

\bibitem[BHPVdV04]{BHPV}
Wolf Barth, Klaus Hulek, Chris Peters, and Antonius Van~de Ven.
\newblock {\em Compact complex surfaces}, volume~4 of {\em Ergebnisse der
  Mathematik und ihrer Grenzgebiete. 3. Folge. A Series of Modern Surveys in
  Mathematics [Results in Mathematics and Related Areas. 3rd Series. A Series
  of Modern Surveys in Mathematics]}.
\newblock Springer-Verlag, Berlin, second edition, 2004.

\bibitem[BHT13]{BHT13}
Arend Bayer, Brendan Hassett, and Yuri Tschinkel.
\newblock Mori cones of holomorphic symplectic varieties of {K}3 type.
\newblock {\em Arxiv preprint}, 2013.
\newblock \href{http://arxiv.org/abs/1307.2291}{arXiv:1307.2291}.

\bibitem[BM14a]{BM13}
Arend Bayer and Emanuele Macr{\`{\i}}.
\newblock M{MP} for moduli of sheaves on {K}3s via wall-crossing: nef and
  movable cones, {L}agrangian fibrations.
\newblock {\em Invent. Math.}, 198(3):505--590, 2014.
\newblock \href{http://arxiv.org/abs/1301.6968}{arXiv:1301.6968v3}.

\bibitem[BM14b]{BM12}
Arend Bayer and Emanuele Macr{\`{\i}}.
\newblock Projectivity and birational geometry of {B}ridgeland moduli spaces.
\newblock {\em J. Amer. Math. Soc.}, 27(3):707--752, 2014.
\newblock \href{http://arxiv.org/abs/1203.4613}{arXiv:1203.4613v2}.

\bibitem[Bri08]{Bri08}
Tom Bridgeland.
\newblock Stability conditions on {$K3$} surfaces.
\newblock {\em Duke Math. J.}, 141(2):241--291, 2008.
\newblock \href{http://arxiv.org/abs/math/0307164}{arXiv:math/0307164}.

\bibitem[C{\u{a}}l00]{Cal00}
Andrei C{\u{a}}ld{\u{a}}raru.
\newblock Derived categories of twisted sheaves on {C}alabi-{Y}au manifolds.
\newblock {\em PhD Thesis, Cornell University}, 2000.
\newblock \textit{Available online at
  }\url{http://www.math.wisc.edu/~andreic/publications/ThesisSingleSpaced.pdf}.

\bibitem[GHJ03]{GHJ03}
Mark Gross, Daniel Huybrechts, and Dominic Joyce.
\newblock {\em Calabi-{Y}au manifolds and related geometries}.
\newblock Universitext. Springer-Verlag, Berlin, 2003.
\newblock Lectures from the Summer School held in Nordfjordeid, June 2001.

\bibitem[HL10]{HuLe10a}
Daniel Huybrechts and Manfred Lehn.
\newblock {\em The geometry of moduli spaces of sheaves}.
\newblock Cambridge Mathematical Library. Cambridge University Press,
  Cambridge, second edition, 2010.

\bibitem[HS05]{HS05}
Daniel Huybrechts and Paolo Stellari.
\newblock Equivalences of twisted {$K3$} surfaces.
\newblock {\em Math. Ann.}, 332(4):901--936, 2005.
\newblock \href{http://arxiv.org/abs/math/0409030}{arXiv:math/0409030}.

\bibitem[HT09]{HT09}
Brendan Hassett and Yuri Tschinkel.
\newblock Moving and ample cones of holomorphic symplectic fourfolds.
\newblock {\em Geom. Funct. Anal.}, 19(4):1065--1080, 2009.
\newblock \href{http://arxiv.org/abs/0710.0390}{arXiv:0710.0390}.

\bibitem[Huy99]{Hu99a}
Daniel Huybrechts.
\newblock Compact hyper-{K}{\"a}hler manifolds: basic results.
\newblock {\em Invent. Math.}, 135(1):63--113, 1999.
\newblock \href{http://arxiv.org/abs/alg-geom/9705025}{arXiv:alg-geom/9705025}.

\bibitem[Huy03]{Huy03}
Daniel Huybrechts.
\newblock The {K}{\"a}hler cone of a compact hyperk{\"a}hler manifold.
\newblock {\em Math. Ann.}, 326(3):499--513, 2003.
\newblock \href{http://arxiv.org/abs/math/9909109}{arXiv:math/9909109}.

\bibitem[Kal06]{Kal06}
Dmitry Kaledin.
\newblock Symplectic singularities from the {P}oisson point of view.
\newblock {\em J. Reine Angew. Math.}, 600:135--156, 2006.
\newblock \href{http://arxiv.org/abs/math/0310186}{arXiv:math/0310186}.

\bibitem[KLS06]{KaLeSo06a}
Dmitry Kaledin, Manfred Lehn, and Christoph Sorger.
\newblock Singular symplectic moduli spaces.
\newblock {\em Invent. Math.}, 164(3):591--614, 2006.
\newblock \href{http://arxiv.org/abs/math/0504202}{arXiv:math/0504202}.

\bibitem[KM98]{KoMo98a}
J{{\'a}}nos Koll{{\'a}}r and Shigefumi Mori.
\newblock {\em Birational geometry of algebraic varieties}, volume 134 of {\em
  Cambridge Tracts in Mathematics}.
\newblock Cambridge University Press, Cambridge, 1998.
\newblock With the collaboration of Herb (Charles) Clemens and Alessio Corti,
  Translated from the 1998 Japanese original.

\bibitem[Lie07]{Lie07}
Max Lieblich.
\newblock Moduli of twisted sheaves.
\newblock {\em Duke Math. J.}, 138(1):23--118, 2007.
\newblock \href{http://arxiv.org/abs/math/0411337}{arXiv:math/0411337}.

\bibitem[Lo12]{Lo12}
Jason Lo.
\newblock On some moduli of complexes on {K}3 surfaces.
\newblock {\em Arxiv preprint}, 2012.
\newblock \href{http://arxiv.org/abs/1203.1558}{arXiv:1203.1558v1}.

\bibitem[LP14]{LP14}
Christian Lehn and Gianluca Pacienza.
\newblock On the log minimal model program for irreducible symplectic
  varieties.
\newblock {\em Arxiv preprint}, 2014.
\newblock \href{http://arxiv.org/abs/1405.5649}{arXiv:1405.5649}.

\bibitem[LQ14]{LQ11}
Jason Lo and Zhenbo Qin.
\newblock Mini-walls for {B}ridgeland stability conditions on the derived
  category of sheaves over surfaces.
\newblock {\em Asian J. Math.}, 18(2):321--344, 2014.
\newblock \href{http://arxiv.org/abs/1103.4352}{arXiv:1103.4352v1}.

\bibitem[LS06]{LS06}
Manfred Lehn and Christoph Sorger.
\newblock La singularit{\'e} de {O}'{G}rady.
\newblock {\em J. Algebraic Geom.}, 15(4):753--770, 2006.
\newblock \href{http://arxiv.org/abs/math/0504182}{arXiv:math/0504182}.

\bibitem[Mac14]{Mac12}
Antony Maciocia.
\newblock Computing the walls associated to {B}ridgeland stability conditions
  on projective surfaces.
\newblock {\em Asian J. Math.}, 18(2):263--279, 2014.
\newblock \href{http://arxiv.org/abs/1202.4587}{arXiv:1202.4587v2}.

\bibitem[Mar11]{Mar11}
Eyal Markman.
\newblock A survey of {T}orelli and monodromy results for
  holomorphic-symplectic varieties.
\newblock In {\em Complex and differential geometry}, volume~8 of {\em Springer
  Proc. Math.}, pages 257--322. Springer, Heidelberg, 2011.
\newblock \href{http://arxiv.org/abs/1101.4606}{arXiv:1101.4606}.

\bibitem[Mar13]{Mar13}
Eyal Markman.
\newblock Prime exceptional divisors on holomorphic symplectic varieties and
  monodromy reflections.
\newblock {\em Kyoto J. Math.}, 53(2):345--403, 2013.
\newblock \href{http://arxiv.org/abs/0912.4981}{arXiv:0912.4981}.

\bibitem[Mea12]{Mea12}
Ciaran Meachan.
\newblock {\em Moduli of Bridgeland-stable objects}.
\newblock PhD thesis, University of Edinburgh, 2012.
\newblock \textit{Available online at
  }\url{http://www.maths.ed.ac.uk/~cmeachan/Thesis.pdf}.

\bibitem[MM13]{MM13}
Antony Maciocia and Ciaran Meachan.
\newblock Rank 1 {B}ridgeland stable moduli spaces on a principally polarized
  abelian surface.
\newblock {\em Int. Math. Res. Not. IMRN}, 107(9):2054--2077, 2013.
\newblock \href{http://arxiv.org/abs/1107.5304}{arXiv:1107.5304}.

\bibitem[Muk84]{Muk84}
Shigeru Mukai.
\newblock Symplectic structure of the moduli space of sheaves on an abelian or
  {$K3$} surface.
\newblock {\em Invent. Math.}, 77(1):101--116, 1984.

\bibitem[Muk87]{Muk87}
Shigeru Mukai.
\newblock On the moduli space of bundles on {K}3 surfaces. {I}.
\newblock In {\em Vector bundles on algebraic varieties ({B}ombay, 1984)},
  volume~11 of {\em Tata Inst. Fund. Res. Stud. Math.}, pages 341--413. Tata
  Inst. Fund. Res., Bombay, 1987.
\newblock \textit{Available online at
  }\url{http://www.kurims.kyoto-u.ac.jp/~mukai/paper/Tata.pdf}.

\bibitem[MYY11]{MYY11a}
Hiroki Minamide, Shintarou Yanagida, and K{\=o}ta Yoshioka.
\newblock Fourier-{M}ukai transforms and the wall-crossing behavior for
  {B}ridgeland's stability conditions.
\newblock {\em Arxiv preprint}, 2011.
\newblock \href{http://arxiv.org/abs/1106.5217}{arXiv:1106.5217v2}.

\bibitem[MYY14]{MYY11b}
Hiroki Minamide, Shintarou Yanagida, and K{\=o}ta Yoshioka.
\newblock Some moduli spaces of {B}ridgeland's stability conditions.
\newblock {\em Int. Math. Res. Not. IMRN}, (19):5264--5327, 2014.
\newblock \href{http://arxiv.org/abs/1111.6187}{arXiv:1111.6187}.

\bibitem[Nam01]{Nam01}
Yoshinori Namikawa.
\newblock Deformation theory of singular symplectic {$n$}-folds.
\newblock {\em Math. Ann.}, 319(3):597--623, 2001.
\newblock \href{http://arxiv.org/abs/math/0010113}{arXiv:math/0010113}.

\bibitem[O'G99]{OG99a}
Kieran O'Grady.
\newblock Desingularized moduli spaces of sheaves on a {$K3$}.
\newblock {\em J. Reine Angew. Math.}, 512:49--117, 1999.
\newblock \href{http://arxiv.org/abs/alg-geom/9708009}{arXiv:alg-geom/9708009}.

\bibitem[O'G03]{OG03}
Kieran O'Grady.
\newblock A new six-dimensional irreducible symplectic variety.
\newblock {\em J. Algebraic Geom.}, 12(3):435--505, 2003.
\newblock \href{http://arxiv.org/abs/math/0010187}{arXiv:math/0010187}.

\bibitem[PR13]{PeRa10a}
Arvid Perego and Antonio Rapagnetta.
\newblock Deformation of the {O}'{G}rady moduli spaces.
\newblock {\em J. Reine Angew. Math.}, 678:1--34, 2013.
\newblock \href{http://arxiv.org/abs/1008.0190}{arXiv:1008.0190}.

\bibitem[Sim94]{Sim94}
Carlos~T. Simpson.
\newblock Moduli of representations of the fundamental group of a smooth
  projective variety. {I}.
\newblock {\em Inst. Hautes \'Etudes Sci. Publ. Math.}, (79):47--129, 1994.

\bibitem[Tod08]{Toda08}
Yukinobu Toda.
\newblock Moduli stacks and invariants of semistable objects on {$K3$}
  surfaces.
\newblock {\em Adv. Math.}, 217(6):2736--2781, 2008.
\newblock \href{http://arxiv.org/abs/math/0703590}{arXiv:math/0703590}.

\bibitem[Wie03]{Wie03}
Jan Wierzba.
\newblock Contractions of symplectic varieties.
\newblock {\em J. Algebraic Geom.}, 12(3):507--534, 2003.
\newblock \href{http://arxiv.org/abs/math/9910130}{arXiv:math/9910130}.

\bibitem[Yos01]{Yos01}
K{\=o}ta Yoshioka.
\newblock Moduli spaces of stable sheaves on abelian surfaces.
\newblock {\em Math. Ann.}, 321(4):817--884, 2001.
\newblock \href{http://arxiv.org/abs/math/0009001}{arXiv:math/0009001}.

\bibitem[Yos06]{Yos06}
K{\=o}ta Yoshioka.
\newblock Moduli spaces of twisted sheaves on a projective variety.
\newblock In {\em Moduli spaces and arithmetic geometry}, volume~45 of {\em
  Adv. Stud. Pure Math.}, pages 1--30. Math. Soc. Japan, Tokyo, 2006.
\newblock \href{http://arxiv.org/abs/math/0411538}{arXiv:math/0411538}.

\bibitem[Yos09]{Yos09}
K{\=o}ta Yoshioka.
\newblock Fourier-{M}ukai transform on abelian surfaces.
\newblock {\em Math. Ann.}, 345(3):493--524, 2009.
\newblock \href{http://arxiv.org/abs/math/0605190}{arXiv:math/0605190}.

\bibitem[Yos12]{Yos12}
K{\=o}ta Yoshioka.
\newblock Bridgeland's stability and the positive cone of the moduli spaces of
  stable objects on an abelian surface.
\newblock {\em Arxiv preprint}, 2012.
\newblock \href{http://arxiv.org/abs/1206.4838}{arXiv:1206.4838v1}.

\bibitem[YY14]{YY14}
Shintarou Yanagida and K{\=o}ta Yoshioka.
\newblock Bridgeland's stabilities on abelian surfaces.
\newblock {\em Math. Z.}, 276(1-2):571--610, 2014.
\newblock \href{http://arxiv.org/abs/1203.0884}{arXiv:1203.0884}.

\end{thebibliography}

\end{document}